\documentclass[aos,preprint]{imsart}

\RequirePackage{amsthm,amsmath,amsfonts,amssymb,amsbsy,latexsym,dsfont,tikz,bbm,verbatim,enumerate,enumitem,subcaption, multirow,threeparttable,circuitikz}
\RequirePackage[numbers]{natbib}
\RequirePackage[colorlinks,citecolor=blue,urlcolor=blue]{hyperref}
\RequirePackage{graphicx}
\graphicspath{{./images/}} 

\RequirePackage{xcolor}

\usetikzlibrary{decorations.pathreplacing}

\startlocaldefs
\theoremstyle{plain}
\newtheorem{theorem}{Theorem}
\newtheorem{lemma}[theorem]{Lemma}
\newtheorem{corollary}[theorem]{Corollary}

\theoremstyle{definition}
\newtheorem{remark}{Remark}
\newtheorem{definition}{Definition}
\newtheorem{assumption}{Assumption}
\newcommand{\bTheta}{\boldsymbol{\Theta}}
\newcommand{\btheta}{\boldsymbol{\theta}}
\newcommand{\balpha}{\boldsymbol{\alpha}}
\newcommand{\bbeta}{\boldsymbol{\beta}}
\newcommand{\bd}[1]{\boldsymbol{#1}}

\newcommand{\hbtheta}{\widehat{\boldsymbol{\theta}}}
\newcommand{\idf}{\mathbbm{1}}

\newcommand{\bX}{\boldsymbol{X}}

\newcommand{\bSigma}{\boldsymbol{\Sigma}}
\DeclareMathOperator*{\argmin}{argmin}
\DeclareMathOperator*{\argmax}{argmax}
\newcommand{\op}[1]{{\left\vert\kern-0.25ex\left\vert\kern-0.25ex\left\vert #1\right\vert\kern-0.25ex\right\vert\kern-0.25ex\right\vert}}
\newcommand{\mn}{\vert\kern-0.25ex \vert\kern-0.25ex \vert}
\newcommand{\lmin}[1]{\lambda_{\min}\left(#1\right)}
\newcommand{\lmax}[1]{\lambda_{\max}\left(#1\right)}
\newcommand{\mR}{\mathbb{R}}
\newcommand{\hS}{\hat{S}}
\newcommand{\ltheta}{\hbtheta^{\textup{Lasso}}}
\newcommand{\hdtheta}{\hat{\theta}^{\textup{hard}}}
\newcommand{\hdbtheta}{\hat{\boldsymbol{\theta}}^{\textup{hard}}}

\newcommand{\ptheta}{\widehat{\boldsymbol{\theta}}^{\textup{OPT}}}
\newcommand{\mS}{\mathcal{S}}
\newcommand{\mI}{\mathcal{I}}

\newcommand{\bY}{\boldsymbol{Y}}
\newcommand{\mN}{\mathcal{N}}
\newcommand{\tbtheta}{\widetilde{\boldsymbol{\theta}}}
\newcommand{\bA}{\boldsymbol{A}}
\newcommand{\mC}{\mathcal{C}}
\newcommand{\tgamma}{\widetilde{\gamma}}

\newcommand{\hdlambda}{\lambda^{\textup{OPT}}}
\newcommand{\tS}{\tilde{S}}

\newcommand{\bu}{\boldsymbol{u}}
\newcommand{\bv}{\boldsymbol{v}}
\newcommand{\tbu}{\tilde{\boldsymbol{u}}}
\newcommand{\hbu}{\hat{\boldsymbol{u}}}

\newcommand{\hbTheta}{\hat{\boldsymbol{\Theta}}}
\newcommand{\bZ}{\boldsymbol{Z}}
\newcommand{\bW}{\boldsymbol{W}}

\newcommand{\mE}{\mathcal{E}}
\newcommand{\pS}{\hat{S}^{\textup{OPT}}}
\newcommand{\hdS}{\hat{S}^{\textup{OPT}}}
\newcommand{\hbSigma}{\hat{\boldsymbol{\Sigma}}}
\newcommand{\bepsilon}{\bd{\epsilon}}

\newcommand{\bx}{\boldsymbol{x}}
\newcommand{\bR}{\mathbb{R}}
\newcommand{\bN}{\mathbb{N}}
\newcommand{\cA}{\mathcal{A}}
\newcommand{\cF}{\mathcal{F}}

\newcommand{\Exp}{\mathbb{E}}
\newcommand{\Pro}{\mathbb{P}}
\newcommand{\tL}{\widetilde{\mathcal{L}}}

\endlocaldefs

\begin{document}

\begin{frontmatter}
\title{Minimax Rate-Optimal Algorithms for High-Dimensional Stochastic Linear Bandits}
\runtitle{}

\begin{aug}
\author{\fnms{Jingyu}~\snm{Liu}\ead[label=e1]{jingyu.liu@queensu.ca}}
\and
\author{\fnms{Yanglei}~\snm{Song}\ead[label=e2]{yanglei.song@queensu.ca}}
\address{Department of Mathematics and Statistics, Queen's University }

\address{\printead[presep={\ }]{e1,e2}}
\end{aug}

\begin{abstract}
We study the stochastic linear bandit problem with multiple arms over $T$ rounds, where the covariate dimension $d$ may exceed $T$, but each arm-specific parameter vector is $s$-sparse. We begin by analyzing the sequential estimation problem in the single-arm setting, focusing on cumulative mean-squared error. We show that Lasso estimators are provably suboptimal in the sequential setting, exhibiting suboptimal dependence on $d$ and $T$, whereas thresholded Lasso estimators---obtained by applying least squares to the support selected by thresholding an initial Lasso estimator---achieve the minimax rate. Building on these insights, we consider the full linear contextual bandit problem and propose a three-stage arm selection algorithm that uses thresholded Lasso as the main estimation method. We derive an upper bound on the cumulative regret of order $s(\log s)(\log d + \log T)$, and establish a matching lower bound up to a $\log s$ factor, thereby characterizing the minimax regret rate up to a logarithmic term in $s$. Moreover, when a short initial period is excluded from the regret, the proposed algorithm achieves exact minimax optimality.
\end{abstract}



\end{frontmatter}

\section{Introduction}
With the rapid growth of data from sources such as mobile applications    
and healthcare monitoring, streaming data has become an important focus in modern statistical analysis \citep{shalev2012online,lattimore2020bandit}. Streaming data refers to data arriving continuously over time, often in large volumes and at high velocity. In these settings, statistical procedures must operate in real time, updating models and making decisions as new observations arrive.
In this work, we focus on multi-armed bandit (MAB) problems,  a fundamental example of sequential decision-making, with wide applications including online advertising \citep{chapelle2011empirical}, personalized medicine \citep{tewari2017ads}, and recommendation systems \citep{Li2010}.

Classical MAB problems, introduced by \cite{thompson1933likelihood} and formalized by \cite{robbins1952some}, concern the sequential selection of arms with unknown mean rewards to maximize cumulative reward---or equivalently, to minimize cumulative regret relative to always choosing the best arm. Foundational algorithms such as upper confidence bound \citep{lai1985asymptotically, auer2002finite}, Thompson sampling \citep{agrawal2012analysis}, and information-directed sampling \citep{russo2018learning} achieve near-optimal regret guarantees in this setting \citep{audibert2009minimax, bubeck2012regret}. However, these models do not incorporate contextual information, or covariates, which can greatly enhance decision-making. The inclusion of context leads to diverse modeling frameworks: parametric \citep{li2010contextual} vs.\ nonparametric \citep{perchet2013multi}, linear \citep{abbasi2011improved, goldenshluger2013linear, bastani2020online, bastani2021mostly} vs.\ nonlinear \citep{jun2017scalable, ding2021efficient}, and stochastic \citep{goldenshluger2013linear, bastani2020online, bastani2021mostly} vs.\ adversarial \citep{auer2002nonstochastic, kakade2008efficient} environments;  see \cite{lattimore2020bandit} for a comprehensive overview. Given the breadth of the literature, we begin by precisely defining our setting and focusing on the most relevant prior work.  

Suppose there are $K$ arms, each associated with an unknown $d$-dimensional parameter vector $\btheta^{(k)} \in \bR^{d}$ for $k \in [K] := \{1,\ldots, K\}$. We consider the following sequential decision-making problem over $T$ rounds. At each round $t \in [T] := \{1,\ldots,T\}$, a covariate vector $\bX_t \in \bR^{d}$ is observed, an arm $A_t \in [K]$ is selected, and a reward $Y_t \in \bR$ is received. The arm selection $A_t$ is allowed to depend only on $\bX_t$ and the past history up to time $t-1$, that is, $\{(\bX_s, Y_s, A_s) : s < t\}$. We assume that the covariate vectors $\{\bX_t: t \in [T]\}$ are independent and identically distributed (i.i.d.) random vectors in $\bR^d$, and that the expected reward is linear in $\bX_t$ and the parameter of the selected arm:
\[
Y_t = \bX_t'\btheta^{(A_t)} + \epsilon_t,
\]
where $\{\epsilon_t: t \in [T]\}$ are i.i.d. zero-mean random variables, independent of $\{\bX_t: t \in [T]\}$. Given an arm selection rule with arm choices $\{A_t: t \in [T]\}$, we evaluate its performance using the expected cumulative regret. Specifically, for each $t \in [T]$, define the instantaneous regret:
$r_t := \max_{k \in [K]} \bX_t'\btheta^{(k)} - \bX_t'\btheta^{(A_t)}$, 
which compares the chosen arm to the best possible arm at time $t$. The expected cumulative regret is then defined as $\sum_{t=1}^T \mathbb{E}[r_t]$. The objective is to find an arm selection rule that minimizes this quantity. 

When the covariate dimension $d$ is much smaller than the time horizon $T$, we refer to this as the low-dimensional regime. In this setting, the arm parameters are typically estimated using ordinary least squares (OLS) or ridge regression. \cite{goldenshluger2013linear} proposes a \emph{forced sampling} strategy based on OLS for the two-arm case ($K = 2$), and establishes an upper bound on the cumulative regret of order $O(d^3 \log T)$ under a \emph{margin condition}. This condition requires that the probability of $\bX_t$ falling within a distance $\tau$ of the decision boundary $\{ \bx \in \mathbb{R}^d : (\btheta^{(1)})'\bx = (\btheta^{(2)})'\bx \}$ is at most $C \tau$ for any $\tau > 0$. This bound is improved to $O(d^2 \log^{3/2} d \log T)$ in \cite{bastani2020online}. The work \cite{goldenshluger2013linear} also establishes a lower bound of order $\Omega(\log T)$, which characterizes the optimal dependence on $T$, though the optimal dependence on $d$ remains unresolved. An alternative approach is proposed in \cite{abbasi2011improved}, where the LinUCB algorithm---based on ridge regression and the principle of upper confidence bounds---is shown to attain a regret upper bound of $O(\log^2 T)$ \citep{DBLP:journals/corr/abs-2002-05152}. However, \cite{song2022truncated} demonstrates that LinUCB is suboptimal, and proposes a variant in which the algorithm is truncated at a prescribed time and followed by pure exploitation. The resulting truncated LinUCB algorithm achieves a regret bound of $O(d \log T)$, which is shown to be optimal in both $d$ and $T$.

In this work, we focus on the \emph{high-dimensional} regime, where the covariate dimension $d$ may substantially exceed the time horizon $T$. Such settings arise in applications where rich contextual features are available but sample sizes are limited due to budget or time constraints \cite{bastani2020online,lattimore2020bandit}. To ensure statistical tractability, we assume that for each arm $k \in [K]$,  $\btheta^{(k)}$ is $s_0$-sparse, meaning that it has at most $s_0$ nonzero components. This reflects the standard assumption that only a small subset of covariates is relevant for reward prediction, and is widely used in the high-dimensional bandit literature \cite{bastani2020online,hao2020high,oh2021sparsity,ren2024dynamic}. Our goal is to characterize the minimax rate of cumulative regret and to develop algorithms that approach this rate.
 
For any bandit algorithm, a fundamental component is the estimation of the arm parameters. Based on these estimates---and possibly on a quantification of their uncertainty---the algorithm selects an arm to pull. In order to understand the statistical limits of estimation methods in sequential, high-dimensional settings,
we first isolate the \textit{sequential estimation} problem, where a \textit{single} sparse parameter vector is estimated sequentially without adaptive arm selection.   This problem is of independent interest and provides insight into the design of effective arm-selection strategies in the full \textit{bandit problem}.

\subsection{Sequential High-Dimensional Linear Regression}\label{sec: sequential estimation}
Consider the setting with a single arm, i.e., $K = 1$, and denote $\btheta^{(1)}$ by $\btheta$. At each round $t \in [T]$, we construct an estimator $\hat{\btheta}_t$ based on the data $\{(\bX_s, Y_s): s \in [t]\}$ observed up to time $t$. We evaluate the sequence of estimators $\{\hat{\btheta}_t: t \in [T]\}$ through the cumulative capped mean-square error:
$\sum_{t=1}^{T} \mathbb{E}[\min\{ \| \hat{\btheta}_t - \btheta \|_2^2, \; \xi \} ]$,
where $\xi > 0$ is a fixed constant;  see Section~\ref{sec: multi linear}.

Since no arm selection is involved, at each round $t \in [T]$, the data $\{(\bX_s, Y_s): s \in [t]\}$ are i.i.d., allowing us to leverage the extensive literature on \textit{fixed-sample} high-dimensional linear regression. In particular, the Lasso estimator \cite{tibshirani1996regression}, which minimizes the $\ell_1$-penalized least squares criterion, is among the most widely used methods in this setting. For fixed $t = \Omega(s_0 \log d)$, and under restricted eigenvalue-type conditions, it is well known \citep{van2008high, van2016estimation, wainwright2019high} that the Lasso estimator $\hat{\btheta}_t^{\textup{Lasso}}$ achieves a mean-squared error of order $O(s_0 \log d / t)$, which matches the minimax rate. Hence, no estimator can uniformly outperform Lasso over the class of all $s_0$-sparse parameter vectors $\btheta$ for a \textit{fixed} $t$. 
Despite its minimax optimality, many variants of Lasso have been proposed---such as the elastic net \cite{zou2005regularization}, group Lasso \cite{yuan2006model}, adaptive Lasso \cite{zou2006adaptive}, and thresholded Lasso  \cite{van2011adaptive, belloni2013least}---with the goal of improving instance-specific estimation accuracy or variable selection performance.

Now, in the sequential estimation problem, we may apply the Lasso estimator $\hat{\btheta}_t^{\textup{Lasso}}$ at each round $t \in [T]$. Summing the upper bound $s_0 \log d  / t$ over $t \in [T]$  yields a cumulative error bound of order $O(s_0 \log(d) \log(T))$ for the sequence of Lasso estimators. However, it remains unclear whether this bound is tight for Lasso, or whether it coincides with the minimax rate. Importantly, the fact that Lasso is minimax optimal at each fixed time $t \in [T]$ does not imply minimax optimality in terms of \textit{cumulative} error. The key distinction is that in the sequential setting, the underlying parameter vector $\btheta$ is shared across all $T$ rounds, whereas in the \textit{fixed-sample} minimax theory, the worst-case $\btheta$ may vary across rounds.

\smallskip
\noindent \textbf{Our contributions.} First, we establish in Theorem \ref{thm: lower bound of lasso} that the worst-case cumulative error of Lasso estimators is of order $\Omega(s_0 \log(d) \log(T))$, indicating that the previously derived upper bound is tight.
Second, we consider the OPT-Lasso estimators\footnote{This is another term for ``thresholded Lasso''. We adopt this terminology from \cite{belloni2013least}.} \cite{van2011adaptive, belloni2013least}, which first threshold an initial Lasso estimator and then apply ordinary least squares to the selected support. In Theorem \ref{thm: seq multi linear, upper bound}, we provide an instance-specific upper bound on the mean-squared error of OPT-Lasso for a fixed $t \in [T]$, which explicitly characterizes the impact of the nonzero components of $\btheta$, and leads to a cumulative error bound of order $s_0(\log d + \log T)$. A key technical tool is the recent $\ell_{\infty}$ bound for Lasso estimators established in \cite{bellec2022biasing}.
Third, we establish a matching lower bound that holds uniformly over all permissible estimators. Specifically, we lower bound the minimax rate via Bayesian risks by constructing two families of prior distributions, and show in Theorem~\ref{thm: seq multi linear, lower bound} that any permissible sequence of estimators must incur a worst-case cumulative error of order at least $s_0(\log d + \log T)$.

These results characterize the minimax rate and demonstrate that Lasso estimators are provably suboptimal in the \textit{sequential} setting, while OPT-Lasso estimators are minimax optimal. The sequential estimation problem is of independent interest, and our analysis further indicates that OPT-Lasso is broadly preferable to Lasso in sequential tasks---for example, in the bandit setting.

\subsection{High-dimensional Linear Contextual Bandits}
Next, we focus on the bandit setting with at least two arms, i.e., $K \geq 2$, under the high-dimensional linear contextual framework. In this setting, \cite{bastani2020online} extends the ``forced sampling'' strategy of \cite{goldenshluger2013linear} by replacing ridge regression with Lasso estimators, and establishes an upper bound on the expected cumulative regret of order $O(s_0^2(\log d + \log T)^2)$. Subsequently, \cite{wang2018minimax} proposes an algorithm based on a two-step weighted Lasso estimator, achieving an improved upper bound of order $O(s_0^2(s_0 + \log d)\log T)$. More recently, \cite{ariu2022thresholded} introduces a pure exploitation algorithm based on thresholded Lasso estimators. Under the assumptions that (i) $s_0$ is a constant, and (ii) the nonzero components of $\{\btheta^{(k)}: k \in [K]\}$ are bounded away from zero by a \textit{constant},
\cite{ariu2022thresholded} derives an upper bound on the cumulative regret of order $O((\log\log d)  \log d  +  \log T)$. 
It is important to note that all of the aforementioned results rely on the margin condition, which enables regret bounds that are logarithmic in the time horizon $T$. Without this condition, the cumulative regret grows polynomially in $T$; see, for instance, \cite{kim2019doubly, oh2021sparsity, ariu2022thresholded, ren2024dynamic}.  

As for lower bounds, the result in the low-dimensional regime \cite{song2022truncated} implies a trivial lower bound of order $\Omega(s_0 \log T)$, which corresponds to the case where the support of $\btheta^{(k)}$ is known for $k \in [K]$. However, it remains unclear whether this lower bound is tight in the high-dimensional setting or what the minimax regret rate should be in general.

\smallskip
\noindent \textbf{Our contributions.} First, building on insights from the sequential estimation problem, we propose a three-stage algorithm. In Stage 1, spanning from round $t = 1$ to $\gamma_1$, arms are selected uniformly at random from $[K]$. In Stage 2, from round $\gamma_1 + 1$ to $\gamma_2$, we estimate the arm parameters using the \textit{Lasso} method and select the arm that maximizes the \textit{estimated} reward based on these estimates. In Stage 3, from round $\gamma_2 + 1$ to $T$, we replace the Lasso estimators with \textit{OPT-Lasso} estimators and continue to perform pure exploitation as in Stage 2. 
To improve computational efficiency, the estimators in Stages 2 and 3 are updated periodically rather than at every round. As discussed in Section~\ref{sec: problem formulation}, we choose $\gamma_1$ to be of order $s_0 \log(dT)$ and $\gamma_2$ to be of order $s_0^5 \log(dT)$, so that Stage 3 dominates the time horizon and the algorithm predominantly relies on OPT-Lasso estimators. The motivation for introducing Stage 2 is discussed in Section~\ref{sec: problem formulation}.

Second, in Theorem \ref{thm: regret over horizon}, we establish an upper bound of order $O(s_0 (\log s_0) (\log d + \log T))$ on the overall cumulative regret of the proposed three-stage algorithm. Furthermore, if the cumulative regret is computed starting from round $C s_0^5 \log(dT)$, for a sufficiently large constant $C$, the upper bound improves to $O\left(s_0 (\log d + \log T)\right)$, removing the $\log s_0$ factor. We emphasize that our analysis allows $s_0$ to grow with $d$ and $T$, and does not require any ``beta-min'' type condition as in \cite{ariu2022thresholded}. See more discussions in Subsection \ref{subsec:bandit_minimax}.

Third, in Theorem~\ref{thm: regret lower bound}, we establish a lower bound of order $s_0 (\log d + \log T)$ for the cumulative regret starting from round $C s_0^5 \log(dT)$, uniformly over all permissible rules and for any constant $C > 0$. This result automatically implies a lower bound of the same order for the \textit{overall} cumulative regret. As a consequence, we characterize the minimax rate of the overall cumulative regret, up to a $\log s_0$ factor, and show that the proposed algorithm is nearly minimax optimal. Moreover, if the initial period of length $O(s_0^5 \log T)$ is ignored, the minimax rate is $s_0 (\log d + \log T)$, which is attained by the proposed algorithm.

Finally, we show in Theorem \ref{thm: bandit lower bound of lasso} that using Lasso instead of OPT-Lasso leads to provably suboptimal performance, which demonstrates the necessity of using thresholded estimators in this setting. 
In addition, simulation results further demonstrate the favorable performance of the proposed three-stage algorithm compared to alternative approaches.

\subsection{Paper Organization and Notations}
In Section \ref{sec: multi linear}, we study the sequential high-dimensional linear regression problem. The Lasso and OPT-Lasso estimators are formally introduced in Subsection \ref{subsec:lasso_and_opt_lasso}. In Subsection \ref{subsec:minimax_rate}, we characterize the minimax rate and show that OPT-Lasso achieves minimax optimality. The suboptimality of Lasso is analyzed in Subsection \ref{sec: suboptimality of lasso}, with supporting simulation results presented in Subsection \ref{sec: HSLR simulation}.

Section \ref{sec: problem formulation} turns to the bandit setting. A three-stage algorithm is introduced in Subsection \ref{subsec:three_stage_algo}, and its cumulative regret is analyzed in Subsection \ref{subsec:bandit_minimax}, where we also establish nearly matching lower bounds. The suboptimality of Lasso in this setting is analyzed in Subsection \ref{subsec:lasso_bandit}. Corresponding simulation results are presented in Subsection \ref{sec: bandit simulation}. The proof of the regret bounds is provided in Subsection \ref{sec: proof of thm: regret over horizon}. We conclude in Section \ref{sec:conclusion}, and include additional proofs and simulation results in the Appendix.

\smallskip
\noindent \textbf{Notations.} The abbreviation ``i.i.d.'' stands for ``independent and identically distributed''. For each $n\in \mathbb{N} := \{1,2,\ldots\}$, we denote by $[n]=\{1,\cdots, n\}$, and by $\bd{0}_n$ the zero vector in $\bR^{n}$. For $a,b \in \bR$, $a \vee b= \max\{a,b\}$ and $a\wedge b = \min\{a,b\}$.  $\idf\{R\}$ denotes the indicator function of event $R$; that is,  $\idf\{R\}$ is $1$ (resp.~$0$) if $R$ is true (resp.~false). 
Let $\{a_T; T \in \mathbb{N}\}$ be a sequence of real numbers. We denote $a_T=O(T)$ (resp.~$a_T=\Omega(T)$) if there exists a constant $C>0$, that does not depend on $T$, such that $a_T \le C T$ (resp.~$a_T \ge C^{-1} T$). 

Unless otherwise specified, all vectors are assumed to be column vectors. For a vector $\bv = (v_1, \ldots, v_d)' \in \mathbb{R}^d$, the $\ell_p$ norm is defined as $\|\bv\|_p = ( \sum_{i \in [d]} |v_i|^p )^{1/p}$ for $1 \le p < \infty$. The $\ell_0$ “norm” $\|\bv\|_{0} =|\{j \in [d]:  \bv_j \neq 0\}|$ denotes the number of nonzero components in $\bv$, and the $\ell_\infty$ norm is given by $\|\bv\|_\infty = \max_{j \in [d]} |v_j|$. 
For a matrix $\bA\in \mR^{m\times n}$,  define the matrix norms $\op{\bA}_1 = \max_{1\le j\le n}\sum_{i=1}^m |\bA_{i,j}|$, $\op{\bA}_{\infty} = \max_{1\le i\le m}\sum_{j=1}^n |\bA_{i,j}|$, and $\op{\bA}_{\textup{max}} = \max_{1\leq i \leq m, 1\leq j \leq n}|\bA_{i,j}|$. If $\bA$ is symmetric  and positive semidefinite,  denote by $\lmin{\bA}$ and $\lmax{\bA}$ the smallest and largest eigenvalue of $\bA$ respectively, and
by $\bA^{1/2}$ a symmetric matrix such that $\bA^{1/2} \bA^{1/2} = \bA$. Given subsets $S_1 \subseteq [m]$ and $S_2 \subseteq [n]$, denote by $\bA_{S_1, S_2} \in \bR^{|S_1| \times |S_2|}$ the submatrix of $\bA$ consisting of the rows indexed by $S_1$ and the columns indexed by $S_2$. We abbreviate $\bA_{[m],S_2}$ as $\bA_{S_2}$. For $\bA\in \mR^{m\times n}$,  denote its transpose by $\bA'$.

\section{Sequential High-Dimensional Linear Regression} \label{sec: multi linear}
In this section, we focus on the \textit{sequential} high-dimensional linear regression problem. Specifically, let $T$ be the total number of rounds, which is also called the horizon. At each round (or time) $t\in [T]$, a pair of observations $(Y_t, \bX_t)\in\mR\times\mR^d$ is collected, which is assumed to follow a linear regression model with an unknown parameter $\btheta\in\mR^d$ and an observation noise  $\epsilon_t\in\mR$, that is,
\begin{equation*}
Y_t=\bX_t'\btheta + \epsilon_t.
\end{equation*}
Assume that $\{\bX_t \in \mathbb{R}^d : t \in [T]\}$ are i.i.d. random vectors drawn from the multivariate normal distribution $\mathcal{N}(\mathbf{0}_d, \bSigma)$, where $\mathbf{0}_d \in \mathbb{R}^{d}$ is the mean vector, and $\bSigma \in \mathbb{R}^{d \times d}$ is the covariance matrix.
Further, assume that $\{\epsilon_t \in \mathbb{R} : t \in [T]\}$ are i.i.d. random variables drawn from the normal distribution $\mathcal{N}(0, \sigma^2)$.

Let $S$ denote the support of $\btheta = (\btheta_1, \ldots, \btheta_d)'$, i.e., $S = \{j \in [d] : \btheta_j \neq 0\}$. We define $s_0 = |S| \vee 1$, so that $s_0 = 1$ when $S$ is empty. Denote by $\bTheta_d[s]:=\{\btheta\in \mR^d: \|\btheta\|_0\leq s\}$ the set of all $s$-sparse vectors. Thus, $\btheta \in \bTheta_d[s_0]$. We focus on the high-dimensional regime, where the covariate dimension $d$  is potentially much larger than the horizon $T$, i.e. $d\gg T$, while only a few covariates capture the overall impact on the responses,  i.e., $s_0 \ll d$. 


For each time $t\in [T]$, we denote by $\hbtheta_t$ an estimator of the unknown parameter $\btheta$ based on the currently available observations $\bX_{[t]} =(\bX_1,\cdots,\bX_t)'$ and $\bY_{[t]} = (Y_1,\cdots,Y_t)'$. Our goal is to find a sequence of estimators $\{\hbtheta_t:t\in [T]\}$ to minimize the cumulative capped mean-square error 
$$
\sum_{t=1}^T \Exp \left[\|\hbtheta_t-\btheta \|_2^2 \wedge \xi\right],
$$
where $\xi>0$ is an arbitrary user-specified parameter. To make the dependence on $\btheta$ explicit, we use the notation   $\Exp_{\btheta}$ in place of $\Exp$ when needed.

\begin{remark}
In high-dimensional regression literature, estimation error bounds are typically stated with high probability: for each $t \in [T]$, there exists a high-probability event on which $\|\hbtheta_t - \btheta\|_2$ is controlled. Introducing a cap parameter allows these bounds to be converted into expected error bounds.
\end{remark}

In the sequential estimation problem considered in this section, no arm selection is involved. As a result, for each round $t \in [T]$, the observations $\{(X_s, Y_s): s \in [t]\}$ are i.i.d. and can be treated as a fixed-sample problem. However, the problems across rounds are \textit{coupled} due to the shared parameter $\btheta$ and the cumulative nature of the performance criterion.

\subsection{Lasso and OPT-Lasso: Fixed-Sample Discussion} \label{subsec:lasso_and_opt_lasso}

In this subsection, we discuss two well-known procedures in the \textit{fixed-sample} high-dimensional linear regression literature. We start with the Lasso estimator \citep{tibshirani1996regression}, which applies $\ell_1$-regularization and is defined as follows.

\begin{definition}\label{def: lasso}
Let $n \in \bN$. 
Given a matrix $\bX\in\mR^{n\times d}$, a vector $\bY\in \mR^n$ and a regularization parameter $\lambda > 0$, the Lasso estimator $\ltheta_n(\bX, \bY, \lambda)$ 
is defined as 
$$\ltheta_n(\bX, \bY, \lambda) :=\argmin_{\bd{\beta}\in \mR^d} \left\{\frac{1}{2n}\|\bY-\bX\bd{\beta} \|_2^2+\lambda\|\bd{\beta}\|_1 \right\}.$$ 
\end{definition}

The Lasso estimator is known to be minimax optimal for a fixed sample size. Specifically, see Example 7.14 in \cite{wainwright2019high} for a high probability upper bound on the $\ell_2$ estimation error, and Theorem 1(b) in \cite{raskutti2011minimax} for a matching minimax lower bound. Therefore, it may be surprising that, as we show in section \ref{sec: suboptimality of lasso}, the Lasso estimator becomes \textit{suboptimal} in the sequential setting in terms of the worst-case cumulative error.   We note that this is not contradictory. For a general sequence of estimators $\{\hbtheta_t: t \in [T]\}$, we have
$$
\sup_{\btheta\in \bTheta_d[s_0]}\sum_{t=1}^T \Exp_{\btheta} \left[\|\hbtheta_t-\btheta\|_2^2 \wedge \xi \right]\le \sum_{t=1}^T \sup_{\btheta\in \bTheta_d[s_0]} \Exp_{\btheta} \left[\|\hbtheta_t-\btheta\|_2^2 \wedge \xi \right].
$$
The Lasso estimator achieves minimax optimality for each (sufficiently large) $t$, and is therefore optimal with respect to the criterion on the right-hand side. However, that does not imply its minimax optimality in the \textit{sequential} setting,  since the problems across rounds are coupled by a single $\btheta$, as previously discussed.


The suboptimality of Lasso in the sequential setup motivates us to consider another well-known estimator in the fixed-sample literature, which we refer to as OPT-Lasso; see, e.g., Section 7.6 of \cite{buhlmann2011statistics} and \cite{belloni2013least}. Note that ``OPT'' stands for ``OLS Post Thresholded''.
\begin{definition} \label{def: OPT-Lasso}
Let $n \in \bN$. Given  a matrix $\bX\in\mR^{n\times d}$, a vector $\bY\in\mR^n$, and two regularization parameters $\lambda, \hdlambda > 0$, the OPT-Lasso estimator $\ptheta_n := \ptheta_n(\bX, \bY, \lambda, \hdlambda)$ is defined as follows:
\begin{enumerate}
    \item Compute the Lasso estimator $\ltheta_n(\bX,\bY,\lambda)$, denoted by $\ltheta_n$; 
    \item Threshold each component of the Lasso estimator using $\hdlambda$: 
    \begin{align*}
        \pS_n := \left\{j \in [d]: |(\ltheta_n)_j| > \hdlambda \right\}.
    \end{align*}
    
    \item Apply the ordinary least squares (OLS) procedure using the covariates in $\hdS_n$, and set the remaining coordinates to zero:
    \begin{equation*}
        (\ptheta_n)_{\hdS_n}=\min_{\bd{\beta}\in \mR^{\left|\hdS_n\right|}} \left\|\bY-\bX_{\hdS_n}\bd{\beta}\right\|_2^2, \quad (\ptheta_n)_j=0 \text{ for } j\in(\hdS_n)^c.
    \end{equation*}
\end{enumerate}
\end{definition}

For a fixed round $t \in [T]$, as discussed in \cite{belloni2013least}, OPT-Lasso achieves the same minimax convergence rate of $O(s_0 \log d/t)$ as Lasso, while additionally providing improved instance-specific rates by reducing Lasso’s bias. Specifically, if the initial Lasso estimator has the variable screening property---that is, its support contains the true support $S$---then OPT-Lasso has a strictly faster rate of convergence. Moreover, if the support of the initial Lasso estimator is equal to $S$, then OPT-Lasso attains the oracle rate $O(s_0/t)$. 

The variable screening property of Lasso in general requires the so-called \textit{beta-min condition} (see (2.23) in \cite{buhlmann2011statistics}) as follows:
\begin{align}\label{def: beta-min}
\min_{j \in S} |\btheta_j| = \Omega(\sqrt{(s_0 \log d)/n}),  
\end{align}
which ensures that the smallest nonzero coefficient in the true support $S$ is sufficiently large in absolute value. If, with a high probability, the following incoherence condition holds:
\begin{equation}\label{def:incoherence}
\max_{i\neq j\in[d]}|n^{-1}\sum_{\ell=1}^n{\bX_{\ell, i}\bX_{\ell, j}}| = O(1/s_0),
\end{equation}
then  condition \eqref{def: beta-min} can be relaxed as $\min_{j \in S} |\btheta_j| = \Omega(\sqrt{\log d/n})$ (see Theorem 2 in \cite{belloni2013least}).

As we will see, in the sequential setup, the beta-min condition is particularly restrictive, and under this condition, it is unclear how to derive the minimax rate; see discussions following Theorem \ref{thm: seq multi linear, upper bound} and in Remark \ref{remark:beta_min_lower_bound_diff}. In this work, we do not impose ``beta-min'' type conditions. 

In this section, for each time $t \in [T]$, we denote by $\ptheta_t$ the OPT-Lasso estimator $\ptheta_t(\bX_{[t]}, \bY_{[t]}, \lambda_t, \hdlambda_t)$ and by $\ltheta_t$ the Lasso estimator $\ltheta_t(\bX_{[t]}, \bY_{[t]}, \lambda_t)$, where recall that $\bX_{[t]}$ and $\bY_{[t]}$ are the observations available up to time $t$, and $\lambda_t, \hdlambda_t$ are tuning parameters to be specified.

\subsection{Minimax Optimality of OPT-Lasso in  Sequential Estimation}\label{subsec:minimax_rate}
In this subsection, we show that in the sequential setup, OPT-Lasso achieves minimax optimality with respect to the cumulative capped mean-squared error. Recall that $S$ denotes the support of the unknown parameter $\btheta$.

\begin{assumption} \label{assumption: sequential estimation}
Let $L_0>0$ be a constant. 
\begin{enumerate}[label={(\alph*)}, ref={\theassumption\ (\alph*)}]
\item \label{assumption: bounded variance} For $j \in [d]$, $\bSigma_{j,j}\le 1$. 
\item \label{assumption: zhang}
  Assume that
$$\max_{A\subseteq [d]: |A\backslash S|\le 1\vee (L_0s_0+1)} \frac{\lmax{\bSigma_{A,A}}}{\lmin{\bSigma_{A,A}}} \le L_0/2, \ \min_{A\subseteq[d]: |A\backslash S|=L_0s_0+1} \lmin{\bSigma_{A,A}} \ge L_0^{-1}.$$
\item \label{assumption: bounded cov mat 1}
If $S$ is non-empty, assume that $L_0^{-1}\le \lmin{\bSigma_{S,S}}\le\lmax{\bSigma_{S,S}}\le L_0$.
\item \label{assumption: constraint on the covariance matrix}
$\bSigma$ is invertible, and $\op{\bSigma^{-1}}_{\infty} \le L_0$.
\end{enumerate}
\end{assumption}

Assumption \ref{assumption: bounded variance} requires that the variances of the covariates are uniformly bounded by a constant, which, without loss of generality, is assumed to be $1$. Assumption \ref{assumption: bounded cov mat 1} ensures that the population covariance matrix of the covariates, when restricted to the support $S$, is non-singular and well-conditioned. Assumption \ref{assumption: zhang} is known as the sparse Riesz condition (see \cite{zhang2008sparsity,bellec2022biasing}), which controls the range of the eigenvalues of the covariance matrices restricted to subsets containing a bounded number of covariates. This condition ensures that Lasso estimators can recover a model with the correct order of sparsity. Assumption \ref{assumption: constraint on the covariance matrix} requires that the $\ell_1$ norm of each column of $\bSigma^{-1}$ is bounded. This type of condition commonly appears in the literature on $\ell_{\infty}$ estimation error bounds for Lasso; see, for example, Theorem 5.1 in \cite{bellec2022biasing}, and Lemma 4.1 in \cite{van2016estimation}.

We note that if $C^{-1} \le \lmin{\bSigma} \le \lmax{\bSigma} \le C$ for some constant $C > 0$, then all conditions in Assumption \ref{assumption: sequential estimation} hold after a suitable normalization and with an adjusted constant $L_0$, except for the second part of Assumption \ref{assumption: constraint on the covariance matrix}. Examples that satisfy Assumption \ref{assumption: constraint on the covariance matrix} include: block diagonal matrices where both the size of blocks and the entries of $\bSigma^{-1}$ are bounded by a constant, or circulant matrices where $\bSigma_{i,j} = r^{|i-j|}$ for $i,j \in [d]$ and some $r\in (0,1)$ (see Problem 2.4 in \cite{buhlmann2011statistics}).

In the following theorem, we provide an instance-specific upper bound on the capped mean-square estimation error of the OPT-Lasso estimator for each round, which leads to an upper bound on the expected cumulative error across all rounds.

\begin{theorem}\label{thm: seq multi linear, upper bound}
Suppose that Assumption \ref{assumption: sequential estimation} holds, and that $d \ge (2L_0+1)s_0 + 2$. There exist constants $\kappa_1, C_0, C_0^{\textup{hard}},C>0$ depending only on $L_0$,  such that if we set the regularization and threshold parameters as follows
\begin{equation}\label{eq: online regularization threshold parameter}
    \lambda_t = C_0\sigma\sqrt{\frac{\log (dt)}{t}},\quad \hdlambda_t =  C_0^{\textup{hard}}\lambda_t,
\end{equation} 
then for all $t\ge \kappa_1 s_0\log d$, we have 
\begin{equation}\label{eq: opt_lasso_t}
    \Exp\left[\|\ptheta_t-\btheta \|_2^2 \wedge \xi\right] \le C\sum_{j\in S} \btheta_j^2 \idf\left\{|\btheta_j| \le 2\hdlambda_t \right\} 
    + C\sigma^2 s_0/t + C \xi (de^{-t/C}+1/t).
\end{equation}
Consequently, we have
$$
\sum_{t=1}^T \Exp\left[\|\ptheta_t-\btheta \|_2^2 \wedge \xi\right] \le C((\sigma^2 \vee 1) +\xi) s_0(\log d + \log T).
$$
\end{theorem}
\begin{proof} See Appendix \ref{sec: proof of seq multi linear, upper bound}.
\end{proof}

First, we focus on the \textit{sequential} setup and explain how the upper bound on the \textit{cumulative} error is derived from the result in \eqref{eq: opt_lasso_t}. For $1 \leq t \leq \kappa_1 s_0\log d$, we simply bound $\Exp\left[\|\ptheta_t-\btheta \|_2^2 \wedge \xi\right]$ by $\xi$. For $\kappa_1 s_0\log d < t \leq T$, we apply the upper bound given in \eqref{eq: opt_lasso_t}, and then control the expected cumulative error using the following lemma.

\begin{lemma}\label{lemma: instance sum}
Let $a  \neq 0$ and $b > 0$. Then
$\sum_{t=1}^T a^2 \idf\{|a|\le \sqrt{b/t}\} \le b$. Consequently, 
$\sum_{t=1}^T \sum_{j\in S} \btheta_j^2\idf\{|\btheta_j|\le \sqrt{b/t}\} \le b s_0$.
\end{lemma}
\begin{proof} Note that
$\sum_{t=1}^T a^2\idf\{|a|\le \sqrt{b/t}\} = \sum_{t=1}^T a^2\idf\{t\le b/a^2\} \le a^2 (b/a^2) = b.$
\end{proof}
Specifically, due to \eqref{eq: online regularization threshold parameter}, we have $\hdlambda_t \leq \sqrt{C \sigma^2 \log(dT)/t}$. Thus,  by applying Lemma \ref{lemma: instance sum} with $b = 4C\sigma^2\log(dT)$, we see that for each $j \in S$, the contribution of the $j$-th covariate to the expected cumulative regret is at most $4C\sigma^2\log(dT)$, \textit{regardless of the magnitude of $\btheta_j$}. Therefore, we avoid imposing any ``beta-min" type conditions.

We note that the bound in \eqref{eq: opt_lasso_t} for a \textit{fixed} round $t \geq \kappa_1 s_0\log d$ is \textit{instance-specific}, as it depends on the parameter $\btheta$. To see its importance, consider the following pessimistic upper bound: for $j \in S$,
\begin{equation}\label{pess_bound}
    \btheta_j^2 \idf\left\{|\btheta_j| \le 2\hdlambda_t \right\} \leq 4 (\hdlambda_t)^2 = C \log(dt)/t.
\end{equation}
Since $\sum_{t=1}^{T}\log(dt)/t = O(\log(T)\log(dT))$, this non-instance-specific bound would introduce an additional $\log T$ factor, leading to suboptimal results.

As mentioned in Subsection \ref{subsec:lasso_and_opt_lasso}, under the ``beta-min'' type conditions \eqref{def: beta-min} or \eqref{def:incoherence}, we can achieve an $O(s_0/t)$ bound on the estimation error of OPT-Lasso for sufficiently large $t$. On one hand, in order for \eqref{def:incoherence} to hold with high probability, we must have $t = \Omega(s_0^2\log d)$ (see discussions following Theorem 5.1 in \cite{bellec2022biasing}), which would introduce a $s_0^2 \log d$ term in the expected cumulative error. Thus, the relaxed beta-min condition \eqref{def:incoherence} is not applicable.
On the other hand, if we expect \eqref{def: beta-min} 
 to hold for $t = O(s_0\log(dT))$, and if $\log T = O(\log d)$, condition \eqref{def: beta-min} becomes: $  \min_{j \in S} |\btheta_j| = \Omega(1)$. 
That is, the minimal signal strength must be bounded away from zero by a constant. This condition is highly restrictive and, moreover, prevents us from establishing a matching lower bound, as discussed in Remark \ref{remark:beta_min_lower_bound_diff}.

Next, we focus on the strategy for deriving the instance-specific bound in \eqref{eq: opt_lasso_t} for a \textit{fixed}, sufficiently large round $t$. 
The key tool is drawn from Theorem 5.1 in \cite{bellec2022biasing}, which provides the following bound on the $\ell_{\infty}$ estimation error of Lasso estimators: for each $t \geq C s_0\log d$, we have that with high probability, $\|\ltheta_t -\btheta \|_{\infty} \le C\lambda_t$;
see Lemma \ref{lemma: improved_ell_infty} in Appendix  \ref{sec: bounds on sequential estimation} for a precise statement. This implies that with a proper choice of $\hdlambda_t$, the support $\pS_t$ defined in Definition \ref{def: OPT-Lasso} is contained in the true support $S$, and contains those components with a magnitude at least $2\hdlambda_t$, that is,
\begin{align}\label{OPT_Lasso_selection_prop}
  S_t^{\textup{str}} := \{j\in[d]: |\btheta_j|\ge 2\hdlambda_t\} \;\subseteq \;   \pS_t \subseteq S.
\end{align}
In Theorem \ref{thm: ols post l2 estimation error} in Appendix  \ref{app:lasso_deterministic_analysis}, we show that the impact of  ``missing'' covariates---those in $S\setminus S_t^{\textup{str}}$---can be controlled, as their magnitude is at most $2\hdlambda_t$. 

An important feature of Theorem 5.1 in \cite{bellec2022biasing} is that it applies when $t = \Omega(s_0\log d)$. In contrast, previous $\ell_{\infty}$ bounds, such as Theorem 7.21 in \cite{wainwright2019high} and Lemma 4.1 in \cite{van2016estimation}, generally require $t = \Omega(s_0^2\log d)$. This refinement contributes to the tighter upper bound established in Theorem \ref{thm: seq multi linear, upper bound}.

Finally, we establish a matching minimax lower bound for all permissible rules, showing that the algorithm based on the OPT-Lasso estimators is minimax optimal. To this end, we strengthen Assumption \ref{assumption: bounded cov mat 1} as follows.

\begin{assumption}\label{assumption: bounded cov mat 1'}
For some constant $L_1 > 1$, 
$L_1^{-1}\le \lmin{\bSigma}\le\lmax{\bSigma}\le L_1$. 
\end{assumption}

Denote by $\bTheta_d[s_0]=\{\btheta\in \mR^d: \|\btheta\|_0\leq s_0\}$ the set of all $s_0$-sparse vectors.

\begin{theorem}\label{thm: seq multi linear, lower bound}
Suppose that Assumption \ref{assumption: bounded cov mat 1'} holds, and that $3 \leq s_0 \leq (d+2)/3$ and $T\ge 2\sigma^2 \xi^{-1}s_0\log d$. 
 Then, there exists a constant $C>0$ depending only on $L_1, \xi$, such that for any permissible estimators $\{\hbtheta_t, t\in [T]\}$, we have the following lower bound:
$$\sup_{\btheta\in \bTheta_d[s_0]} \sum_{t=1}^{T} \Exp_{\btheta} \left[\|\hbtheta_t-\btheta\|_2^2\wedge \xi \right] \ge C^{-1} \sigma^2 s_0 \left(\log\left( \frac{T+\sigma^2 s_0}{2+\sigma^2s_0} \right)+ \log\frac{d}{s_0} \right).
$$
Consequently, if, additionally, $d \geq s_0^{\kappa}$ and $T \geq (2+\sigma^2 s_0)^{\kappa}$ for some $\kappa > 1$, then 
\begin{align*}
    \sup_{\btheta\in \bTheta_d[s_0]} \sum_{t=1}^{T} \Exp_{\btheta} \left[\|\hbtheta_t-\btheta\|_2^2\wedge \xi \right] \ge C^{-1} \sigma^2  (1-1/\kappa) s_0\left(\log\left( d\right)+ \log\left(T\right) \right).
\end{align*}
\end{theorem}
\begin{proof}
See Appendix \ref{sec: proof of seq multi linear, lower bound}.
\end{proof}

The proof strategy for Theorem \ref{thm: seq multi linear, lower bound} is to lower bound the worst-case risk over $\btheta \in \bTheta_d[s_0]$ by the Bayes risk with respect to a prior distribution, which is a \textit{mixture} of two distributions, $\omega_1$ and $\omega_2$ on $\bTheta_d[s_0]$. Under $\omega_1$, $\btheta$ is supported on $[s_0]$---that is, the first $s_0$-components---corresponding to a low-dimensional setup. We then apply van Tree's inequality \citep{gill1995applications} to show that the Bayes risk under $\omega_1$ is $\Omega(s_0\log T)$. Further, $\omega_2$ is a finitely discrete distribution constructed based on Theorem 1(b) in \cite{raskutti2011minimax}, and we apply Fano's inequality (see \cite[Theorem 4.10]{rigollet2023highdimensionalstatistics}) to show that the Bayes risk under $\omega_2$ is $\Omega(s_0\log d)$.

\begin{remark} \label{remark:beta_min_lower_bound_diff}
Under the ``beta-min'' condition  $\min_{j \in S} |\btheta_j| = \Omega(1)$, we would need to construct a prior distribution supported on $\{\btheta \in \bTheta_d[s_0]: \min_{j \in [d]}\{|\btheta_j|; \btheta_j \neq 0\} \geq C^{-1}\}$ for some constant $C > 0$. However, neither $\omega_1$ nor $\omega_2$ mentioned above satisfies this property.
\end{remark}
 
Combining Theorems \ref{thm: seq multi linear, upper bound} and \ref{thm: seq multi linear, lower bound}, we obtain the following corollary, which characterizes the minimax rate of the expected cumulative error as $s_0(\log d + \log T)$ and shows that the OPT-Lasso estimators are minimax optimal in the sequential setup.

\begin{corollary} \label{cor: seq multi linear}
Suppose that Assumptions \ref{assumption: sequential estimation} and \ref{assumption: bounded cov mat 1'} hold, and that for some $\kappa > 1$,
\begin{align*}
    \begin{split}
    &s_0 \geq 3, \quad d \ge \max\{(2L_0+1)s_0 + 2,\; s_0^{\kappa}\}, \;\;
    T\ge \max\{2\sigma^2 \xi^{-1}s_0\log d,\; (2+\sigma^2 s_0)^{\kappa}\}.
        \end{split}
\end{align*}
Then, there exists a constant $C>0$ only depending on $L_0,L_1, \xi$, such that 
$$C^{-1}(1-1/\kappa)\sigma^2 \le \frac{\inf_{\{\hbtheta_t, t\in [T]\}}\sup_{\btheta\in \bTheta_d[s_0]} \Exp_{\btheta}\sum_{t=1}^{T} (\|\hbtheta_t-\btheta\|_2^2\wedge \xi)} {s_0(\log d + \log T)} \le C((\sigma^2 \vee 1) + \xi),$$
where the infimum is taken over all permissible estimators. Further, the minimax rate is achieved by the OPT-Lasso estimators if we set the tuning parameters as in \eqref{eq: online regularization threshold parameter}.
\end{corollary}

\begin{remark}
Note that in \eqref{eq: online regularization threshold parameter}, both $\lambda_t$ and $\hdlambda_t$ are of order $\sqrt{\log(dt)/t}$. Without further assumptions on $d$ and $T$, OPT-Lasso estimators are minimax optimal in terms of the \textit{cumulative} error. On the other hand, for a fixed, sufficiently large $t$, applying  \eqref{pess_bound}, the upper bound in \eqref{eq: opt_lasso_t} on the mean-squared estimation error of OPT-Lasso at time $t$ becomes $O(s_0\log(dt)/t)$. Consequently, for a \textit{fixed} $t$, OPT-Lasso would be minimax optimal only if $\log T = O(\log d)$. This also highlights the distinction between sequential and fixed-sample setups.
\end{remark}

\subsection{Suboptimality of Lasso in Sequential Estimation} \label{sec: suboptimality of lasso}

In this subsection, our goal is to show that, in the sequential setting, unlike OPT-Lasso, the Lasso estimators are minimax \textit{suboptimal}. Under mild conditions, it is well known that for a sufficient large $C > 0$ and for each $t \geq C s_0 \log d$, with high probability,  $\|\ltheta_t-\btheta \|_2^2 \leq C \sigma^2 s_0 \log(d)/t$; see, for example, Example 7.14 in \cite{wainwright2019high}. Thus, for the sequential setup, summing over $T$ rounds yields:
$$\sum_{i=1}^T \Exp\left[\|\ltheta_t-\btheta \|_2^2\wedge \xi\right] = O((\sigma^2+\xi)s_0\log d\log T).$$
This \textit{upper} bound does not match the minimax rate in Corollary \ref{cor: seq multi linear}, namely $s_0(\log d + \log T)$. In the following, we establish the suboptimality of the Lasso estimators by deriving a lower bound on their worst-case cumulative error.

\begin{assumption}\label{assumption: multi linear cov matrix}
There exist a constant $L_2 > 1$ and a subset $S \subseteq [d]$ such that $|S| = s_0$ and 
$\max_{j\in S^c}\left\|(\bSigma_{S,S})^{-1} \bSigma_{S,\{j\}}\right\|_1\le  L_2^{-1}.$
\end{assumption}
\begin{remark}
Assumption \ref{assumption: multi linear cov matrix} is known as the irrepresentable condition, which imposes a constraint on the correlation among covariates in the support of $\btheta$ and those outside it. For more details, we refer to \cite{zhao2006model}. 
\end{remark}

\begin{theorem}\label{thm: lower bound of lasso}
Suppose that Assumptions  \ref{assumption: bounded variance},  \ref{assumption: bounded cov mat 1}, and  \ref{assumption: multi linear cov matrix} hold. 
Let $\kappa >1$ and $C_0 > 0$ be any constant. Assume that 
\begin{align*}
   1 \geq  \lambda_t \geq \frac{16\sigma}{1 - L_2^{-1}} \sqrt{\frac{2\log d}{t}},\quad \text{ for } \; t \geq C_0 s_0^2 \log d.
\end{align*}
There exist $C \geq C_0$ depending only on $L_0, L_2, \xi$, and $C'$ depending only on $L_0, L_2$ such that if 
$d \geq C$ and $T \geq \left(C s_0^2 \log d \right)^{\kappa}$, 
then for the Lasso estimators $\{\ltheta_t: t\in[T]\}$, we have
$$\sup_{\btheta\in \bTheta_d[s_0]} \Exp_{\btheta} \left[\sum_{t=1}^T (\|\ltheta_t-\btheta \|_2^2 \wedge \xi) \right] \ge (C')^{-1} \sigma^2 (1-1/\kappa) s_0\log d\log T.$$
\end{theorem}
\begin{proof}
    See Appendix \ref{sec: proof of thm: lower bound of lasso}. 
\end{proof}

If $\lambda_t = C \sqrt{\log(d)/t}$ for a sufficiently large $C > 0$, the Lasso estimator is minimax optimal for a \textit{fixed} and sufficiently large $t$ \cite{wainwright2019high}. However, as shown in Theorem \ref{thm: lower bound of lasso}, the worst-case cumulative error of the Lasso estimators is of order $s_0\log(d)\log(T)$, indicating that Lasso is suboptimal in the sequential setting.



\subsection{Simulation Study in Sequential Estimation} \label{sec: HSLR simulation}
In this section, we present simulations to compare the performance of OPT-Lasso and Lasso estimators in the sequential setup. 

Specifically, we consider the following scenarios: (a) $s_0=5, d=100, T=10000$; (b) $s_0=10, d=500, T=10000$; (c) $s_0=5,d=1000, T=5000$; (d) $s_0=10,d=1000, T=5000$.  For each scenario, we perform 200 repetitions and report the average of various metrics for OPT-Lasso and Lasso estimators under different tuning parameters. In each repetition, for the parameter $\btheta$, we first randomly select $s_0$ indices, with their values independently sampled from the uniform distribution on $[0,1]$; the remaining indices are set to zero. Further, for each $t \in [T]$, $\bX_t$ follows a normal distribution with zero mean and an identity covariance matrix, i.e., $\bSigma=\mathbb{I}_d$, and the variance of the noise $\epsilon_t$ is set to $\sigma^2 = 1$.

For the Lasso estimators in Definition \ref{def: lasso} and the OPT-Lasso estimators in Definition \ref{def: OPT-Lasso}, we set the regularization parameter $\lambda_t = C_0\sqrt{(\log d)/t}$ and the threshold parameter $\hdlambda_t=C_0^{\textup{hard}}(C_0\sqrt{[\log (dt)]/t})$, where $C_0$ is chosen from $\{0.4,0.6,0.8,1,1.2\}$ and $C_0^{\textup{hard}}$ is chosen from $\{0.2,0.4,0.6,0.8,2\}$.

In Table \ref{fig: HSLR compare}, we consider two choices for the tuning parameters, and compare the cumulative estimation error of OPT-Lasso and Lasso, computed from time $T/10$ to $T$.  In Figure \ref{fig: HSLR sensitivity} for scenario (c), for $t \in [T/10,T]$, we report the \textit{running} cumulative error from time $T/10$ to time $t$ of OPT-Lasso, Lasso, and an \textit{oracle}, which knows the true support and runs the ordinary least squares on it. The results for the other scenarios can be found in Figure \ref{fig: HSLR plot 1} in Appendix \ref{app: more_simulations}. We observe that, across scenarios, OPT-Lasso consistently achieves smaller cumulative estimation error than Lasso.

In Figure \ref{fig: HSLR support}, we compare the support recovery property of the OPT-Lasso and Lasso estimators for scenario (c) and $C_0=0.8, C_0^{\textup{hard}}=0.6$. Specifically, for each time $t \in [T]$, we report the number of false positives (incorrectly selecting covariates not in the true support) and false negatives (failing to select covariates in the true support). The results show that the better performance of OPT-Lasso compared to Lasso is due to OPT-Lasso having significantly smaller false positives. Further, as discussed in \eqref{OPT_Lasso_selection_prop}, OPT-Lasso does contain ``strong signals''. Finally, we present a sensitivity analysis of OPT-Lasso and Lasso with respect to the tuning parameters $C_0$ and $C_0^{\textup{hard}}$, reported in Table \ref{fig: HSLR sensitivity 1} in Appendix \ref{app: more_simulations}. The results indicate that both estimators are sensitive to the choice of tuning parameters, highlighting the importance of developing principled tuning strategies as a direction for future research.


\begin{table}[t!]
\centering
\caption{Each row corresponds to one scenario. The second and third columns show the cumulative estimation error from round $T/10$ to $T$, along with its standard variance, of OPT-Lasso with 
$(C_0,C_0^{\textup{hard}}) = (0.8,0.6)$ and $(1,0.4)$,  respectively.
The last two columns are for  Lasso with $C_0=0.8$ and $1$, respectively.
} \label{fig: HSLR compare}
\begin{tabular}{l|cc|cc}
  $(s_0,d,T)$   & (0.8, 0.6) & (1, 0.4)   & 0.8         & 1.0           \\ \hline
(a). $(5,100,10000)$ & 16.0 $\pm$ 0.4   & 15.9 $\pm$ 0.3 & 52.5 $\pm$ 1.5  & 64.6 $\pm$ 1.4  \\
(b). $(10,500,10000)$ & 33.3 $\pm$ 0.9 & 33.5 $\pm$ 0.7 & 120.7 $\pm$ 2.2 & 162.4 $\pm$ 2.1 \\
(c). $(5,1000,5000)$ & 21.1 $\pm$ 0.4 & 20.6 $\pm$ 0.4 & 79.9 $\pm$ 1.1  & 94.3 $\pm$ 1.9  \\
(d). $(10,1000,5000)$  & 40.5 $\pm$ 0.8 & 41.0 $\pm$ 0.6   & 144.9 $\pm$ 1.0   & 182.0 $\pm$ 1.2   \\  
\end{tabular}
\end{table}

\begin{figure}[t!] 
\centering
\includegraphics[width=0.93\textwidth]{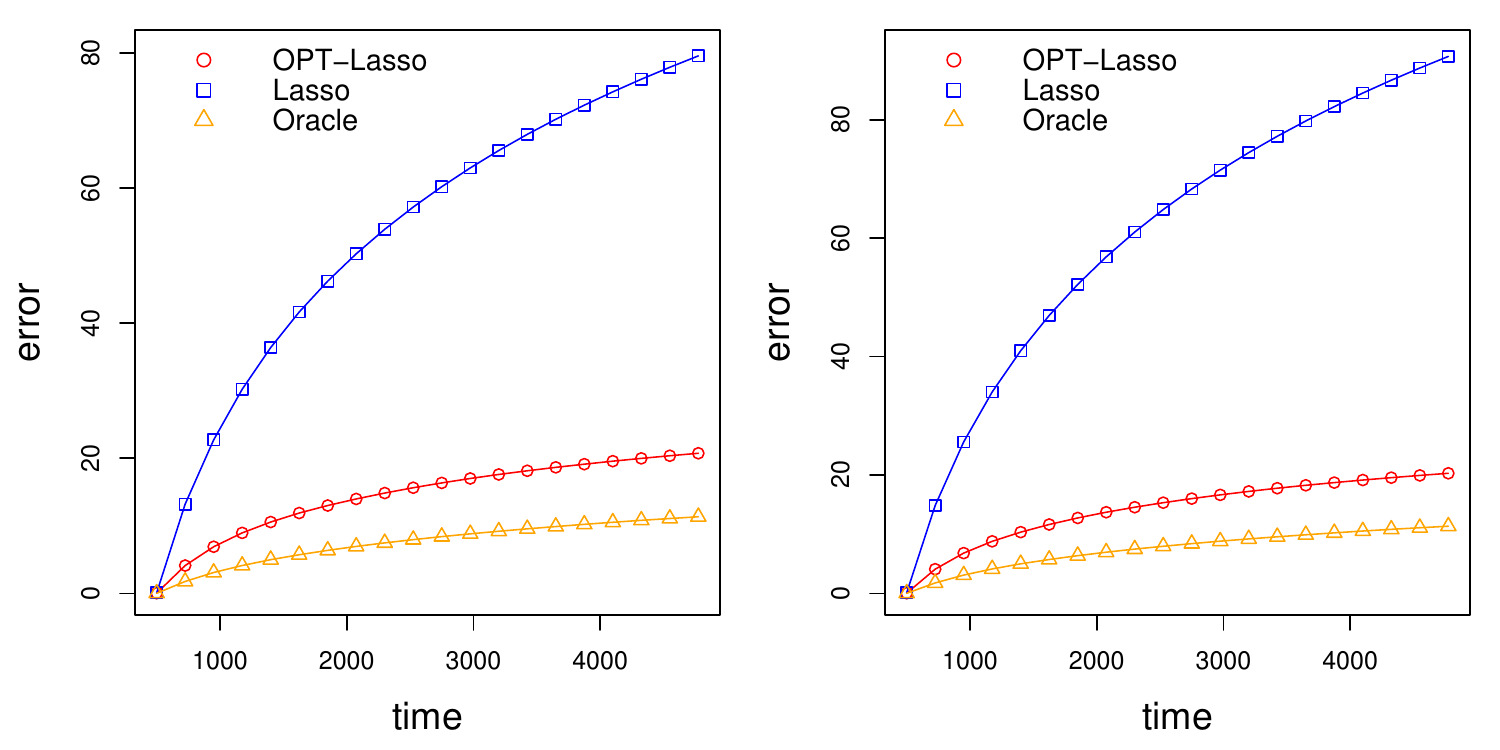}
\caption{
The $x$-axis represents time, and the $y$-axis the cumulative error from time $T/10$ to time $t \in [T/10,T]$. For scenario (c), we report the running cumulative estimation error of OPT-Lasso, Lasso and an Oracle (``LS''). The left plot corresponds to $C_0=0.8, C_0^{\textup{hard}}=0.6$ and the right plot to $C_0=1, C_0^{\textup{hard}}=0.4$.
} \label{fig: HSLR sensitivity}
\end{figure}

\begin{figure}[t!] 
\centering
\includegraphics[width=0.93\textwidth]{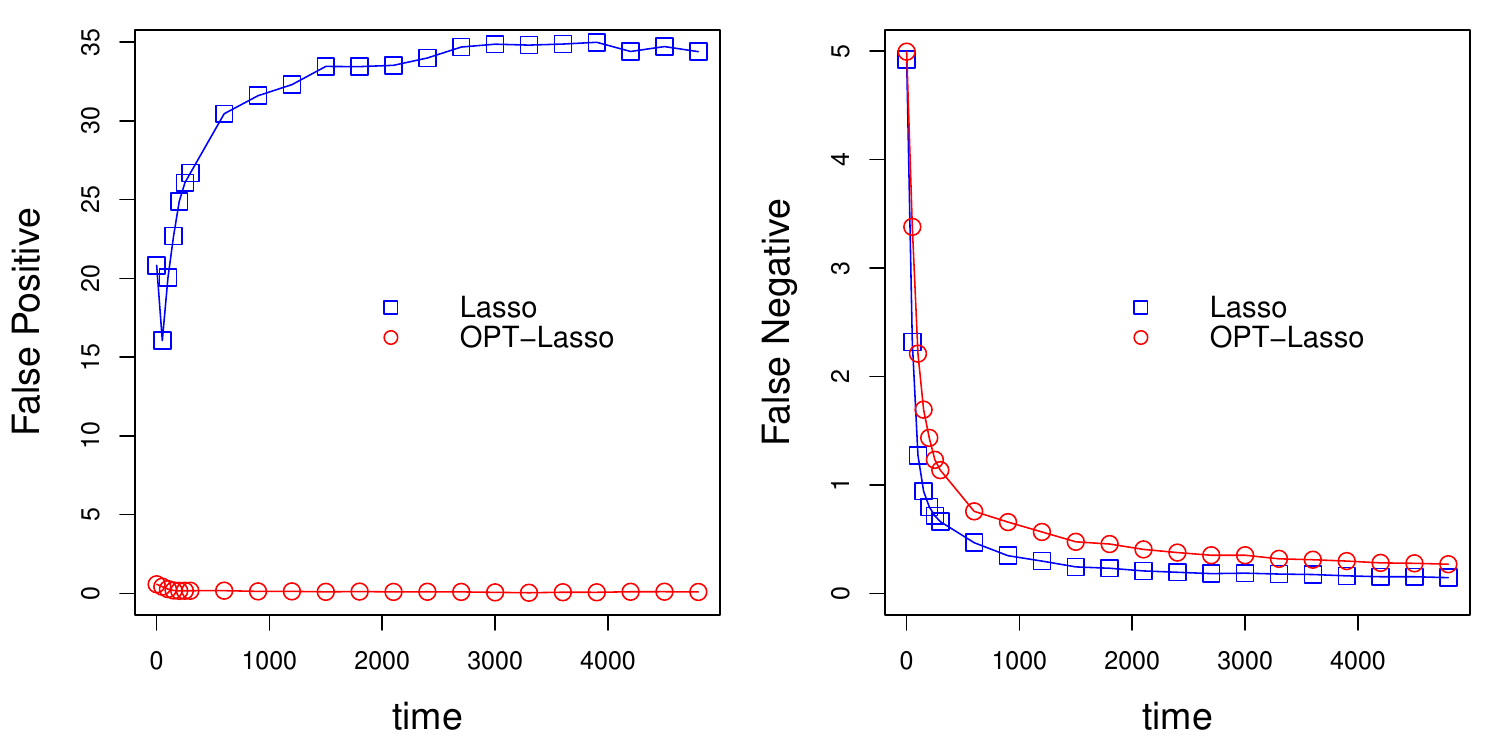}
\caption{
We consider scenario (c) and set $C_0=0.8, C_0^{\textup{hard}}=0.6$. The left (resp.~right) plot shows the number of false positives (resp.~negatives) at each time $t \in [T]$ for Lasso and OPT-Lasso.
} \label{fig: HSLR support}
\end{figure}

\section{High-dimensional Linear Contextual Bandits} \label{sec: problem formulation}
In this section, we investigate the linear contextual bandit problem in high dimensions. Specifically, consider a set of $K \geq 2$ arms, each associated with an arm parameter $\btheta^{(k)} \in \mathbb{R}^d$. In addition, assume that there are $T$ rounds, and at each round/time $t \in [T]$, we observe a covariate vector $\bX_t \in \mathbb{R}^d$, select an arm $A_t \in [K]$, and receive a reward $Y_t \in \bR$, which satisfies the following linear model:
\begin{equation*} 
    Y_t = \bX_t' \btheta^{(A_t)} + \epsilon_t.
\end{equation*}

We assume that the covariate vectors $\{\bX_t : t \in [T]\}$ are i.i.d., and that $\{\epsilon_t \in \mathbb{R} : t \in [T]\}$ are zero-mean i.i.d. random variables that are $\sigma$-sub-Gaussian, i.e., $\Exp[\exp(\tau \epsilon_1)] \leq \exp(\sigma^2 \tau^2 / 2)$ for $\tau \in \mathbb{R}$. Further, these two sequences are independent. At each round $t \in [T]$, we can select the arm $A_t$ based solely on $\bX_t$ and the previous observations $\{ (\bX_s, A_s, Y_s) : s < t \}$, possibly with the help of a random variable $U_t$.  Specifically, a permissible rule $\{\pi_t : t \in [T]\}$ is a sequence of measurable functions such that 
$$ 
A_t = \pi_t\left( \bX_t, \{(\bX_s, A_s, Y_s) : s < t\}, U_t \right),\; \text{ for } t \in [T],
$$
where $\{U_t : t \in [T]\}$ are i.i.d. random variables uniformly distributed on $(0,1)$ and independent of all other random variables. Let $\cF_t := \sigma(\{(\bX_s, A_s, Y_s), s\le t\})$ be the $\sigma$-algebra generated by the observations up to time $t \geq 0$.

Given a permissible rule $\{\pi_t : t \in [T]\}$, or equivalently the arm selections $\{A_t : t \in [T]\}$, we define the instantaneous regret $r_t$ at time $t \in [T]$ as follows:
\begin{equation}
    \label{eq: def of regret}
    r_t := \bX_{t}'\left(\btheta^{(A_t^*)} - \btheta^{(A_t)} \right), \;\; \text{ where } \;A_t^*:= \argmax_{k \in [K]} \bX_t'\btheta^{(k)}.
\end{equation}
Our goal is to find a procedure that minimizes the expected cumulative regret $\sum_{t \in [T]} \Exp[r_t]$.

We focus on the high-dimensional regime, where the time horizon $T$ may exceed the dimension $d$, but the arm parameters are assumed to be sparse. Specifically, for each arm $k \in [K]$, let $S_k = \{ j \in [d] : \btheta^{(k)}_j \neq 0 \}$ denote the support of $\btheta^{(k)}$, and define $s_0 := \max_{k \in [K]} |S_k|$ as the maximum support size across all arms. We consider the setting where $s_0 \ll d$. Throughout, we assume $s_0,d, T \geq 2$.

\subsection{A Three-Stage Algorithm} \label{subsec:three_stage_algo}
We propose a three-stage algorithm that separates exploration and exploitation, as depicted in Figure \ref{fig:proposed_three_stages} and described in detail below.

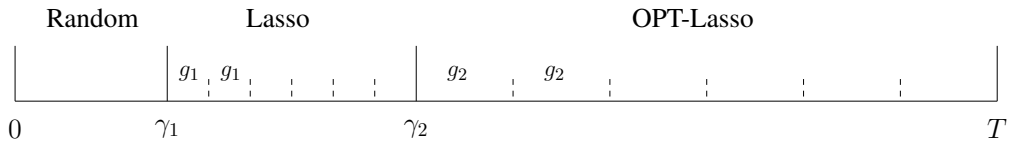
\begin{figure}[!t]
\centering
\resizebox{0.95\textwidth}{!}{
\begin{circuitikz}
\tikzstyle{every node}=[font=\large]
\draw [short] (6,24) -- (23.75,24);
\draw [short] (6,25) -- (6,24);
\draw [short] (8.75,25) -- (8.75,24);
\draw [dashed] (9.5,24.4) -- (9.5,24);
\draw [dashed] (10.25,24.4) -- (10.25,24);
\draw [dashed] (11,24.4) -- (11,24);
\draw [dashed] (11.75,24.4) -- (11.75,24);
\draw [dashed] (12.5,24.4) -- (12.5,24);
\draw [short] (13.25,25) -- (13.25,24);
\draw [dashed] (15,24.4) -- (15,24);
\draw [dashed] (16.75,24.4) -- (16.75,24);
\draw [dashed] (18.5,24.4) -- (18.5,24);
\draw [dashed] (20.25,24.4) -- (20.25,24);
\draw [dashed] (22,24.4) -- (22,24);
\draw [short] (23.75,25) -- (23.75,24);

\node [font=\Large] at (7.4,25.5) {Random};
\node [font=\Large] at (10.75,25.5) {Lasso};
\node [font=\Large] at (18.25,25.5) {OPT-Lasso};
\node [font=\large] at (9.15,24.5) {$g_1$};
\node [font=\large] at (9.9,24.5) {$g_1$};
\node [font=\large] at (14,24.5) {$g_2$};
\node [font=\large] at (15.75,24.5) {$g_2$};
\node [font=\Large] at (6,23.5) {$0$};
\node [font=\Large] at (8.75,23.5) {$\gamma_1$};
\node [font=\Large] at (13.25,23.5) {$\gamma_2$};
\node [font=\Large] at (23.75,23.5) {$T$};
\end{circuitikz}
}
\caption{An illustration of the three-stage algorithm. The three stages are as follows: a pure exploration stage $(0, \gamma_1]$ with random arm selection; an exploitation stage $(\gamma_1, \gamma_2]$ based on Lasso estimators; and a final exploitation stage $(\gamma_2, T]$ based on OPT-Lasso estimators. During Stages 2 and 3, the estimators are updated every $g_1$ and $g_2$ rounds, respectively.}
\label{fig:proposed_three_stages}
\end{figure}

For round $t \in [1, \gamma_1]$, where $\gamma_1$ is to be specified, due to data scarcity, we perform pure exploration by randomly selecting one of the $K$ arms. Afterwards, we estimate the parameter of each arm and select the arm that maximizes the \textit{estimated} reward, i.e., by pure exploitation. For the exploitation period, we divide it into two segments. For round $t \in (\gamma_1, \gamma_2]$, where $\gamma_2$ is to be specified, we adopt Lasso as the estimation method, while for round $t \in (\gamma_2, T]$, we use OPT-Lasso. For computational efficiency, we do not update the Lasso (resp.~OPT-Lasso) estimator at each round, but rather every $g_1$ (resp.~$g_2$) rounds, where $g_1$ and $g_2$ are tuning parameters.

Specifically, for each time $t \in [T]$ and arm $k \in [K]$, we define
\begin{equation*}
\begin{aligned}
    &\bX_{[t]}^{(k)} := \left(\bX_1 \mathbb{I}\{A_1 = k\},\; \cdots\;, \bX_t \mathbb{I}\{A_t = k\}\right)', \\
    &\bY_{[t]}^{(k)} := \left(Y_1 \mathbb{I}\{A_1 = k\},\; \cdots\;, Y_t \mathbb{I}\{A_t = k\}\right)',
\end{aligned}
\end{equation*}
where $(\bX_{[t]}^{(k)}, \bY_{[t]}^{(k)})$ represents the data available for estimating $\btheta^{(k)}$ up to time $t$. Furthermore, for each $t \in [T]$ and $k \in [K]$, given a regularization parameter $\lambda_t^{(k)}$ and a hard-threshold parameter $\lambda_t^{\textup{OPT},(k)}$, we denote by
\begin{equation*}
\hbtheta_t^{(k)} = \ltheta_t\left(\bX_{[t]}^{(k)}, \bY_{[t]}^{(k)}, \lambda_t^{(k)}\right), \quad \tbtheta_t^{(k)} = \ptheta_t\left(\bX_{[t]}^{(k)}, \bY_{[t]}^{(k)}, \lambda_t^{(k)}, \lambda_t^{\textup{OPT},(k)}\right), 
\end{equation*}
the Lasso and OPT-Lasso estimators, respectively, of the $k$-th arm parameter $\btheta^{(k)}$, where we recall their definitions in Definition \ref{def: lasso} and \ref{def: OPT-Lasso}. In general, $\lambda_t^{(k)}$ and $\lambda_t^{\textup{OPT},(k)}$ can depend on the data up to time $t$, that is, $\cF_t$-measurable.

\begin{definition} \label{def:three_stages}

Given positive integers $\gamma_1, \gamma_2, g_1, g_2$ such that $\gamma_1 \leq \gamma_2$, and  parameters $\{\lambda_t^{(k)}, \lambda_t^{\textup{OPT},(k)} : t \in [T], k \in [K]\}$, we formally define the three-stage algorithm as follows:
\begin{enumerate}[label={Stage \arabic*} ]
    \item (Random selection).  At round $t \in [1, \gamma_1]$, $A_t$ is randomly selected from $[K]$ according to a uniform distribution;
    \item(Lasso).  At round $t \in (\gamma_1+mg_1,(\gamma_1+(m+1)g_1) \wedge \gamma_2]$ with $m\in\{0,1,2,\cdots\}$, we use the Lasso estimators at time $\gamma_1 + m g_1$, and select the arm  greedily:
    $$A_t = \text{argmax}_{k\in[K]} (\hbtheta_{\gamma_1+mg_1}^{(k)})' \bX_t.$$
    \item(OPT-Lasso).   At round $t \in (\gamma_2+mg_2,(\gamma_2+(m+1)g_2) \wedge T]$ with $m\in\{0,1,2,\cdots\}$, we use the OPT-Lasso estimators at time $\gamma_2 + m g_2$, and select the arm  greedily:
    $$A_t = \text{argmax}_{k\in[K]} (\tbtheta_{\gamma_2+mg_2}^{(k)})' \bX_t.$$ 
\end{enumerate}
Here, ties in the “argmax” are broken either according to a fixed rule or randomly.
\end{definition}  

A few remarks are in order. First, based on the results in Section~\ref{sec: multi linear}, OPT-Lasso is the preferred estimation method over Lasso in the sequential setting. Therefore, the durations of the first two stages should be significantly shorter than that of the third stage; that is, we choose $\gamma_2 \ll T$. Second, the second stage is introduced primarily for technical reasons. Specifically, as we explain in Appendix \ref{subsec:discussion_bandit_est}, the $\ell_2$ estimation error of the Lasso estimators can be bounded when $t = \Omega(s_0 \log(dT))$. In contrast, our analysis of the OPT-Lasso estimators requires $t = \Omega(s_0^5 \log(dT))$. To bridge this gap, we include an intermediate stage that uses Lasso before transitioning to OPT-Lasso in the third stage. In the simulation studies presented in Section~\ref{sec: bandit simulation}, we observe that including Stage 2 leads to a slight improvement in the performance of the proposed procedure. Finally, in Stage 2 (resp.~Stage 3), the estimators are updated every $g_1$ (resp.~$g_2$) steps and may use arm-specific tuning parameters $\lambda_t^{(k)}$ (resp.~$\lambda_t^{(k)}$ and $\lambda_t^{\textup{OPT},(k)}$). For simplicity and to reduce the number of free parameters in the subsequent \textit{analysis}, we set
\begin{align}\label{parameters_for_analysis}
    g_1 = \gamma_1, \;\; \quad g_2 = \gamma_2, \;\; 
    \lambda_t^{(k)} = \lambda_t \;\text{ and }\;\lambda_t^{\textup{OPT},(k)} = \lambda_t^{\textup{OPT}}\; \text{ for } k \in [K],\; t \in [T].
\end{align}
This choice allows us to (nearly) characterize the minimax rate of the cumulative regret.

\subsection{Minimax optimality of Proposed Algorithm}\label{subsec:bandit_minimax}
For the regret analysis, we first establish an upper bound on the cumulative regret of the proposed algorithm, followed by a (nearly) matching lower bound that holds for all admissible policies. These results demonstrate, in particular, that the proposed algorithm is (nearly) minimax optimal. We begin by stating our assumptions.

A Lebesgue-density function $f:\mR^d \mapsto [0,\infty)$ is called log-concave  if the function $\log f:\mR^d \mapsto [-\infty,\infty)$ is concave.
In addition, we define the following matrices: for $t \in [T]$,
\begin{align*}
\bSigma :=\Exp[\bX_t\bX_t'], \quad
    \bSigma^{(k)} := \Exp\left[\bX_t\bX_t'\idf\{(\btheta^{(k)})'\bX_t\ge \max_{\ell \neq k}((\btheta^{(\ell)})'\bX_t)\}\right] \text{ for } k \in [K]. 
\end{align*}
Thus, $\bSigma$ is the covariance matrix of the covariate vector, while $\bSigma^{(k)}$ denotes the covariance matrix of the covariate vector on the event that the $k$-th arm is the optimal one. For a random variable $Z$,  its $\psi_{2}$ norm is defined as follows: $ \|Z\|_{\psi_{2}} := \inf\{C > 0: \exp((Z/C)^2) \leq 2\}$.

\begin{assumption}[Covariate vectors]\label{assumption: covariates_bandit}
For constants $m_X, \alpha_X, L_3>1$,  we assume
\begin{enumerate}[label={(\alph*)}, ref={\theassumption\ (\alph*)}]
\item \label{assumption: context}
For each $t \in [T]$, the covariate vector $\bX_t$ satisfies $\Exp[\bX_t] = \bd{0}_d$, $\|\bX_t\|_{\infty} \leq m_X$, and has a log-concave Lebesgue-density. Moreover, $\bX_t$ is sub-Gaussian in the sense that for every unit vector $\bu \in \bR^d$ with $\|\bu\|_2 = 1$, we have $\|\bu' \bX_t\|_{\psi_2} \leq \alpha_X$.

 \item \label{assumption: bounded cov mat}
The eigenvalues of $\bSigma$ satisfy $L_3^{-1}\le \lmin{\bSigma}\le \lmax{\bSigma} \le L_3$.
\item \label{assumption: constraint on the covariance matrix 1}
The covariance matrix for each arm $k\in[K]$ satisfies $\mn(\bSigma^{(k)})^{-1}\mn_{\infty} \le L_3$. 
\end{enumerate}    
\end{assumption}

For Assumption \ref{assumption: context}, it is common to assume that the entries of the covariate vectors are bounded; see, e.g., \cite{bastani2020online, ariu2022thresholded}. 
Examples of distributions that are bounded, sub-Gaussian, and possess log-concave densities include the truncated normal distributions and the uniform distribution over a convex and compact set in $\bR^{d}$.

When the covariate vector $\bX_t$ has a log-concave density and bounded eigenvalues as assumed in Assumption~\ref{assumption: bounded cov mat}, we highlight the following two properties.  First, if arms are selected to maximize the estimated rewards, then the conditional regret $\Exp[r_t \mid \mathcal{F}_{t-1}]$ at time $t \in [T]$ is, up to a constant depending only on $L_3$, both upper and lower bounded by the mean-squared estimation error of the arm parameters at time $t-1$; see Lemma~\ref{lemma: estErr_to_regret} in Appendix.  Second, for any unit vector $\bu \in \bR^{d}$, the random variable $\bu' \bX_t$ has a density bounded by a constant $C$ depending only on $L_3$; see Lemma~\ref{lemma: margin condition} in Appendix. This implies that for any $i \neq j \in [K]$, the probability that $\bX_t$ falls within a $\tau$-distance of the  boundary 
$\{ \bx \in \bR^{d} : (\btheta^{(i)})'\bx = (\btheta^{(j)})' \bx \}$ 
is at most $C \tau$ for any $\tau > 0$, a condition known as the margin condition \citep{goldenshluger2013linear, bastani2020online}. 
Additional properties are summarized in Appendix~\ref{sec: log-concave}.

Similar to Assumption \ref{assumption: constraint on the covariance matrix} imposed in the sequential estimation problem, Assumption \ref{assumption: constraint on the covariance matrix 1} is used to derive instance-specific upper bounds on the $\ell_2$ estimation error of the OPT-Lasso estimators for $\{\btheta^{(k)} : k \in [K]\}$. 
For example, if $K = 2$ and $\bX_t$ has a symmetric distribution, then $\bSigma_1 = \bSigma_2 = \bSigma/2$. In this case, Assumption~\ref{assumption: constraint on the covariance matrix 1} holds if $\mn (\bSigma)^{-1} \mn_{\infty} \leq L_3/2$.

Next, we state the assumptions on the arm parameters. For a vector $\bv \in \mathbb{R}^d$ and a linear subspace $\mathcal{S} \subseteq \mathbb{R}^d$, let $\textup{proj}(\bv;\; \mathcal{S})$ denote the projection of $\bv$ onto $\mathcal{S}$. Additionally, for a collection of vectors $\{\bv_1, \ldots, \bv_{\ell}\} \subseteq \mathbb{R}^d$, let $\textup{span}(\bv_1, \ldots, \bv_{\ell})$ denote the linear subspace spanned by this collection. We adopt the convention that the span of an empty set is $\{\bd{0}_d\}$.

\begin{assumption}[Arm Parameters] \label{assumption:arm_parameters}
For constants $m_{\theta} > 1$ and $L_4 > 1$, we assume
\begin{enumerate}[label={(\alph*)}, ref={\theassumption\ (\alph*)}]
\item \label{assumption: theta_bound}
$\|\btheta^{(k)}\|_2\le m_{\theta}$ for all $k\in [K]$.

\item \label{assumption: linear_independent}
Let $\bu_{i,j} := \btheta^{(i)}-\btheta^{(j)}$ for $i,j\in[K]$. Assume that for arms $k,\ell \in [K]$ with $k \neq \ell$, 
$$
\|\bu_{k,\ell}-\textup{proj}(\bu_{k,\ell};\;\textup{span}\left(\{\bu_{k,j}:j\neq k,\ell\}\right)\|_2\ge L_4^{-1}.
$$
\end{enumerate}
\end{assumption}

Intuitively, Assumption~\ref{assumption: linear_independent} requires that the arm parameters are well-separated and linearly independent.  When $K=2$,   Assumption \ref{assumption: linear_independent} is \textit{equivalent} to requiring $\|\btheta^{(1)}  - \btheta^{(2)}\|_2 \geq L_4^{-1}$. For $K \geq 2$, Assumption \ref{assumption: linear_independent} implies that $\|\btheta^{(k)}  - \btheta^{(\ell)}\|_2 \geq L_4^{-1}$ for $k \neq \ell \in [K]$. 

Note that by definition, for $k \neq \ell \in [K]$, we have
\begin{align*}
\|\bu_{k,\ell}-\textup{proj}(\bu_{k,\ell};\textup{span}(\{\bu_{k,j}:j\neq k,\ell\})\|_2 \ge \|\btheta^{(\ell)}-\textup{proj}(\btheta^{(\ell)};\; \textup{span}(\{\btheta^{(j)}:j\neq l\}))\|_2.
\end{align*}
Therefore, a sufficient and more interpretable condition for Assumption~\ref{assumption: linear_independent} is that, for each arm $k \in [K]$, $\|\btheta^{(k)}-\textup{proj}(\btheta^{(k)};\; \textup{span}(\{\btheta^{(j)}:j\neq k\}))\|_2\ge L_4^{-1}$.


Next, we provide an upper bound on the cumulative regret of the proposed three-stage algorithm defined in Definition \ref{def:three_stages}. For simplicity, we assume that the parameter choices in \eqref{parameters_for_analysis} hold, and recall that the instantaneous regret $r_t$ is defined in~\eqref{eq: def of regret}.

\begin{theorem} \label{thm: regret over horizon}
Suppose Assumptions \ref{assumption: covariates_bandit} and  \ref{assumption:arm_parameters} hold, and consider the proposed three-stage algorithm in Definition \ref{def:three_stages}. Set the regularization and threshold parameters as
\begin{equation} \label{eq: bandit_paras_est}
\lambda_t = 6m_X\sigma\sqrt{\frac{\log (dT)}{t}},\quad \text{ and } \quad \hdlambda_t = 28L_3\lambda_t,
\end{equation}
and for some integers $C_{\gamma_1}$ and $C_{\gamma_2}$, set the end times of Stage 1 and 2, respectively, as follows
\begin{equation} \label{eq: end_times}
    \gamma_1 = C_{\gamma_1} \lceil(\sigma^2\vee 1)\rceil s_0 \lceil \log (dT) \rceil, \quad \text{ and } \quad \gamma_2 = C_{\gamma_2} s_0^4 \gamma_1.
\end{equation} 
Then, there exist constants $C,C^* > 0$ depending only on $K,  m_X, \alpha_X,  L_3, m_{\theta}, L_4$, such that if $C_{\gamma_1} \geq C^*$ and $C_{\gamma_2} \geq C^*$, then
$$
\Exp\left[\sum_{t=1}^T r_t \right] \le C C_{\gamma_1} (\sigma^2\vee 1) (\log s_0) s_0(\log d + \log T).
$$ 
Further, the following upper bound on the cumulative regret in Stage 3 holds:
$$\Exp\left[\sum_{t=\gamma_2+1}^T r_t \right] \le C (\sigma^2\vee 1) s_0(\log d + \log T).$$
\end{theorem}
\begin{proof}
The proof is provided in Section \ref{sec: proof of thm: regret over horizon}.
\end{proof}

Under the same problem formulation,   \cite{bastani2020online} proposes the Lasso bandit algorithm, and provides an upper bound on the expected regret of order $O(s_0^2(\log d+\log T)^2)$, which has worse dependence on $s_0, d, T$ compared to the upper bound established above. As we demonstrate in Subsection \ref{sec: bandit simulation}, algorithms based on Lasso estimators are provably \emph{suboptimal}.

Furthermore, \cite{ariu2022thresholded} proposes a related algorithm called the ``thresholded Lasso bandit'',   which does not have the pure exploration stage---that is,  it does not include Stage 1---and applies thresholding twice to the initial Lasso estimators.  For the analysis, they assume that $s_0$ is constant and that the nonzero components of the parameter vectors are bounded away from zero by a constant, that is,
\begin{align}\label{ariu_assumption}
    \text{(i) } s_0 \text{ is constant,} \quad \text{and} \quad \text{(ii) } \min_{j \in S_k,\; k \in [K]} |\btheta^{(k)}_j| = \Omega(1).
\end{align}
See Assumption 3.1 in \cite{ariu2022thresholded}.  Under this restrictive assumption,  
\cite{ariu2022thresholded} establishes an upper bound on the cumulative regret of order $O( (\log\log d)  \log d  +  \log T)$ (see Theorem 5.2 therein). Under Assumption \eqref{ariu_assumption}, our result in Theorem \ref{thm: regret over horizon} provides a mild improvement by a factor of $\log\log d$ on the dependence of $d$.  Importantly, our analysis allows   $s_0$ to grow with $d$ and $T$, and does not require the beta-min condition. Moreover, if the beta-min condition is imposed, we are unable to establish a matching lower bound as in Corollary~\ref{cor:bandit_minimax_rate}—that is, the minimax rate under this assumption remains unclear.


Finally, we show that the algorithm is nearly minimax optimal with respect to $s_0,d,T$, by establishing a nearly matching lower bound of the cumulative regret under the case $K=2$.  We define the following space of parameters for $\btheta^{(1)}$ and $\btheta^{(2)}$: for $L > 0$,
\begin{align*}
\begin{split}
        \widetilde{\bTheta}_{d}(s_0,L) := & \left\{(\btheta^{(1)},\btheta^{(2)}) \in \bR^{d + d}:\;\; \|\btheta^{(1)}\|_0 \leq s_0, \; \;\|\btheta^{(2)}\|_0 \leq s_0, \;\; \right.\\
    &  \left. \qquad\qquad \|\btheta^{(1)}\|_{2} \leq 2L, \;\; \|\btheta^{(2)}\|_{2} \leq 2L, \;\;
    \|\btheta^{(1)} -\btheta^{(2)} \|_{2} \geq L^{-1} \right\}.
    \end{split}
\end{align*}
As $L$ increases, $\widetilde{\bTheta}_{d}(s_0,L)$ becomes larger. Further, as $L \uparrow \infty$, $\widetilde{\bTheta}_d(s_0, L)$ exhausts all possible pairs of $s_0$-sparse vectors in $\bR^{d}$.

\begin{theorem}\label{thm: regret lower bound}
Consider the case $K=2$, and let $L > 1$ and $C_{*} > 0$ be arbitrary fixed constants. Suppose that Assumptions \ref{assumption: context} and \ref{assumption: bounded cov mat} hold for $\{\bX_t: t \in [T]\}$, and that the noise $\{\epsilon_t:t\in[T]\}$ are $i.i.d.$ normal random variables with distribution $\mathcal{N}(0,\sigma^2)$. Further, assume that for some constant $\kappa > 1$, we have
\begin{equation}\label{eq: regret lower bound 1}
3 \leq s_0 \leq (d+2)/3, \quad  d \geq s_0^{\kappa}, \quad 
T \geq \left( 2 (C_{*} \vee \sigma^2 \vee 1) s_0^5 \log(dT)\right)^{\kappa}.
\end{equation}
Then there exists a constant $C>0$ depending only on $L_3$ and $L$, such that for any permissible rule,  the following lower bound on its worst-case cumulative regret  over $\widetilde{\bTheta}_{d}(s_0,L)$, starting from time $\tilde{\gamma} := \left\lfloor  C_{*} s_0^5\log(dT) \right\rfloor$ holds:
$$
\sup_{(\btheta^{(1)}, \btheta^{(2)}) \in \widetilde{\bTheta}_{d}(s_0,L)}  \Exp_{\btheta^{(1)}, \btheta^{(2)}}\left[\sum_{t=\tilde{\gamma}}^T r_t \right] \geq  C^{-1} \sigma^2 (1-1/\kappa) s_0 \left(\log d + \log T \right),$$
where the expectation $\Exp_{\btheta^{(1)}, \btheta^{(2)}}$ is taken under the true parameter values $\btheta^{(1)}$ and $\btheta^{(2)}$.
\end{theorem}
\begin{proof}
The proof is provided in Appendix \ref{sec: proof of cumulative regret, lower bound}.
\end{proof}

\begin{remark}
The key argument is as follows: since the context vectors have a log-concave density with bounded eigenvalues, the conditional expected instantaneous regret $\Exp[r_t \vert \cF_{t-1}]$ can be lower bounded by the best possible mean-squared estimation errors of the arm parameters $\btheta^{(1)}$ and $\btheta^{(2)}$ at time $t-1$, using Lemma \ref{lemma: estErr_to_regret} in Appendix. The remainder of the argument then follows similarly to the proof of Theorem \ref{thm: seq multi linear, lower bound}.
\end{remark}

Note that a lower bound on the cumulative regret computed from time $\tilde{\gamma}$ also serves as a lower bound on the cumulative regret computed from the beginning (i.e., from time $1$). Similarly, a lower bound for a larger $L$ also applies to any smaller $L$.     
Combining Theorem \ref{thm: regret over horizon} with Theorem \ref{thm: regret lower bound}, we immediately obtain the following corollary.

\begin{corollary}\label{cor:bandit_minimax_rate}
Consider the case $K=2$ and let $L, \kappa > 1$ be fixed constants. Suppose that Assumption \ref{assumption: covariates_bandit} holds, and  that the noise $\{\epsilon_t:t\in[T]\}$ are $i.i.d.$ normal random variables with distribution $\mathcal{N}(0,\sigma^2)$. 
Then there exist constants $C^*,C>0$ depending only on $m_X, \alpha_X,  L_3, L$, such that 
if condition \eqref{eq: regret lower bound 1} holds with $\kappa$ and $C^*$ in place of $C_*$, then
$$
C^{-1} (1-1/\kappa)\sigma^2 \leq \frac{\inf_{\{\pi_t\}}\sup_{(\btheta^{(1)}, \btheta^{(2)}) \in \widetilde{\bTheta}_{d}(s_0,L)}  \Exp_{\btheta^{(1)}, \btheta^{(2)}} \sum_{t=1}^T r_t }{ s_0 \left(\log d + \log T \right)}
 \leq  C (\sigma^2 \vee 1) \log(s_0),$$
 where the infimum is taken over all permissible rules $\{\pi_t:t \in [T]\}$. Further, if the cumulative regret is computed starting from time $\tilde{\gamma} := \lfloor C^* s_0^5\log(dT)\rfloor$, then
 $$
C^{-1} (1-1/\kappa)\sigma^2 \leq \frac{\inf_{\{\pi_t\}}\sup_{(\btheta^{(1)}, \btheta^{(2)}) \in \widetilde{\bTheta}_{d}(s_0,L)}  \Exp_{\btheta^{(1)}, \btheta^{(2)}} \sum_{t=\tilde{\gamma}}^T r_t  }{ s_0 \left(\log d + \log T \right)}
 \leq  C (\sigma^2 \vee 1).$$
\end{corollary}
\begin{proof}
    The lower bound is established in Theorem \ref{thm: regret lower bound}, while the upper bound is from Theorem \ref{thm: regret over horizon}, which concerns the regret analysis of the proposed procedure. Note that $K=2$ is fixed, and for $(\btheta^{(1)}, \btheta^{(2)}) \in \widetilde{\bTheta}_{d}(s_0,L)$, Assumption \ref{assumption:arm_parameters} holds with $m_{\theta} = 2L$ and $L_4 = L$.
\end{proof}

Corollary \ref{cor:bandit_minimax_rate} characterizes the minimax rate of the overall cumulative regret up to a $\log(s_0)$ factor, with optimal dependence on $d$ and $T$. By Theorem \ref{thm: regret over horizon}, the proposed procedure achieves this rate up to a $\log(s_0)$ factor and is therefore minimax optimal when $s_0$ is constant. 

The extra $\log(s_0)$ factor in the upper bound arises from Stage 2 of the proposed algorithm, where Lasso estimators are applied to bridge the gap between the pure exploration phase and Stage 3, in which OPT-Lasso is used. Indeed, if the cumulative regret is computed starting from time $C^* s_0^5 \log(dT)$, for some sufficiently large constant $C^*$, thereby ignoring the short initial period, then Corollary~\ref{cor:bandit_minimax_rate} implies that the minimax rate is exactly of order $s_0(\log d + \log T)$, and the proposed procedure is minimax optimal.

\subsection{Suboptimality of Lasso in the Bandit Setting}\label{subsec:lasso_bandit}
In this subsection, we analyze the proposed algorithm without Stage 3 (see also Figure \ref{fig:proposed_three_stages}). Specifically, by setting the end time $\gamma_2$ of Stage 2 to $T$ in Definition \ref{def:three_stages}, the algorithm reduces to two stages. After the initial exploration phase from time $1$ to $\gamma_1$, we employ Lasso estimators to estimate the arm parameters and select arms greedily by maximizing the estimated rewards. The goal is to demonstrate the suboptimality of Lasso in the bandit setting.

For simplicity, we assume that condition \eqref{parameters_for_analysis} holds for $g_1$ and $\lambda_t^{(k)}$. We start with an upper bound on the cumulative regret.

\begin{corollary} 
Suppose Assumptions \ref{assumption: covariates_bandit} and \ref{assumption:arm_parameters} hold. 
Let $\gamma_2 = T$ in Definition \ref{def:three_stages}, that is, without Stage 3. 
Set the regularization parameter $\lambda_t$ and the end time $\gamma_1$ of Stage 1 as in \eqref{eq: bandit_paras_est} and \eqref{eq: end_times}, respectively. There exist constants $C,C^* > 0$ depending only on $K,  m_X, \alpha_X,  L_3, m_{\theta}, L_4$, such that
 if the integer constant $C_{\gamma_1}$ in \eqref{eq: end_times} satisfies  $C_{\gamma_1} \geq C^*$, then
$$
\Exp\left[\sum_{t=1}^T r_t \right] \le C C_{\gamma_1} (\sigma^2\vee 1)  s_0(\log d + \log T)\log(T).
$$ 
\end{corollary}
\begin{proof}
This is implied by the proof of Theorem \ref{thm: regret over horizon}; see Subsection \ref{sec: proof of thm: regret over horizon}.
\end{proof}

Compared with the minimax rate in Corollary \ref{cor:bandit_minimax_rate}, the above upper bound includes an additional $\log T$ factor. The following theorem shows that this bound is generally tight, implying that the use of standard Lasso estimators--—rather than OPT-Lasso—--is suboptimal in the bandit setting. Denote by $\mathbb{I}_{d \times d}$ the $d$-by-$d$ identity matrix.

\begin{theorem}\label{thm: bandit lower bound of lasso} Consider the case $K =2$ and set $\gamma_2 = T$ in Definition \ref{def:three_stages}, that is, without Stage 3. 
Suppose that Assumption  \ref{assumption: context} holds and that $\bSigma^{(1)} = \bSigma^{(2)} = 2^{-1} \mathbb{I}_{d \times d}$.  Set the regularization parameter $\lambda_t$ and the end time $\gamma_1$ of Stage 1 as follows:
\begin{align*}
\gamma_1 = C_{\gamma_1} \lceil(\sigma^2\vee 1)\rceil s_0 \lceil \log (dT) \rceil, \quad
\lambda_t = 48 m_X\sigma\sqrt{\frac{\log (dT)}{t}}.
\end{align*}
Let $\kappa >1$ be a constant. There exist  constants $C^*, C > 0$ depending only on $m_X$ such that if 
$
C_{\gamma_1} \geq C^*$ and $T  \geq (2C_{\gamma_1} C^*  s_0^5\log(dT) )^\kappa$,
then we have the following lower bound on the worst-case cumulative regret over $\widetilde{\bTheta}_{d}(s_0,L)$, excluding an initial period:
$$
\sup_{(\btheta^{(1)}, \btheta^{(2)}) \in \widetilde{\bTheta}_{d}(s_0,L)}  \Exp_{\btheta^{(1)}, \btheta^{(2)}}\left[\sum_{t=C_* C_{\gamma_1} s_0^5 \log(dT)}^T r_t \right] 
\ge C^{-1} \sigma^2 (1-1/\kappa)  \log(T)(\log d+\log T).$$
\end{theorem}
\begin{proof}
    See Appendix \ref{sec: proof of thm: bandit lower bound of lasso}. 
\end{proof}

\begin{remark}
As previously discussed, a \emph{lower} bound over fewer rounds immediately implies a lower bound on the overall regret.
\end{remark}

In Theorem \ref{thm: bandit lower bound of lasso}, we assume $\bSigma^{(1)} = \bSigma^{(2)} = 2^{-1}\mathbb{I}_{d \times d}$. As discussed following Assumption \ref{assumption: covariates_bandit}, this condition is satisfied by many examples. Under the assumptions of Theorem \ref{thm: bandit lower bound of lasso}, Corollary \ref{cor:bandit_minimax_rate} implies that the minimax rate for cumulative regret over the same time horizon is of order $s_0(\log d + \log T)$. In contrast, the lower bound in Theorem \ref{thm: bandit lower bound of lasso} includes an additional $\log T$ factor. This indicates that when using Lasso estimators for pure exploitation, the resulting algorithm is minimax \emph{suboptimal} in terms of both $d$ and $T$, in general. We also note that, due to technical challenges, the lower bound does not capture the explicit dependence on $s_0$.

\subsection{Simulation Study in the Bandit Setting} \label{sec: bandit simulation}
In this section, we conduct simulations to compare the empirical performance of the proposed procedure in Definition~\ref{def:three_stages} (denoted as ``Three-stage'') with the following alternatives:  
(i) the Lasso bandit algorithm \citep{bastani2020online};  
(ii) the procedure in Definition~\ref{def:three_stages} with $\gamma_2 = \gamma_1$, i.e., an algorithm that uses only OPT-Lasso estimators (denoted as ``Two-stage with OPT'');  
(iii) the procedure in Definition~\ref{def:three_stages} with $\gamma_2 = T$, i.e., an algorithm that uses only Lasso estimators (denoted as ``Two-stage with Lasso'').

We fix the time horizon to $T = 10000$ and consider eight scenarios with varying sparsity $s_0$, dimension $d$, and number of arms $K$. The scenarios for $(s_0, d, K)$ are: (a) $(5, 100, 5)$, (b) $(5, 100, 10)$, (c) $(10, 500, 5)$, (d) $(10, 500, 10)$, (e) $(5, 1000, 5)$, (f) $(5, 1000, 10)$, (g) $(10, 1000, 5)$, and (h) $(10, 1000, 10)$.
In each scenario, for each arm $k \in [K]$, we first randomly select a subset $S_k$ of $s_0$ covariates. Then, for each $j \notin S_k$, we set $\btheta^{(k)}_j = 0$, and for each $j \in S_k$, we sample $\btheta^{(k)}_j$ independently from the uniform distribution on $[0,1]$.  
Next, at each time $t \in [T]$, the covariate vector $\bX_t$ is independently drawn from the standard $d$-dimensional normal distribution $\mathcal{N}(\mathbf{0}_d, \mathbb{I}_{d \times d})$, with entries truncated to lie within $[-1, 1]$; that is, for each $j \in [d]$, if $(\bX_t)_j > 1$, we set $(\bX_t)_j = 1$, and if $(\bX_t)_j < -1$, we set $(\bX_t)_j = -1$.  
Finally, at each time $t \in [T]$, the noise $\bepsilon_t$ is independently drawn from the standard normal distribution $\mathcal{N}(0,1)$.  
We perform $1000$ repetitions for each scenario.

For the proposed algorithm (``Three-stage'') and the alternative algorithms (ii) ``Two-stage with OPT'' and (iii) ``Two-stage with Lasso,'' we set the regularization and threshold parameters in Definition \ref{def:three_stages} as follows: for each $t \in [T]$ and $k \in [K]$,
\[
\lambda_t^{(k)} = C_0 \hat{p}_t^{(k)} \sqrt{{ \log(d)}/{t}}, \quad
\lambda_t^{\mathrm{OPT},(k)} = C_0^{\mathrm{hard}} \left( C_0 \sqrt{{ \log(dt)}/{t}} \right),
\]
where  $\hat{p}_t^{(k)} := t^{-1} \sum_{s=1}^t \mathbf{1}\{A_s = k\}$ denotes the proportion of times the $k$-th arm has been pulled up to time $t$.
The end times of Stage 1 and 2 are set to $\gamma_1 = 10K$ and $\gamma_2 = 8\gamma_1$, respectively. Furthermore, $g_1$ and $g_2$ in Definition \ref{def:three_stages} are both set to $50$; that is, the arm parameter estimates are updated every $50$ rounds during Stages 2 and 3. Following Section 5.1 of \cite{bastani2020online}, for the Lasso bandit algorithm in \cite{bastani2020online}, we set $q = 1$, $h = 5$, and $\lambda_1 = \lambda_{2,0} = 0.025$\footnote{The Lasso estimator in \cite{bastani2020online} uses a slightly different constant than ours.}.


Table \ref{table: compare stages} presents the cumulative regret up to time $T$ for all scenarios, where $C_0 = 2$, $C_0^{\textup{hard}} = 0.6$ when $K = 5$, and $C_0 = 2$, $C_0^{\textup{hard}} = 1$ when $K = 10$. For scenarios (e) and (f), Figure \ref{fig: cum regret} shows the cumulative regret up to time $t$ for each $t \in [\gamma_2, T]$ for the proposed three-stage algorithm and the ``Two-stage with Lasso'' algorithm. The remaining scenarios are presented in Appendix \ref{app: more_simulations}. From Table \ref{table: compare stages}, we observe that the proposed three-stage procedure performs favorably against the alternatives, and that the impact of $K$ is more significant than that of $s_0$ or $d$. The Lasso bandit algorithm exhibits significantly higher cumulative regret than the others. Furthermore, although we introduce Stage 2 with Lasso estimators (see Definition \ref{def:three_stages}) for technical reasons, Table \ref{table: compare stages} shows that its inclusion slightly improves performance over the ``Two-stage with OPT'' algorithm. One possible explanation is that when the data is limited, as in Stage 2 under the bandit setup, Lasso outperforms OPT-Lasso. Finally, from both Table \ref{table: compare stages} and Figure \ref{fig: cum regret}, it is evident that using OPT-Lasso estimators in the bandit setup has much better empirical performance than using Lasso estimators. This aligns with our theoretical findings and with the observations from the sequential estimation problem in Section~\ref{sec: HSLR simulation}.


\begin{table}[!t]
\caption{The table shows the cumulative regret of scenarios (a)(c)(e)(g) with $C_0=2, C_0^{\textup{hard}}=0.6$ and (b)(d)(f)(h) with $C_0=2, C_0^{\textup{hard}}=1$, along with the standard deviations. } \label{table: compare stages}
\begin{tabular}{l|cccc}
 $(s_0,d,K)$   &Three-stage    & Two-stage with OPT & Two-stage with Lasso & Lasso bandit   \\ \hline
(a). $(5,100,5)$ &341.4 $\pm$ 1.6  & 367.5 $\pm$ 2.5     & 376.7 $\pm$ 1.7     & 2109.8 $\pm$ 11.0    \\
(b). $(5,100,10)$ &869.1 $\pm$ 2.8  & 924.5 $\pm$ 3.3     & 986.2 $\pm$ 3.2     & 4174.8 $\pm$ 12.6  \\
(c). $(10,500,5)$ &666.0 $\pm$ 2.2    & 692.5 $\pm$ 2.3     & 724.9 $\pm$ 2.3     & 6584.9 $\pm$ 32.1  \\
(d). $(10,500,10)$ &1644.5 $\pm$ 4.2 & 1728.8 $\pm$ 4.4    & 1881.2 $\pm$ 4.1    & 12967.9 $\pm$ 40.6 \\
(e). $(5,1000,5)$ &481.6 $\pm$ 2.1  & 516.2 $\pm$ 2.3     & 559.3 $\pm$ 2.4     & 7356.6 $\pm$ 30.4  \\
(f). $(5,1000,10)$ &1271.4 $\pm$ 5.5 & 1391.4 $\pm$ 5.1    & 1624.8 $\pm$ 4.7    & 12025.5 $\pm$ 64.4 \\
(g). $(10,1000,5)$ &744.8 $\pm$ 2.5 &761.4 $\pm$ 2.5 &825.2 $\pm$ 2.7 &9511.8 $\pm$ 46.2 \\
(h). $(10,1000,10)$ &1861.3 $\pm$ 4.7 &1951.6 $\pm$ 4.9 &2190.4 $\pm$ 5.3 &16414.6 $\pm$ 57.7  
\end{tabular}
\end{table}

 \begin{figure}[!t]
\centering
\includegraphics[width=0.93\textwidth]{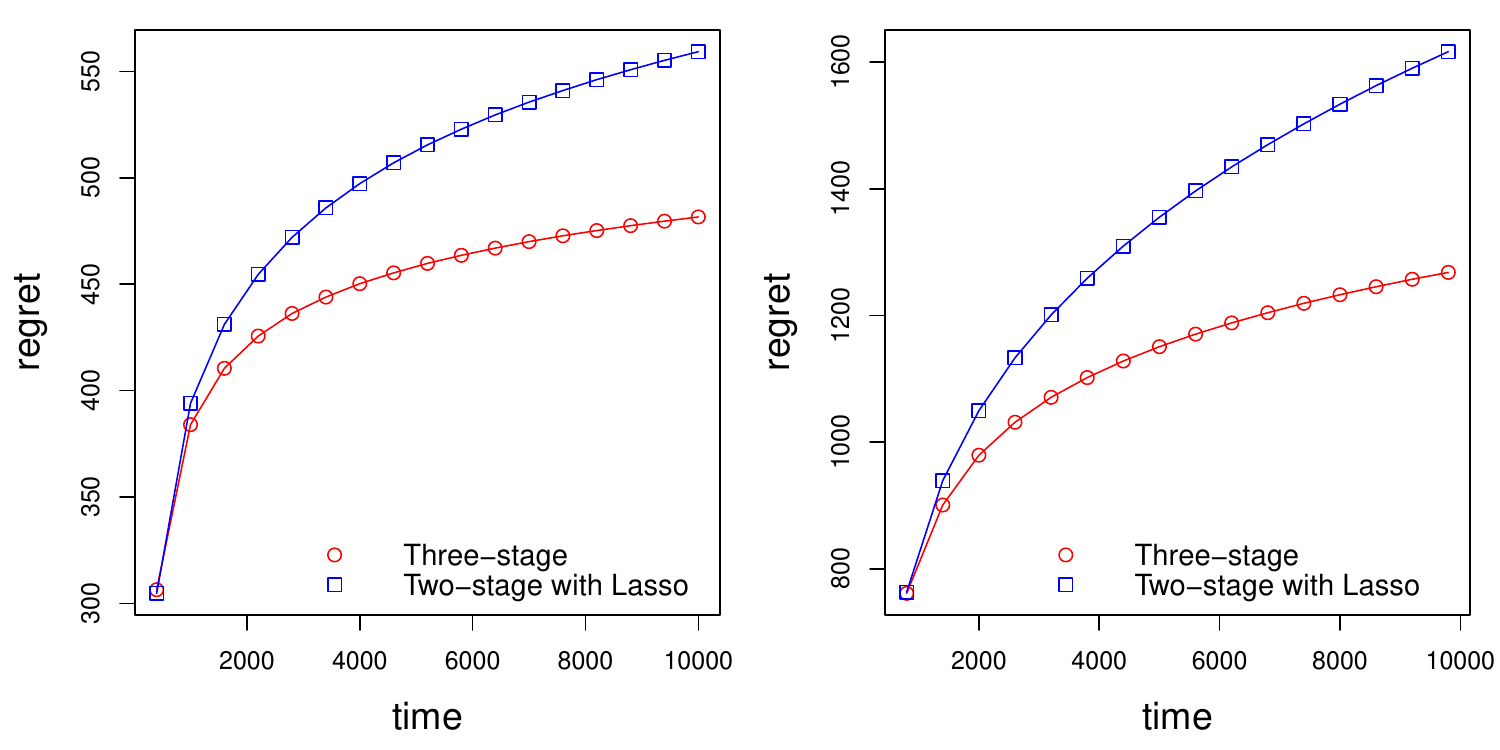}
\caption{The $y$-axis represents the cumulative regret up to time $t \in [\gamma_2,T]$. The left plot is for scenario (e) with $C_0=2, C_0^{\textup{hard}}=0.6, \gamma_2 = 400$, while the right plot is for  scenario (f) with $C_0=2, C_0^{\textup{hard}}=1$, $\gamma_2 =800$.} \label{fig: cum regret}
\end{figure}

In Figure \ref{fig: bandit support}, we compare the support recovery performance of the OPT-Lasso and Lasso estimators under the bandit setup for scenario (e), with $C_0 = 2$ and $C_0^{\textup{hard}} = 0.6$. Specifically, for each $t \in [T/2]$, we report the number of false positives (incorrectly selecting covariates not in the true support) and false negatives (failing to select covariates in the true support),  averaged over the $K$ arms. We observe similar patterns to those in the simulation results for sequential estimation in Section \ref{sec: HSLR simulation}; that is, the better performance of the OPT-Lasso estimator compared to Lasso is largely due to its significantly lower number of false positives.

Finally, we examine the sensitivity of the proposed three-stage algorithm to the tuning parameters $C_0$ and $C_0^{\textup{hard}}$. Table \ref{table: bandit sensitivity} shows the cumulative regret under scenarios (e) and (f); additional results are in Appendix \ref{app: more_simulations}. The results indicate that developing principled methods for tuning parameter selection remains an important direction for future work.


\begin{figure}[!t]
\centering
\includegraphics[width=0.93\textwidth]{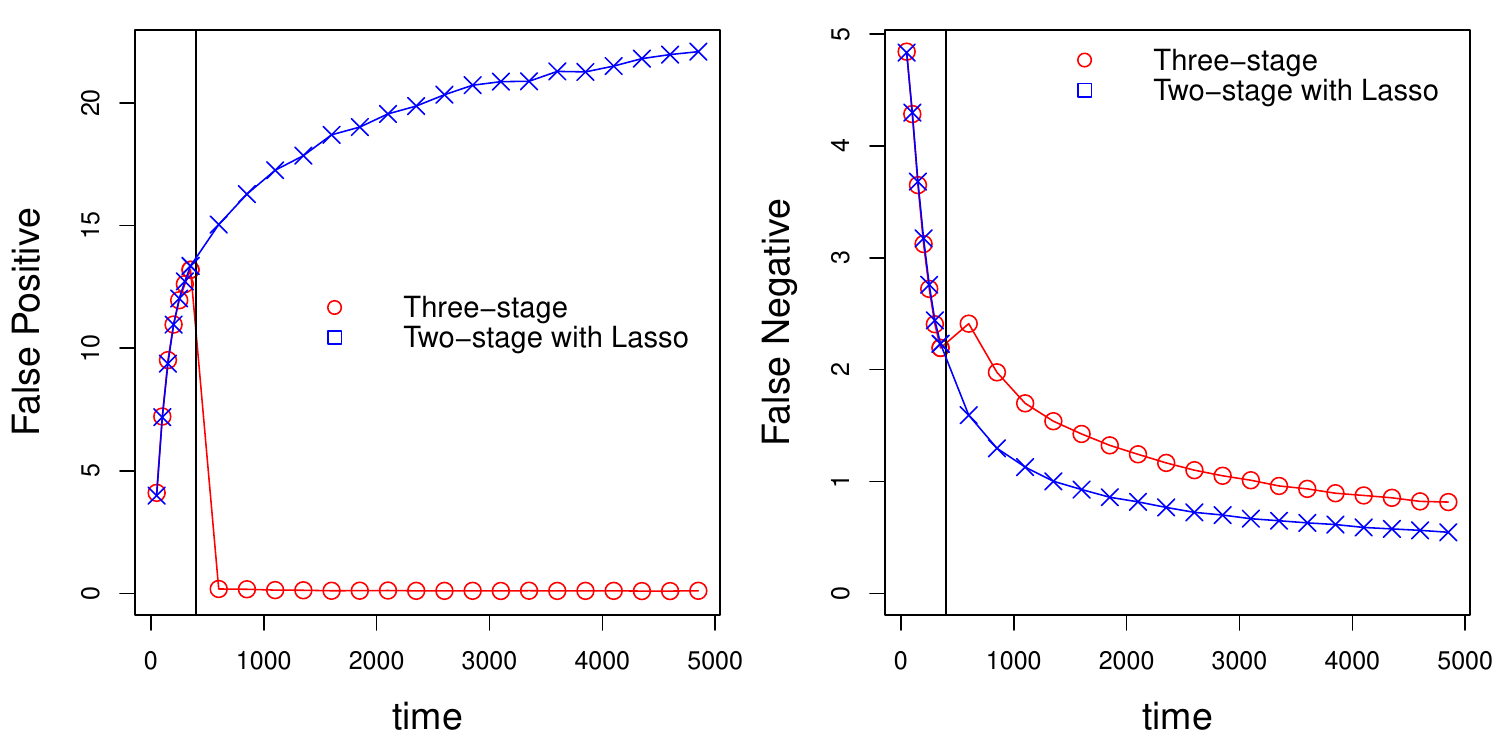}
\caption{We consider scenario (e) and (f) with $C_0=2, C_0^{\textup{hard}}=0.6$. The left plot shows the number of false positives, while the right plot shows the number of false negatives at each time $t \in [T/2]$,  averaged across $K$ arms.  The trend after time $T/2$ remains similar and is therefore not shown.} \label{fig: bandit support}
\end{figure}

\begin{table}[!t]
\caption{Cumulative regret for scenarios (e) and (f). Bold indicates the lowest cumulative regret.} \label{table: bandit sensitivity}
\centering
$(e)$
\begin{tabular}{c|ccccc}
$C_0\;\backslash \; C_0^{\textup{hard}}$ & 0.2          & 0.6          & 1.0           & 1.5           & 2             \\ \hline
1.0   & 2208.7 $\pm$ 5.1 & 1374.5 $\pm$ 4.2 & 784.5 $\pm$ 2.8 & 557.5 $\pm$ 2.2   & 544.1 $\pm$ 2.1   \\
1.6 & 1166.6 $\pm$ 3.8 & 539.5 $\pm$ 2.3  & 495.0 $\pm$ 3.0     & 599.3 $\pm$ 4.1   & 790.8 $\pm$ 6.3   \\
2.0   & 734.1 $\pm$ 3.1  & \textbf{481.6 $\pm$ 2.1}  & 561.3 $\pm$ 3.7 & 766.1 $\pm$ 6.3   & 1098.8 $\pm$ 10.8 \\
2.6 & 505.8 $\pm$ 2.3  & 553.0 $\pm$ 2.8    & 761.8 $\pm$ 4.5 & 1226.2 $\pm$ 9.9  & 1873.5 $\pm$ 12.2 \\
3.0   & 497.0 $\pm$ 2.3    & 642.7 $\pm$ 3.7  & 983.5 $\pm$ 7.0   & 1608.4 $\pm$ 10.7 & 2456.4 $\pm$ 15.4  
\end{tabular} 

\medskip

$(f)$
\begin{tabular}{c|ccccc}
\hline
$C_0\;\backslash \; C_0^{\textup{hard}}$ & 0.2          & 0.6          & 1.0            & 1.5           & 2            \\ \hline
1.0   & 4695.3 $\pm$ 8.2 & 4023.5 $\pm$ 7.0   & 2976.4 $\pm$ 5.3 & 1932.0 $\pm$ 4.3    & 1481.4 $\pm$ 3.6  \\
1.6 & 3593.9 $\pm$ 7.4 & 2211.4 $\pm$ 5.2 & 1456.3 $\pm$ 5.9 & 1301.8 $\pm$ 5.8  & 1425.4 $\pm$ 7.8  \\
2.0   & 2849.1 $\pm$ 5.8 & 1587.5 $\pm$ 4.4 & \textbf{1271.4 $\pm$ 5.5} & 1413.6 $\pm$ 7.8  & 1721.7 $\pm$ 11.1 \\
2.6 & 1990.5 $\pm$ 5.2 & \textbf{1273.4 $\pm$ 4.4} & 1350.3 $\pm$ 5.1 & 1776.4 $\pm$ 9.1  & 2416.4 $\pm$ 14.5 \\
3.0   & 1639.0 $\pm$ 4.8   & 1257.0 $\pm$ 4.4   & 1516.9 $\pm$ 7.0   & 2145.7 $\pm$ 12.3 & 3063.0 $\pm$ 19.5     
\end{tabular}
\end{table}

\subsection{Proof of Theorem \ref{thm: regret over horizon}: upper bound for the proposed procedure} \label{sec: proof of thm: regret over horizon}

In this subsection, we prove Theorem \ref{thm: regret over horizon}, which provides an upper bound on the cumulative regret of the procedure in Definition \ref{def:three_stages}. A key step is bounding the $\ell_2$ estimation error of the Lasso estimators in Stage 2 and the OPT-Lasso estimators in Stage 3, as stated in the following theorem.

\begin{theorem}\label{thm:bandit_est_accuracy}
 Suppose Assumptions \ref{assumption: covariates_bandit} and  \ref{assumption:arm_parameters} hold. Set $\lambda_t, \hdlambda_t$ as in \eqref{eq: bandit_paras_est}, and $\gamma_1,\gamma_2$ as in \eqref{eq: end_times}. 
 There exist constants $C,C^* > 0$ depending only on $K,  m_X, \alpha_X,  L_3, m_{\theta}, L_4$, such that
 if the integer constants in \eqref{eq: end_times} satisfy  $C_{\gamma_1} \geq C^*$ and $C_{\gamma_2} \geq C^*$, then
\begin{enumerate}[label=\roman*)]
\item (Stage 2) For $t = m\gamma_1$ with $m \in \{1,2, \ldots, \gamma_2/\gamma_1\}$, we have
$$
\Pro\left( \mathcal{E}_{t} \right) \geq 1- \frac{6K}{T}, \;\; \text{ where }\; \mathcal{E}_{t} := \bigcap_{k=1}^{K} \left\{\|\hbtheta_t^{(k)}-\btheta^{(k)}\|_2^2\le  \frac{C \sigma^2 s_0\log (dT)}{t}\right\}.
$$

\item (Stage 3) For $t = m\gamma_2$ with $m \in \{1,2,\ldots, \lfloor T/\gamma_2 \rfloor\}$, there exists an event $\mathcal{A}_t \in \cF_t$ such that $\Pro(\mathcal{A}_t)\ge 1-17K/T$, and for each arm $k \in [K]$,
\begin{equation*}
    \Exp\left[\|\tbtheta_t^{(k)}-\btheta^{(k)} \|_2^2 \idf{\{\mathcal{A}_t\}}\right] \le C\sum_{j\in S_k} (\btheta^{(k)}_j)^2 \idf\left\{|\btheta^{(k)}_j| \le 2\hdlambda_t \right\} + C\sigma^2 s_0/t.
\end{equation*}
\end{enumerate}   
\end{theorem}

Part (i) of Theorem \ref{thm:bandit_est_accuracy} follows directly from Lemma \ref{lemma: RE and estimation error of stage 2} in Appendix \ref{sec: proof of thm: regret in stage 2}, and Part (ii) is established in Lemma \ref{stage3: main lemma} in Appendix \ref{sec: proof of thm: regret in stage 3}. An overview of the proof strategy is provided in Appendix \ref{subsec:discussion_bandit_est}, where we also elaborate on the motivation for introducing Stage 2. Note that, as with \eqref{eq: opt_lasso_t}, the upper bound in part (ii) is instance-specific.

Given the upper bounds on the $\ell_2$ estimation errors, we use the assumption that the covariate vectors have a log-concave density to convert these bounds into bounds on the instantaneous regret (see Lemma \ref{lemma: estErr_to_regret} in Appendix). The key observation is that for  $t \in [T]$, $\hbtheta_t^{(k)} \in \cF_t$ and 
$\tbtheta_t^{(k)} \in \cF_t$ for $k \in [K]$, while $\bX_{t+1}$ is independent from $\cF_{t}$.

\begin{proof}[Proof of Theorem \ref{thm: regret over horizon}]
We choose $C^* > 0$ as in Theorem \ref{thm:bandit_est_accuracy}. 
In this proof, $C$ denotes a constant, that only depends on $K,  m_X, \alpha_X,  L_3, m_{\theta}, L_4$ and may vary from line to line. Recall the definition of $r_t$ in \eqref{eq: def of regret}.

\medskip
\noindent \underline{\textbf{Stage 1.}} Let $t \in [1,\gamma_1]$. By definition, $r_t \leq \sum_{k \in [K]} |\bX_t'\btheta^{(k)}|$. For each arm $k \in [K]$, 
due to Assumption \ref{assumption: bounded cov mat} and \ref{assumption: theta_bound}, we have 
\begin{align}\label{trival_bound}
\Exp\left[|\bX_t'\btheta^{(k)}|\right] \leq  \sqrt{ (\btheta^{(k)})'\bSigma \btheta^{(k)} } \leq m_{\theta} \sqrt{L_3},
\end{align}
which, together with the choice of $\gamma_1$ in \eqref{eq: end_times}, implies  $\sum_{t=1}^{\gamma_1} \Exp[r_t] \leq C C_{\gamma_1}(\sigma^2\vee 1) s_0\log(dT)$.

\medskip
\noindent \underline{\textbf{Stage 2.}} Let $t \in (\gamma_1,\gamma_2]$. For some $m \in \{1,\ldots, \gamma_2/\gamma_1\}$, we have $t \in (m\gamma_1,(m+1)\gamma_1]$. Let
\begin{align*}
    \bu_{i,j} := \btheta^{(i)} - \btheta^{(j)}, \quad
    \hat{\bu}_{i,j} := \hbtheta^{(i)}_{m\gamma_1} - \hbtheta^{(j)}_{m \gamma_1}, \quad \text{ for } i,j \in [K].
\end{align*}
Due to the definition of $r_t$ in \eqref{eq: def of regret} and $A_t$ in Definition \ref{def:three_stages}, we have
\begin{align*}
    r_t &= |\bu_{A_t^*,A_t}'\bX_t| \idf\{\textup{sgn}(\bu_{A_t^*,A_t}'\bX_t)\neq \textup{sgn}(\hbu_{A_t^*,A_t}'\bX_t)\},
\end{align*}
where for $\tau \in \bR$, $\textup{sgn}(\tau) := \idf\{\tau > 0\} - \idf\{\tau < 0\}$. This is because if $r_t = 0$, both sides equal zero; while if $r_t > 0$, then $\bu_{A_t^*,A_t}'\bX_t > 0$ and $\hbu_{A_t^*,A_t}'\bX_t \leq 0$. Therefore, we have
\begin{align*}
    r_t\leq \sum_{i \neq j \in [K]}
    |\bu_{i,j}'\bX_t| \idf\{\textup{sgn}(\bu_{i,j}'\bX_t)\neq \textup{sgn}(\hbu_{i,j}'\bX_t)\}.
\end{align*}

Since $\hat{\bu}_{i,j} \in \cF_{m\gamma_1}$  and $\bX_t$ is independent from $\cF_{m\gamma_1}$, due to Assumption \ref{assumption: context} and by applying Lemma \ref{lemma: estErr_to_regret}  conditional on $\cF_{m\gamma_1}$, we have that for $i\neq j \in [K]$, almost surely (a.s.),
\begin{align*}
    \Exp\left[|\bu_{i,j}'\bX_t| \idf\{\textup{sgn}(\bu_{i,j}'\bX_t)\neq \textup{sgn}(\hbu_{i,j}'\bX_t)\}\vert \cF_{m\gamma_1}\right] \leq C  {\|\bu_{i,j} -\hbu_{i,j}\|_2^{2}}/{\|\bu_{i,j}\|_2}.
\end{align*}
Due to Assumption \ref{assumption: linear_independent}, $\|\bu_{i,j}\|_2 \geq L_4^{-1}$ for $i \neq j \in [K]$. By the triangle inequality, a.s.,
\begin{align*}
\Exp\left[r_t \vert \cF_{m\gamma_1}\right]
\leq C \sum_{k \in [K]} \|\hbtheta^{(k)}_{m\gamma_1} - \btheta^{(k)}\|_2^2.
\end{align*}
Recall the event $\mathcal{E}_{m\gamma_1} \in \cF_{m\gamma_1}$ in Theorem \ref{thm:bandit_est_accuracy}. By definition, a.s.,
\begin{align*}
  \Exp[r_t \idf\{\mathcal{E}_{m\gamma_1}\} \vert \cF_{m\gamma_1}] \leq  C \sum_{k \in [K]} \|\hbtheta^{(k)}_{m\gamma_1} - \btheta^{(k)}\|_2^2 \idf\{\mathcal{E}_{m\gamma_1}\} \leq C \sigma^2 \frac{s_0\log(dT)}{m\gamma_1}.
\end{align*}
Further, on the event $\mathcal{E}_{m\gamma_1}^c$, we use the bound: $r_t \leq \sum_{k \in [K]} |\bX_t'\btheta^{(k)}|$. Since $X_t$ is independent from $\cF_{m\gamma_1}$, a.s.,
\begin{align*}
 \Exp[r_t \idf\{\mathcal{E}_{m\gamma_1}^c\} \vert \cF_{m\gamma_1}] =   \sum_{k \in [K]}\Exp\left[  |\bX_t'\btheta^{(k)}|\right] \idf\{\mathcal{E}_{m\gamma_1}^c\} \leq C \idf\{\mathcal{E}_{m\gamma_1}^c\},
\end{align*}
where we use \eqref{trival_bound} in the last step. 
Due to Theorem \ref{thm:bandit_est_accuracy}(i) and by the law of total expectation,
\begin{align*}
    \Exp\left[r_t\right] 
    \leq  C \sigma^2 \frac{s_0\log(dT)}{m\gamma_1} + C \frac{1}{T},
\end{align*}
which implies the following upper bound on the cumulative regret in Stage 2:
\begin{align*}
  \sum_{t = \gamma_1 + 1}^{\gamma_2} \Exp[r_t]  
  \leq C \sum_{m=1}^{\gamma_2/\gamma_1} \gamma_1\left(\sigma^2 \frac{s_0\log(dT)}{m\gamma_1} +  \frac{1}{T}\right)
  \leq C (\sigma^2 \vee 1)\log(\gamma_2/\gamma_1) s_0\log(dT).
\end{align*}
Due to the choice of $\gamma_1$ and $\gamma_2$ in \eqref{eq: end_times}, $\sum_{t = \gamma_1 + 1}^{\gamma_2} \Exp[r_t] \leq C (\sigma^2 \vee 1) \log(s_0) s_0\log(dT)$.

\medskip
\noindent \underline{\textbf{Stage 3.}} 
Let $t \in (\gamma_2,T]$. Then for some $m \in \{1,\ldots, \lfloor T/\gamma_2\rfloor\}$, we have $t \in (m\gamma_2,(m+1)\gamma_2]$.   Recall the event $\mathcal{A}_{m\gamma_2}$ in Theorem \ref{thm:bandit_est_accuracy}(ii). 
By similar arguments as for Stage 2,  
\begin{align*}
    \Exp[r_t]  \le C\sum_{k\in[K]}\left( \Exp\left[\|\tbtheta_{m\gamma_2}^{(k)}-\btheta^{(k)}\|_2^2\idf{\{\mathcal{A}_{m\gamma_2}\}}\right]\right)  +  C \Pro(\cA_{m\gamma_2}^c).
\end{align*}
Then due to Theorem \ref{thm:bandit_est_accuracy}(ii),
\begin{align*}
    \Exp[r_t]  \le  C\sum_{k\in[K]}\left(\sum_{j\in S_k} (\btheta^{(k)}_j)^2 \idf\left\{|\btheta^{(k)}_j| \le 2\hdlambda_{m\gamma_2} \right\} + \frac{\sigma^2 s_0}{m\gamma_2}   \right) + \frac{C}{T} .
\end{align*}
It is elementary to see that
$\sum_{m=1}^{\lfloor T/\gamma_2\rfloor} \gamma_2 \left( \frac{\sigma^2 s_0}{m\gamma_2} + \frac{1}{T} \right) \leq C (\sigma^2 \vee 1) s_0 \log T$. 
Further, due to Lemma \ref{lemma: instance sum} and the choice of $\hdlambda_t$ in \eqref{eq: bandit_paras_est}, we have
$$
 \sum_{k\in[K]} \sum_{m=1}^{\lfloor T/\gamma_2\rfloor} \gamma_2\sum_{j\in S_k} (\btheta^{(k)}_j)^2 \idf\left\{|\btheta^{(k)}_j| \le 2\hdlambda_{m\gamma_2}\right\} \leq C \sigma^2 s_0 \log(dT).
$$
Thus, $\sum_{t = \gamma_2+1}^{T} \Exp[r_t] \leq C (\sigma^2 \vee 1) s_0\log(dT)$. The proof is then complete by combining the upper bounds on the cumulative regret from each stage.
\end{proof}

\section{Conclusion} \label{sec:conclusion}
We study the problem of high-dimensional stochastic linear contextual bandits under sparsity constraints. To isolate the effect of estimation methods, we first analyze a related sequential estimation problem in which no arm selection is involved. We show that while the Lasso estimator is minimax optimal in the \textit{fixed-time} setting, it is suboptimal in terms of cumulative estimation error in the sequential setting. In contrast, OPT-Lasso estimators attain the minimax rate in the sequential regime. Building on these insights, we propose a three-stage algorithm for the contextual bandit problem that primarily relies on OPT-Lasso estimators. We derive upper bounds on its cumulative regret and matching lower bounds over all permissible policies, thereby showing that the algorithm is nearly minimax optimal---up to a $\log(s_0)$ factor---over the full horizon, and exactly minimax optimal if an initial short transient phase is excluded.

Several directions remain for future investigation. First, it is an open question whether the $\log s_0$ gap in the overall regret bound can be eliminated. Second, developing principled methods for tuning parameter selection in the proposed three-stage algorithm is of practical importance. Third, extending the framework to nonlinear models with high-dimensional covariates is an interesting direction for future research.

\bibliographystyle{imsart-number} 
\bibliography{ref}       

\newpage
\begin{appendix}

\section{Lasso and OPT-Lasso: Deterministic Analysis}
\label{app:lasso_deterministic_analysis}
In this Appendix, we consider the following \textit{deterministic} linear regression model 
$$\bY = \bZ\btheta +\bepsilon,$$ 
where $\bZ \in \mR^{n\times d}$ is a deterministic design matrix, $\bepsilon\in \mR^n$ is a deterministic noise vector, and $\btheta = (\btheta_1,\ldots,\btheta_d)' \in \bR^{d}$ is an unknown vector. We review several well-known results for Lasso estimators and derive new results for OPT-Lasso estimators. Specifically, denote by 
$$\ltheta_n:=\ltheta_n(\bZ,\bY,\lambda_n),\quad
\text{ and } \quad
\ptheta_n := \ptheta_n(\bZ,\bY,\lambda_n,\hdlambda_n)
$$ 
the Lasso estimator (see Definition \ref{def: lasso}) and the OPT-Lasso estimator (see Definition \ref{def: OPT-Lasso}) respectively,
where $\lambda_n, \hdlambda_n > 0$ are to be specified.

Denote by $S := \{j \in [d]: \btheta_j \neq 0\}$ the support of $\btheta$ and let $s_0 = |S| \vee 1$. Further, denote by $\hbSigma_n = n^{-1}\bZ'\bZ \in \bR^{d\times d}$ the sample covariance matrix. Note that   $\bZ_S'\bZ_S/n = (\hbSigma_n)_{S,S}$.

\begin{definition} \label{def: RE}
A matrix $\bZ\in \mR^{n\times d}$ is said to satisfy the restricted eigenvalue condition $\textup{RE}(s,\kappa,\rho)$ for some $s \in [d]$ and $\kappa, \rho > 0$ if 
$$n^{-1}\|\bZ \bv\|_2^2 \ge \rho\|\bv\|_2^2, \quad \text{ for all } \bv\in \mathcal{C}(s,\kappa),
$$
where we define
$$
\mathcal{C}(s,\kappa):=\{\bv\in \mR^d \backslash \{\bd{0}_d\}: \text{there exists } J\subseteq [d] \text{ such that } |J| = s, \text{ and } \|\bv_{J^c}\|_1 \le \kappa\|\bv_J\|_1\}.
$$
\end{definition}

\begin{remark} 
The above definition is adapted from \cite[Definition 7.12]{wainwright2009sharp}.
\end{remark}

The following $\ell_1, \ell_2$, and $\ell_{\infty}$ bounds on the estimation error of the Lasso estimator $\ltheta_n$ are well-known and are provided here for convenient reference.

\begin{lemma} \label{lemma: bounds of lasso}
Let $a >0$ and $\bA\in \mR^{d\times d}$ be any invertible matrix. Assume that $\bZ$ satisfies the $\textup{RE}(s_0, 3, a)$ condition, and that 
$\lambda_n \ge 2\|\bZ'\bepsilon/n\|_{\infty}$. Then
\begin{align*}
&\|\ltheta_n-\btheta\|_1\le {8s_0\lambda_n}/{a}, \\
&\|\ltheta_n-\btheta\|_2 \le {3\sqrt{s_0}\lambda_n}/{a},\\
& \|\ltheta_n-\btheta\|_{\infty} \le \mn\bA^{-1}\mn_1 \left(4 + 8\mn\hbSigma_n-\bA\mn_{\textup{max}} s_0/a\right)\lambda_n.
\end{align*}
\end{lemma}

\begin{proof}
The $\ell_1$ bound follows from Theorem 6.1 in \cite{buhlmann2011statistics}, while the $\ell_2$ bound from Theorem 7.13 (a) in \cite{wainwright2019high}. The $\ell_{\infty}$ bound
is derived from Lemma 4.1 in \cite{van2016estimation}, combined with the $\ell_1$ bound above.
\end{proof}

Recall the definition of $\pS_n:= \{j \in [d]: |(\ltheta_n)_j| > \hdlambda_n\}$ in Definition \ref{def: OPT-Lasso}. Define the following set
$$S_n^{\textup{str}}:=\{j\in [d]: |\btheta_j| > 2\hdlambda_n\},$$
which contains the indices of components of $\btheta$ whose magnitudes exceed $2\hdlambda_n$.


\begin{lemma}
\label{lemma: threshold Lasso_ell_infty bound bandit}
If $\|\ltheta_n-\btheta\|_{\infty} \le \hdlambda_n$, then $S_n^{\textup{str}}\subseteq \hdS_n\subseteq S$.
\end{lemma}
\begin{proof}
Let $j \in  S_n^{\textup{str}}$ and $k \in \hdS_n$. Since $\|\ltheta_n-\btheta\|_{\infty} \le \hdlambda_n$, by triangle inequality and definition, we have
$$
|(\ltheta_n)_j| \geq |\btheta_j| - \hdlambda_n > \hdlambda_n, \qquad
|\btheta_k| \geq  |(\ltheta_n)_k| - \hdlambda_n > 0,
$$
which implies $j \in \hdS_n$ and $k \in S$. The proof is complete.
\end{proof}


Next, we provide an instance-specific upper bound on the $\ell_2$ estimation error of the OPT-Lasso estimator $\ptheta_n$.

\begin{theorem} \label{thm: ols post l2 estimation error}
Assume $\|\ltheta_n-\btheta\|_{\infty}\le \hdlambda_n$.
\begin{enumerate}[label=\roman*)]
    \item If the support $S$ of $\btheta$ is an empty set, then
$\|\ptheta_n-\btheta \|_2 = 0$;
\item If the support $S$ of $\btheta$ is not an empty set, assume further that for some $0< a<b$,
$$a\le \lmin{(\hbSigma_n)_{S,S}} \le \lmax{(\hbSigma_n)_{S,S}} \le b.
$$
Then we have 
$$\|\ptheta_n-\btheta \|_2^2 \le (2b/a+1)\sum_{j\in S} \btheta_j^2 \idf\left\{|\btheta_j| \le 2\hdlambda_n \right\} + (2/a^2)\|\bZ_S'\bepsilon/n\|_2^2.$$
\end{enumerate}
\end{theorem}

\begin{proof}
By Lemma \ref{lemma: threshold Lasso_ell_infty bound bandit}, we have $    S_n^{\textup{str}}\subseteq \hdS_n\subseteq S$. By definition, we have
\begin{equation}\label{eq: proof of estimation error in stage 3 3}
    \left\|\ptheta_n-\btheta \right\|_2^2  = \left\|(\ptheta_n-\btheta)_S \right\|_2^2 = \left\|\btheta_{S\backslash\hdS_n} \right\|_2^2+ \left\|(\ptheta_n-\btheta)_{\hdS_n} \right\|_2^2
\end{equation}

If $S$ is an empty set, then $\ptheta_n = \btheta = \boldsymbol{0}_d$.  

If $S$ is not empty and $\hdS_n$ is an empty set, then $S_n^{\textup{str}}$ must be empty, which implies that $|\btheta_j| \leq 2 \hdlambda_n$ for all $j \in [d]$. Thus by \eqref{eq: proof of estimation error in stage 3 3}, we have
\begin{equation*}
    \left\|\ptheta_n-\btheta \right\|_2^2  = \|\btheta_{S}\|_2^2 = \sum_{j\in S} \btheta_j^2 \idf\left\{\left|\btheta_j\right| \le 2\hdlambda_n \right\}.
\end{equation*}

Now, we focus on the case that neither $S$ nor $\hdS_n$ is empty, and we bound the two terms on the right-hand side of \eqref{eq: proof of estimation error in stage 3 3}. For the first term in \eqref{eq: proof of estimation error in stage 3 3}, by definition, we have 
\begin{equation}\label{eq: proof of estimation error in stage 3 4}
    \|\btheta_{S\backslash\hdS_n}\|_2^2 \le \|\btheta_{S\backslash S_n^{\textup{str}}}\|_2^2 = \sum_{j\in S} \btheta_j^2 \idf\left\{\left|\btheta_j\right| \le 2\hdlambda_n \right\}.
\end{equation}
Next, we bound the second term in \eqref{eq: proof of estimation error in stage 3 3}. To simplify notations, denote $\pS_n$ by $\hat{S}_n$. Since
$\hat{S}_n \subseteq S$, we have
$$\lmin{(\hbSigma_n)_{\hat{S}_n,\hat{S}_n}}\ge \lmin{(\hbSigma_n)_{S,S}} \geq a > 0,$$
which in particular implies that $\bZ_{\hat{S}_n}' \bZ_{\hat{S}_n}$ is invertible. By the definition of OPT-Lasso (see Definition \ref{def: OPT-Lasso}),  we have
\begin{align*}
    (\ptheta_n)_{\hat{S}_n} &= \left(\bZ_{\hat{S}_n}' \bZ_{\hat{S}_n} \right)^{-1}\bZ_{\hat{S}_n}'\bY
    = \left(\bZ_{\hat{S}_n}' \bZ_{\hat{S}_n} \right)^{-1}\bZ_{\hat{S}_n}'\left(\bZ_{\hat{S}_n}\btheta_{\hat{S}_n} + \bZ_{S\backslash\hat{S}_n}\btheta_{S\backslash\hat{S}_n} + \bepsilon \right)
    \\&= \btheta_{\hat{S}_n} + \left(\bZ_{\hat{S}_n}' \bZ_{\hat{S}_n} \right)^{-1}\bZ_{\hat{S}_n}'\left( \bZ_{S\backslash\hat{S}_n}\btheta_{S\backslash\hat{S}_n} + \bepsilon \right),
\end{align*}
which implies that
\begin{equation} 
\label{eq: proof of estimation error in stage 3 1}
    \left\|\bZ_{\hat{S}_n}\left((\hbtheta_n-\btheta)_{\hat{S}_n} \right)\right\|_2^2 = \left\|\mathcal{P}_{\hat{S}_n} \left( \bZ_{S\backslash\hat{S}_n}\btheta_{S\backslash
    \hat{S}_n} + \bepsilon \right)\right\|_2^2,
\end{equation}
where for a non-empty set $A \subseteq [S]$, we define
$\mathcal{P}_{A}:=\bZ_A\left(\bZ_A'\bZ_A \right)^{-1}\bZ_A'$.

For the left-hand side of \eqref{eq: proof of estimation error in stage 3 1}, we have
\begin{equation}\label{eq: proof of estimation error in stage 3 5}
\begin{split}
\left\|\bZ_{\hat{S}_n}\left((\hbtheta_n-\btheta)_{\hat{S}_n} \right)\right\|_2^2 
&\geq    n\lmin{(\hbSigma_n)_{\hat{S}_n,\hat{S}_n}}\cdot \left\|(\ptheta_n-\btheta)_{\hat{S}_n} \right\|_2^2 \\
    &\geq n a \left\|(\ptheta_n-\btheta)_{\hat{S}_n} \right\|_2^2.
\end{split}
\end{equation}

For the right-hand side of \eqref{eq: proof of estimation error in stage 3 1}, we first note that $\mathcal{P}_A$ is the projection matrix onto the column space of $Z_A$.
Since $\hat{S}_n\subseteq S$,  we have 
\begin{equation}\label{eq: proof of estimation error in stage 3 7}
\begin{aligned}
    &\left\|\mathcal{P}_{\hat{S}_n}\left( \bZ_{S\backslash\hat{S}_n}\btheta_{S\backslash\hat{S}_n} + \bepsilon \right)\right\|_2^2 \le 2\left\|\mathcal{P}_{\hat{S}_n}\bZ_{S\backslash\hat{S}_n}\btheta_{S\backslash\hat{S}_n}\right\|_2^2 + 2\left\| \mathcal{P}_{\hat{S}_n}\bepsilon\right\|_2^2
    \\&\le 2\left\|\bZ_{S\backslash\hat{S}_n}\btheta_{S\backslash\hat{S}_n}\right\|_2^2 +2 \left\| \mathcal{P}_{S}\bepsilon\right\|_2^2
    \\&\le 2n\lmax{(\hbSigma_n)_{S\backslash\hat{S}_n, S\backslash\hat{S}_n}} \left\|\btheta_{S\backslash\hat{S}_n} \right\|_2^2 + 2n[\lmin{(\hbSigma_n)_{S,S}}]^{-1}\|\bZ_S'\bepsilon/n\|_2^2.
\end{aligned}    
\end{equation}
where the first inequality is due to the triangle inequality, the second inequality follows from the properties of projection matrices, and the final inequality follows from the definition of the eigenvalues.

Note that if $S \backslash \hat{S}_n$ is empty, the first term above is understood to be zero by convention. If  $S \backslash \hat{S}_n$ is empty, 
we have
$\lmax{(\hbSigma_n)_{S\backslash\hat{S}_n, S\backslash\hat{S}_n}}\le \lmax{(\hbSigma_n)_{S,S}} \leq b$. 
Therefore by \eqref{eq: proof of estimation error in stage 3 1}, \eqref{eq: proof of estimation error in stage 3 5}, and \eqref{eq: proof of estimation error in stage 3 7}, we have 
\begin{equation*}
\begin{aligned}
    \left\|(\ptheta_n-\btheta)_{\hat{S}_n} \right\|_2^2 &\le \frac{2b}{a} \left\|\btheta_{S\backslash\hat{S}_n} \right\|_2^2 + \frac{2}{a^2}\|\bZ_S'\bepsilon/n\|_2^2.
\end{aligned}
\end{equation*}
Then the proof is complete due to \eqref{eq: proof of estimation error in stage 3 4} and \eqref{eq: proof of estimation error in stage 3 3}.
\end{proof}

\section{Lower Bounds on Bayes Risks}
In this appendix, we consider the following linear model $\bY = \bZ\btheta +\bepsilon$ within the Bayesian framework, where $\bZ = (\bZ_1,\cdots,\bZ_n)'\in \mR^{n\times d}$ is a random design matrix with i.i.d. rows, $\bepsilon = (\epsilon_1, \cdots, \epsilon_n)'\in \mR^n$ is a noise vector with i.i.d. entries distributed as $\mathcal{N}(0,\sigma^2)$, and $\btheta = (\btheta_1,\cdots, \btheta_d) \in \mathbb{R}^{d}$ is a random vector with some prior distribution to be specified.

Our goal is to derive lower bounds on the Bayes risks for estimating certain functionals of $\btheta$ based on  $\bY$ and $\bZ$, under \textit{two} types of prior distributions on $\btheta$. Note that each estimator belongs to $\bd{\Theta}_{n,d}$, the set of all measurable functions from $\mathbb{R}^{d} \times \mathbb{R}^{n \times d}$ to $\mathbb{R}^{d}$.

\begin{assumption}\label{assumption: bounded cov mat 2}
For some constant $L>1$, we have $L^{-1}\le \lmin{\bSigma}\le\lmax{\bSigma}\le L$. 
\end{assumption}

First, we define a family of prior distributions for $\btheta$, under which the support of $\btheta$ is known. For $s\ge 2$ and $r > 0$, define  a probability density function (pdf) with respect to the Lebesgue measure on $\mathbb{R}^{s}$ as follows: for $\bu \in \mathbb{R}^{s}$,
$$\rho_{s,r}(\bu) := \frac{\tilde{\rho}_r(\|\bu\|_2)}{A_{s}\|\bu\|_2^{s-1}}, \quad \text{ where }  \;\tilde{\rho}_r(\tau) := 4r^{-1}\sin^2(2\pi r^{-1} \tau)\idf\{r/2 \le \tau\le r\},$$
where $A_{s}$ is the Lebesgue area of the $s$ dimensional unit sphere $\mS^{s-1}$.




\begin{definition} \label{def: omega_1}
Let $2 \leq s\leq  d$ and $r > 0$. Define $\omega_1(s,r)$ as the distribution of the random vector $\btheta \in \bR^{d}$  given by the following construction: $\btheta_j = 0$ for $j \in \{s+1,\ldots,d\}$, and $\btheta_{[s]}$ has the Lebesgue density $\rho_{s,r}$.
\end{definition}

\begin{lemma}\label{lemma: K known}
Let $2 \leq s\leq  d$ and $r > 0$. 
Suppose that Assumption \ref{assumption: bounded cov mat 2} holds, and that $\btheta$ has the prior distribution $\omega_1(s,r)$. Then, for some constant $C>0$, depending only on $r$, we have
\begin{align*}
    &\inf_{\hbtheta_n \in \bd{\Theta}_{n,d}}\; \Exp_{\btheta\sim \omega_1(s,r)}  \left(\left\|\hbtheta_n-\btheta \right\|_2^2 \right) \ge \frac{\sigma^2 s}{C(n+\sigma^2 s)},\\
    &\inf_{\hbtheta_n \in \bd{\Theta}_{n,d}}\; \Exp_{\btheta\sim \omega_1(s,r)}  \left(\left\|\hbtheta_n-\frac{\btheta}{\|\btheta\|_2} \right\|_2^2 \right) \ge \frac{\sigma^2 s}{C(n+\sigma^2 s)}.
\end{align*}
\end{lemma}
\begin{proof}
Under the prior distribution  $\omega_1(s,r)$, $\btheta$ is supported on $[s]$ with probability $1$, which implies the same property holds for the posterior means: $\Exp[\btheta|Y,Z]$ and 
$\Exp[\btheta/\|\btheta\|_2\;|Y,Z]$. Then, the results follow from Theorem 31 in \cite{song2022truncated}.
\end{proof}

Next, we define the second family of prior distributions for $\btheta \in \bR^{d}$, which is supported on finitely many $\bR^{d}$ vectors. We begin with a lemma.

\begin{lemma} \label{lemma: packing set}
Let $3 \leq s \leq d$ and $r,\delta > 0$. There exists $M \in \bN$ and vectors $\balpha_1,\cdots, \balpha_M\in \mR^d$ such that for any indexes $i\neq j \in [M]$, we have 
\begin{equation}\label{eq: sparse V-G}
\begin{aligned}
   \left\|\balpha_i\right\|_0 = s, \quad \|\balpha_i\|_2^2 = r^2+2\delta^2, \quad \delta^2 \le \|\balpha_i-\balpha_j\|_2^2\le 8\delta^2,
    \quad \log M\ge \mathcal{L}_{d,s}.
\end{aligned}
\end{equation}
where $\mathcal{L}_{d,s} := \frac{s-1}{2}\log\frac{d-s}{(s-1)/2}$.
\end{lemma}

\begin{proof} By
Lemma 5 in \cite{raskutti2011minimax},  there exists 
$M \in \bN$ and  $\tilde{\balpha}_1,\cdots, \tilde{\balpha}_M \in \{-1,0,1\}^{d-1}$ such that for any indexes $i\neq j \in [M]$ it satisfies
$$\left\|\tilde{\balpha}_i\right\|_0 = s-1, \quad \rho(\tilde{\balpha}_i, \tilde{\balpha}_j)\ge \frac{s-1}{2}, \quad \log M\ge \frac{s-1}{2}\log\frac{d-s}{(s-1)/2} = \mathcal{L}_{d,s},$$
where $\rho(\bu,\bv) = \sum_{k=1}^d\idf\{u_k\neq v_k\}$ for $\bu,\bv\in\mR^d$ is the Hamming distance. In particular, for $i\neq j \in [M]$, we have
\begin{align*}
\|\tilde{\balpha}_i\|_2 = \sqrt{s-1},\quad \text{ and } \quad
\sqrt{\frac{s-1}{2}}  \le \|\tilde{\balpha}_i-\tilde{\balpha}_j\|_2\le   2\sqrt{s-1}.
\end{align*}

Finally, for each $i\in[M]$, define the following vector:
$$\balpha_i = (r,\sqrt{\frac{2}{s-1}}\delta\tilde{\balpha}_i')' \in \mR^d.$$
It is elementary to see that $\balpha_1,\cdots, \balpha_M$ satisfy the desired property. The proof is complete.
\end{proof}

\begin{definition}\label{def: omega_2}
Let $3 \leq s \leq d$ and $r,\delta > 0$.  Let $\mathcal{A}(s,r,\delta):=\{\balpha_1,\cdots, \balpha_M\} \subseteq \bR^{d}$ be any set of vectors such that \eqref{eq: sparse V-G} holds. Define $\omega_2(s,r,\delta)$ as the discrete uniform distribution over the set $\mathcal{A}(s,r,\delta)$.
\end{definition}

\begin{lemma}\label{lemma: k unknown}
Let $3 \leq s \leq (d+2)/3$, $r > 0$, and $N \geq n$. Suppose that Assumption \ref{assumption: bounded cov mat 2} holds, and that $\btheta$ has the prior distribution $\omega_2(s,r,\delta)$, where
\begin{align}\label{def:w2_delta}
    \delta^2 := \frac{\sigma^2}{16LN} \mathcal{L}_{d,s},
\end{align}
where $L$ appears in Assumption \ref{assumption: bounded cov mat 2} and $\mathcal{L}_{d,s}$ in Lemma \ref{lemma: packing set}.  Then there exist  constants $C > 0$, depending only on $L$, such that
\begin{align*}
    &\inf_{\hbtheta_n \in \bd{\Theta}_{n,d}}\; \Exp_{\btheta\sim \omega_2} \left(\left\|\hbtheta_n-\btheta \right\|_2^2\right) \ge \frac{\sigma^2}{CN} s\log \frac{d}{s}, \\
  &\inf_{\hbtheta_n\in \bd{\Theta}_{n,d}}\; \Exp_{\btheta\sim \omega_2} \left(\left\|\hbtheta_n-\frac{\btheta}{\|\btheta\|_2} \right\|_2^2 \right)  \ge \frac{1}{r^2 + 2\delta^2} \frac{\sigma^2}{CN} s\log \frac{d}{s}.  
\end{align*}
\end{lemma}

\begin{proof} Denote by $\cF_{M}$ the collection of all measurable function from $\bR^{d}\times\ \bR^{d\times n}$ to $[M]$, where $M$ is the size of the support of the prior distribution  $\omega_2(s,r,\delta)$.

\medskip

\noindent  
\underline{Step 1: reduction to a hypothesis testing problem.} Let $\hbtheta_n$ be any estimator of $\btheta$. Define the following procedure $\hat{\psi} \in \cF_M$:
$$
\hat{\psi}(\bY,\bZ) := \argmin_{j\in [M]}\|\hbtheta_n-\balpha_j\|_2,
$$
with ties resolved in a fixed, deterministic manner. For $j \in [M]$, on the event $\{\hat{\psi}(\bY,\bZ) \neq j\}$, by definition, there exits some $j' \neq j \in [M]$ such that
\begin{align*}
   \|\hbtheta_n -\balpha_j \|_2 \geq  
      \|\hbtheta_n -\balpha_{j'} \|_2,
\end{align*}
which, by the triangle inequality and \eqref{eq: sparse V-G}, implies that
\begin{align*}
     \|\hbtheta_n -\balpha_j \|_2 \geq  
        \|\balpha_{j} -\balpha_{j'} \|_2
-\|\hbtheta_n -\balpha_j \|_2 \geq \delta -\|\hbtheta_n -\balpha_j \|_2.
\end{align*}
Thus, on the event $\{\hat{\psi}(\bY,\bZ) \neq j\}$, we have  $\|\hbtheta_n -\balpha_j \|_2^2 \geq \delta^2/4$. Therefore,
\begin{equation}\label{eq: aux1}
\begin{aligned}
    \Exp_{\btheta\sim \omega_2} (\|\hbtheta_n -\btheta \|_2^2 ) &\ge \frac{\delta^2}{4} \inf_{\psi \in \cF_M} \frac{1}{M}\sum_{i=1}^M \Pro_{\balpha_i}(\psi(\bY,\bZ) \neq i),
\end{aligned}
\end{equation}
where $\Pro_{\balpha_i}$ denotes the conditional distribution of $(\bY,\bZ)$ on the event $\{\btheta = \balpha_i\}$.

By the same argument, we also have
\begin{equation}\label{eq:aux2}
\begin{aligned}
\Exp_{\btheta\sim \omega_2} \left(\left\|\hbtheta_n -\frac{\btheta}{\|\btheta\|_2} \right\|_2^2 \right) &\ge \frac{\delta^2}{4(r^2+2\delta^2)} \inf_{\psi \in \cF_M} \frac{1}{M}\sum_{i=1}^M \Pro_{\balpha_i}(\psi(\bY,\bZ) \neq i).
\end{aligned}
\end{equation}


\medskip

\noindent  
\underline{Step 2: lower bound the average testing error.}  
By Fano's inequality (see \cite[Theorem 4.10]{rigollet2023highdimensionalstatistics}), we have
$$\inf_{\psi \in \cF_M} \frac{1}{M}\sum_{i=1}^M \Pro_{\balpha_i}(\psi(\bY,\bZ) \neq i) \ge 1-\frac{\frac{1}{M^2}\sum_{i,j\in[M]}D(\Pro_{\balpha_i}|| \Pro_{\balpha_j})+\log2}{\log M},$$
where $D(P||Q)=\int \log(dP/dQ) dP$ is the Kullback-Leibler (KL) divergence between two distributions $P$ and $Q$. For $i \neq j \in [M]$, since $\epsilon_1,\ldots, \epsilon_n$ are i.i.d.~distributed as $\mathcal{N}(0,\sigma^2)$,
\begin{align*}
\log\frac{d \Pro_{\balpha_i}}{d\Pro_{\balpha_j}}(\bY,\bZ) = \frac{1}{2\sigma^2} \left[2\bY'\bZ(\balpha_i - \balpha_j) - \|\bZ \balpha_i\|_2^2 + \|\bZ \balpha_j\|_2^2 
\right],
\end{align*}
which implies that $D(\Pro_{\balpha_i}||\Pro_{\balpha_j}) = (2\sigma^2)^{-1}\Exp_{\balpha_i}\left[\|\bZ(\balpha_i-\balpha_j)\|_2^2\right]$. Then, due to Assumption \ref{assumption: bounded cov mat 2} and \eqref{eq: sparse V-G}, we have
\begin{align*}
  \frac{1}{M^2}\sum_{i,j\in[M]}D(\Pro_{\balpha_i}||\Pro_{\balpha_j}) \leq   \frac{1}{M^2}\sum_{i,j\in[M]} \frac{1}{2\sigma^2}8\delta^2 n L = \frac{4 \delta^2 n L}{\sigma^2}.
\end{align*}
Then due to the definition of $\delta^2$ in \eqref{def:w2_delta} and since $N > n$ and $\log M \geq \mathcal{L}_{d,s}$ (see \eqref{eq: sparse V-G}), 
$$
\frac{1}{M^2}\sum_{i,j\in[M]}D(\Pro_{\balpha_i}||\Pro_{\balpha_j}) \leq \frac{1}{4} \mathcal{L}_{d,s} \leq \frac{1}{4}\log M.
$$
Since $3 \leq s \leq (d+2)/3$, we have $\log(M) \geq \mathcal{L}_{d,s} \geq 2\log(2)$, and thus
$$
\inf_{\psi \in \cF_M} \frac{1}{M}\sum_{i=1}^M \Pro_{\balpha_i}(\psi(\bY,\bZ) \neq i) \ge 1 - \frac{1}{4} - \frac{1}{2} = \frac{1}{4}.
$$
Since $3 \leq s \leq (d+2)/3$, we further have
$$
\delta^2 \geq  \frac{\sigma^2}{16 LN} \frac{s}{4} \log\left(\frac{2 d}{ 3 s}\right), \quad \text{ and } \quad \frac{2d}{3s} \geq \left(\frac{d}{s}\right)^{1/3}.
$$
Then the proof is complete in view of \eqref{eq: aux1} and \eqref{eq:aux2}.
\end{proof}

\section{Proofs for the Sequential Estimation Problem}\label{sec: bounds on sequential estimation}
\subsection{Proof of Theorem \ref{thm: seq multi linear, upper bound}} \label{sec: proof of seq multi linear, upper bound}
We begin with a Lemma that provides a high probability upper bound on the $\ell_{\infty}$ estimation error of Lasso estimators. 

\begin{lemma}
\label{lemma: improved_ell_infty}
Suppose that Assumption \ref{assumption: sequential estimation} holds, and that $d \ge (2L_0+1)s_0 + 2$. There exist constants $C_0, \kappa_1>0$ depending only on $L_0$ and a constant $C >0$ depending only on $C_0, L_0$, such that if we set the regularization parameter $\lambda_t$  as follows:
$$
\lambda_t = C_0 \sigma \sqrt{\frac{\log(dt)}{t}},
$$
then for $t\ge \kappa_1s_0\log d$, with probability at least $1 - {C}de^{-t/{C}} - {C}/t$, we have
$$\|\ltheta_t -\btheta \|_{\infty} \le C\lambda_t.$$
\end{lemma}
\begin{proof}
This result follows essentially from Theorem 5.1 in \cite{bellec2022biasing}, with a minor modification that we outline below. Specifically, let 
  \begin{align*}
    &k=1, \quad m=L_0 s_0, \quad \rho_* = L_0^{-1}, \quad \eta_2=0.05, \quad  \epsilon_1=\epsilon_2=0.05, \quad \epsilon_3=\epsilon_4=\epsilon_2^2/16. \end{align*}
Then, if Assumption \ref{assumption: sequential estimation} holds and $d \ge (2L_0+1)s_0 + 2$, there exists a large enough constant $\kappa_1 > 0$ depending only on $L_0$ such that for all $t \geq \kappa_1 s_0 \log(d)$,
\begin{enumerate}[label=\roman*)]
    \item $\bSigma$ is invertible, and for $j \in [d]$, $\bSigma_{j,j} \leq 1$;
    \item with $\tau_* = (1-\epsilon_1-\epsilon_2)^2, \tau^* = (1+\epsilon_1+\epsilon_2)^2$, we have 
    $$
    s_0+k <  \frac{(1-\eta_2)^2 2m}{(1+\eta_2)^2 \left((\tau^*/\tau_*) \phi_{\text{cond}}(m+k,S,\bSigma)-1\right)},
    $$
    where $\phi_{\text{cond}}(m+k,S,\bSigma) := \max_{A\subseteq [d]: |A\backslash S|\le 1\vee (L_0s_0+1)} \frac{\lmax{\bSigma_{A,A}}}{\lmin{\bSigma_{A,A}}}$;

    \item $\rho_* \leq \min_{A\subseteq[d]: |A\backslash S|=L_0s_0+1} \lmin{\bSigma_{A,A}}$ and $\lambda_0\sqrt{s_*}\le 1$, where
    $$\lambda_0 = \sqrt{2(1+4L_0^2)\frac{\log (8dt)}{t}}, \quad \text{ and } \quad s_* = s_0+m+k;
    $$

\item finally,
\begin{align*}
    &2(m+k)+s_0+1\le (t-1)\wedge (d+1), \quad
    \epsilon_1+\epsilon_2<1, \quad\epsilon_3+\epsilon_4=\epsilon_2^2/8, \\
    & s_0+m+k+1\le (d+1)\wedge (\epsilon_1^2 t/2), \quad \log\binom{d-s_0}{m+k}\le \epsilon_3 t.
\end{align*}
\end{enumerate}
That is, the conditions (3.5)-(3.8) in Assumption 3.1 of \cite{bellec2022biasing} hold with a modified $\lambda_0$ defined above. 

Then, following the comment after Proposition A.6 of \cite{bellec2022biasing}, if we replace $L_k = \sqrt{2\log (d/k)}$ by $L_k = \sqrt{2(1+4L_0^2) \log (dt)}$ in Proposition A.6 of \cite{bellec2022biasing}, we have $\Pro\left(\Omega_{noise}^{(1)}\cap \Omega_{noise}^{(2)}\right) \geq 1- t^{-1}$. The remainder of the argument closely follows the proof of Theorem 5.1 in \cite{bellec2022biasing} and is therefore omitted.
\end{proof}

\begin{lemma} \label{lemma: lower eigenvalue whp}
Suppose that Assumption \ref{assumption: bounded cov mat 1} holds, and that the support $S$ of $\btheta$ is non-empty.  Let $C = (16L_0^2)^2$.
Then, for each $t \in [T]$, if $t \ge  C s_0$, then with probability at least $1-2e^{-t/(2C^2)}$, 
$$
\frac{L_0^{-1}}{2} \le \lmin{\frac{\bX_{[t],S}'\bX_{[t],S}}{t}} \le \lmax{\frac{\bX_{[t],S}'\bX_{[t],S}}{t}}\le \frac{3L_0}{2},
$$
where $\bX_{[t],S}$ is the submatrix of $\bX_{[t]}$ with columns indexed by $S$. 
\end{lemma}
\begin{proof} Let $\bZ := \bX_{[t],S}$. 
Note that $\Exp[\bZ'\bZ/ t] =  \bSigma_{S,S}$. By Example 6.3 in \cite{wainwright2019high}, with probability at least $1-2e^{-t\delta^2/2}$,
$$\op{\frac{\bZ'\bZ}{t}-\bSigma_{S,S}}_{\textup{OP}} \le \lmax{\bSigma_{S,S}} \left(2\sqrt{\frac{s_0}{t}}+2\delta+\left(\sqrt{\frac{s_0}{t}}+\delta\right)^2\right),$$
where $\op{\cdot}_{\textup{OP}}$ denotes the operator norm of a matrix. 

Let $C =  (16L_0^2)^2$ and $\delta = 1/C$ . Then, if $t \geq Cs_0$, due to Assumption \ref{assumption: bounded cov mat 1}, with probability at least $1-2e^{-t/(2C^2)}$,
$\op{{\bZ'\bZ}/{t}-\bSigma_{S,S}}_{\textup{OP}} \leq 1/(2L_0)$. Then the proof is complete due to Assumption \ref{assumption: bounded cov mat 1} and the variational characterization of eigenvalues.
\end{proof}

Finally, we prove Theorem \ref{thm: seq multi linear, upper bound}.

\begin{proof}[Proof of Theorem \ref{thm: seq multi linear, upper bound}]
For $t \in [T]$, denote by $\bX_{[t],S}$  the submatrix of $\bX_{[t]}$ with columns indexed by $S$. 
By Lemma \ref{lemma: improved_ell_infty} and Lemma  \ref{lemma: lower eigenvalue whp}, there exist constants $\kappa_1, C_0, C_0^{\textup{hard}}>0$ depending only on $L_0$, and a large enough constant $C>0$ depending only on $ L_0, C_0$, such that for all $t\ge \kappa_1 s_0\log d$, the events
\begin{align*}
    \mathcal{A}_{1,t} &:=\left\{\frac{L_0^{-1}}{2} \le \lmin{\frac{\bX_{[t],S}'\bX_{[t],S}}{t}} \le \lmax{\frac{\bX_{[t],S}'\bX_{[t],S}}{t}}\le \frac{3L_0}{2} \right\}, \\
    \mathcal{A}_{2,t} &:=\left\{\left\|\ltheta_t -\btheta \right\|_{\infty} \le \hdlambda_t \right\},
\end{align*}
satisfy the following:
$\Pro\left((\mathcal{A}_{1,t} \cap \mathcal{A}_{2,t})^c\right)\le Cde^{-t/C} + C/t$.

First, we fix some $t\ge \kappa_1 s_0\log d$ and bound the $\ell_2$ estimation error of $\ptheta_t$ on the event $\mathcal{A}_{1,t}\cap\mathcal{A}_{2,t}$. By Theorem \ref{thm: ols post l2 estimation error}, on the event $\mathcal{A}_{1,t}\cap\mathcal{A}_{2,t}$, there exists a constant $C>0$ depending only on $L_0$ such that 
\begin{equation*}
\Exp\left[\left\|\ptheta_t-\btheta \right\|_2^2\idf{\{\cA_{1,t}\cap \cA_{2,t}\}}\right] \le C\sum_{j\in S} \btheta_j^2 \idf\left\{|\btheta_j| \le 2\hdlambda_t  \right\} + C\Exp\left[\|\bX_{[t],S}'\bepsilon_{[t]}\|_2^2\right],
\end{equation*} 
where $\bepsilon_{[t]} = (\epsilon_1,\ldots,\epsilon_t)'$. 
Due to Assumption \ref{assumption: sequential estimation},  we have $\bSigma_{j,j} \leq 1$ for $j \in [d]$, and thus 
$$\Exp\left[\|\bX_{[t],S}'\bepsilon_{[t]}\|_2^2\right] = \text{trace}(\Exp[\bX_{[t],S}'\bepsilon_{[t]}\bepsilon_{[t]}'\bX_{[t],S}]) = \sigma^2 t \times \text{trace}(\bSigma_{S,S}) \leq \sigma^2 s_0 t.$$
Therefore, for $t\ge \kappa_1 s_0\log d$ we have
\begin{equation*} 
\begin{aligned}
&\Exp\left[\left\|\ptheta_t-\btheta \right\|_2^2 \wedge \xi \right]^2 \le \Exp\left[\left\|\ptheta_t-\btheta \right\|_2^2 \idf{\{\mathcal{A}_{1,t} \cap \mathcal{A}_{2,t}\}} \right] + \xi \Pro\left((\mathcal{A}_{1,t} \cap \mathcal{A}_{2,t})^c\right)
\\&\le C \sum_{j\in S} \btheta_j^2 \idf\left\{|\btheta_j| \le 2\hdlambda_t  \right\} + C\sigma^2 s_0/t + \xi (Cde^{-t/C} + C/t),
\end{aligned}
\end{equation*}
which completes the proof for the first claim.

Next, we focus on the cumulative estimation error. Let $\widetilde{C} > 0$ be a constant to be specified, and define $t_0 = \lceil(\widetilde{C}\vee\kappa_1) s_0\log d\rceil$.  For $1 \leq t \leq t_0$, we bound each estimation error by $\xi$:
\begin{equation*} 
    \sum_{t=1}^{t_0} \Exp\left[\|\ptheta_t-\btheta \|_2^2\wedge \xi\right]  \le \xi (\widetilde{C}\vee\kappa_1) s_0\log d.
\end{equation*}

For $t > t_0$, we use the first claim: for some constant $C > 0$ depending on $\kappa_1, L_0, C_0, C_0^{\textup{hard}}$, but not on $\widetilde{C}$, we have
\begin{equation*} 
\begin{aligned}
    \sum_{t=t_0+1}^T \Exp\left[\|\ptheta_t-\btheta \|_2^2 \wedge \xi\right] &\le C \sum_{t=1}^T \sum_{j\in S} \btheta_j^2 \idf\left\{|\btheta_j| \le \sqrt{\frac{C\sigma^2\log (dt)}{t}} \right\} + C\sigma^2 \sum_{t=2}^{T}s_0/t
    \\&\quad + \xi \sum_{t=t_0+1}^T Cde^{-t/C} + \xi \sum_{t=2}^T C/t. 
\end{aligned}
\end{equation*}
For the first term on the right-hand side, by Lemma \ref{lemma: instance sum} with $b = C \sigma^2\log(dT)$,
\begin{equation*}
    \sum_{t=1}^T \sum_{j\in S} \btheta_j^2 \idf\left\{|\btheta_j| \le \sqrt{\frac{C\sigma^2\log (dt)}{t}} \right\} \le C\sigma^2 s_0(\log d + \log T).
\end{equation*}
For the second and fourth terms, note that $\sum_{t=2}^T 1/t \le \int_1^T (1/x) dx\le \log T$. Finally, for the third term, if we choose $\widetilde{C} = C$, we have
$$
\sum_{t=t_0+1}^T de^{-t/C} \le \int_{C\log d}^{+\infty} e^{-(t-C\log d)/C} dt = C.
$$ 
Then the proof is complete.

\end{proof}

\subsection{Proof of Theorem \ref{thm: seq multi linear, lower bound}} \label{sec: proof of seq multi linear, lower bound}

Recall the two families of distributions on $\bR^{d}$: $\omega_1(\cdot)$ in Definition \ref{def: omega_1} and $\omega_2(\cdot)$ in Definition \ref{def: omega_2}.

\begin{proof}[Proof of Theorem \ref{thm: seq multi linear, lower bound}]
Let $\{\hbtheta_t, t\in [T]\}$ be any sequence of permissible estimators.

\medskip
\noindent \underline{Step 1. Lower bound the supremum risk by the average risk.}  Let $\omega$ be a distribution on $\bR^{d}$ defined as an equally weighted mixture of the following two distributions:
\begin{align*}
    \omega_1(s_0, \sqrt{\xi/16}) \quad \text{ and } \quad 
    \omega_2(s_0, \sqrt{\xi/32}, 
    \tilde{\delta}) \;\; \text{ with } 
    \tilde{\delta}^2 := \frac{\sigma^2}{16 L_0 T} \frac{s_0-1}{2} \log\left( \frac{d-s_0}{(s_0-1)/2}\right).
\end{align*}
Importantly, $\omega$ does not depend on $t$, and both $\omega_1$ and $\omega_2$ are supported on $\bTheta_d[s_0]$, where we omit the arguments of $\omega_1(\cdot)$ and $\omega_2(\cdot)$ for simplicity. Thus,
\begin{align*}
   &2 \sup_{\btheta\in \bTheta_d[s_0]} \Exp_{\btheta} \left(\sum_{t=1}^{T}  \|\hbtheta_t-\btheta\|_2^2\wedge \xi \right) \geq 
    2\Exp_{\btheta \sim \omega}\left( \sum_{t=1}^{T}   \|\hbtheta_t-\btheta\|_2^2\wedge \xi  \right) \\
    &\qquad =\sum_{t=1}^{T}      \Exp_{\btheta \sim \omega_1}\left(  \|\hbtheta_t-\btheta\|_2^2\wedge \xi  \right) + 
   \sum_{t=1}^{T} \Exp_{\btheta \sim \omega_2}\left(    \|{\hbtheta}_t-\btheta\|_2^2\wedge \xi  \right),
\end{align*}
where $\Exp_{\btheta \sim \omega}$ means that the parameter $\btheta$ is treated as a random vector with the prior distribution $\omega$.

\medskip
\noindent \underline{Step 2. Remove the impact of $\xi$.} Define the following estimators: for $t \in [T]$,
\begin{align*}
    \widetilde{\btheta}_t := \begin{cases}
        \hbtheta_t, & \text{ if } \|\hbtheta_t\|_2 \leq \sqrt{\xi/4}, \\
        0, & \text{ if } \|\hbtheta_t\|_2 > \sqrt{\xi/4}. 
    \end{cases}
\end{align*}
If $\|\btheta\|_2 \leq \sqrt{\xi/16}$, we have the following two cases:
\begin{enumerate}[label=\roman*)]
    \item If $\|\hbtheta_t\|_2 \leq \sqrt{\xi/4}$, then $\|\widetilde{\btheta}_t - \btheta\|_2 =   \|\hbtheta_t- \btheta\|_2$;
    \item If $\|\hbtheta_t\|_2 > \sqrt{\xi/4}$, then 
    $\|\widetilde{\btheta}_t - \btheta\|_2 \leq \sqrt{\xi/16} =    \sqrt{\xi/4} - \sqrt{\xi/16} \leq \|\hbtheta_t- \btheta\|_2$.
\end{enumerate}
That is, if $\|\btheta\|_2 \leq \sqrt{\xi/16}$, we always have $\|\widetilde{\btheta}_t - \btheta\|_2 \leq   \|\hbtheta_t- \btheta\|_2$.

By Definition \ref{def: omega_1}, under the prior $\omega_1(s_0, \sqrt{\xi/16})$, we have $\|\btheta\|_2 \leq \sqrt{\xi/16}$ with probability 1. By Definition \ref{def: omega_2}, under the prior $\omega_2(s_0, \sqrt{\xi/32}, \tilde{\delta})$, we have $\|\btheta\|_2 = \sqrt{\xi/32 + 2\tilde{\delta}^2}\le \sqrt{\xi/16}$ with probability 1, where the last inequality is because $\tilde{\delta}^2 \leq \xi/64$ since $T\ge 2\sigma^2 \xi^{-1}s_0\log d$ and $L_0 > 1$. As a result, we have
$$\Exp_{\btheta \sim \omega_\ell}\left(  \|\hbtheta_t-\btheta\|_2^2\wedge \xi  \right)
\geq \Exp_{\btheta \sim \omega_\ell}\left(  \|\widetilde{\btheta}_t-\btheta\|_2^2\wedge \xi  \right)
, \text{ for } t \in [T],\; \ell \in \{1,2\}.
$$
In addition, for $\|\btheta\|_2 \leq \sqrt{\xi/16}$, we also have $
\|\widetilde{\btheta}_t-\btheta\|_2^2 \leq 2 ( \xi/4 + \xi/16) \leq \xi$, which implies
$$\Exp_{\btheta \sim \omega_\ell}\left(  \|\hbtheta_t-\btheta\|_2^2\wedge \xi  \right)
\geq \Exp_{\btheta \sim \omega_\ell}\left(  \|\widetilde{\btheta}_t-\btheta\|_2^2  \right)
, \text{ for } t \in [T],\; \ell \in \{1,2\}.
$$

\medskip
\noindent \underline{Step 3. Lower bound the Bayes risks.} By Lemma \ref{lemma: K known} and Lemma \ref{lemma: k unknown},  for some constant $C > 0$, depending only on $\xi$ and $L_1$, we have
\begin{equation*} 
\begin{aligned}
\sum_{t=1}^{T}\sum_{\ell=1}^{2}   \Exp_{\btheta \sim \omega_\ell}\left(  \|\widetilde{\btheta}_t-\btheta\|_2^2  \right) 
\geq&  C^{-1} \sum_{t=1}^{T} \left(\frac{\sigma^2 s_0}{t+\sigma^2s_0} + \frac{\sigma^2}{T}s_0\log\frac{d}{s_0} \right) \\
\geq&  C^{-1} \sigma^2 s_0 \left(\log\left( \frac{T+\sigma^2 s_0}{2+\sigma^2s_0} \right)+ \log\frac{d}{s_0} \right).
\end{aligned}
\end{equation*}
The proof is complete.
\end{proof}

\subsection{Proof of Theorem \ref{thm: lower bound of lasso}} \label{sec: proof of thm: lower bound of lasso}
Let $t \in [T]$ and $j \in [d]$. Recall that $\bX_{[t],S}$ (resp.~$\bX_{[t],j}$) is the submatrix of $\bX_{[t]}$ with columns indexed by $S$ (resp.~$\{j\}$). Further, $\ltheta_t$ is the Lasso estimator at time $t$, and denote by $\hat{S}_t$ the support of $\ltheta_t$.

Define the following events:
\begin{equation}\label{def:lasso_lower_events}
\begin{aligned}
    &\mathcal{A}_{1,t}=\left\{\frac{L_0^{-1}}{2}\le \lmin{\frac{\bX_{[t],S}'\bX_{[t],S}}{t}} \le \lmax{\frac{\bX_{[t],S}'\bX_{[t],S}}{t}} \le \frac{3L_0}{2}  \right\},
    \\&
    \mathcal{A}_{2,t}=\left\{\hat{S}_{t} \subseteq S \right\},
    \\&
    \mathcal{A}_{3,t}=\left\{\|\ltheta_t - \btheta\|_{\infty} \leq \left( 
    \sqrt{L_0} + 2 L_0 s_0
    \right) \lambda_t\right\},
\end{aligned}
\end{equation}
where recall that $L_0$ and $L_2$ appears in Assumptions \ref{assumption: bounded cov mat 1} and \ref{assumption: multi linear cov matrix}.

\begin{lemma}\label{lemma:lasso_lower_events}
Suppose that Assumptions \ref{assumption: bounded variance},  \ref{assumption: bounded cov mat 1}, and  \ref{assumption: multi linear cov matrix} hold. Further, assume $d \geq 8$ and for some constant $C_0 > 0$, $\lambda_t$ satisfies the following:
\begin{align}
    \label{lasso:lower_lambda}
   1 \geq  \lambda_t \geq \frac{16\sigma}{1 - L_2^{-1}} \sqrt{\frac{2\log(d)}{t}}, \quad \text{ for } \; t \geq C_0 s_0^2 \log d,
\end{align}
Then there exists some constant $C_T \geq C_0$, that only depends on $L_0,L_2$, such that for all $t \geq C_T s_0^2 \log(d)$, we have $\Pro\left(\bigcap_{i=1}^{3} \cA_{i,t} \right) \geq 1/5$.
\end{lemma}

The proof of Lemma \ref{lemma:lasso_lower_events} is presented after the proof of Theorem  \ref{thm: lower bound of lasso}.

\begin{proof}[Proof of Theorem \ref{thm: lower bound of lasso}]
Recall the events defined in \eqref{def:lasso_lower_events}. Assume $d \geq 8$.  By Lemma \ref{lemma:lasso_lower_events}, there exists some constant $C_T \geq C_0$, that only depends on $L_0,L_2$, such that for all $t \geq C_T s_0^2 \log(d)$, we have $\Pro\left(\cA_t\right) \geq 1/5$, where $\cA_t := \bigcap_{i=1}^{3} \cA_{i,t}$. Without loss of generality, assume the subset $S$ in Assumption \ref{assumption: multi linear cov matrix} is $[s_0]$.

For an integer $\ell \geq 1$, let $\mathbf{0}_{\ell}$  and $\mathbf{1}_{\ell}$  denote the all-zeros and all-ones vector of length $\ell$ respectively. Define
$$
\btheta^*=(2\left( 
    \sqrt{L_0} + 2 L_0 s_0
    \right)  \mathbf{1}_{s_0}',\; \mathbf{0}_{d-s_0}')',
$$
and we focus on the case $\btheta = \btheta^*$ and $t \geq C_T s_0^2\log(d)$.

By the definition of $\cA_{2,t}$ and $\cA_{3,t}$ and since $\lambda_t \leq 1$, when $\btheta = \btheta^*$, we have that on the event $\cA_t$, 
$$
(\ltheta_t)_j > 0 \;\;\text{ for } j \in S,\quad \text{ so that }\quad
\hat{S}_t = S,$$
which implies that 
$$(\ltheta_t)_S = \argmin_{\bbeta \in \bR^{s_0}} \left\{\frac{1}{2t}\|\bY-\bX_{[t],S}\bbeta\|_2^2 + \lambda_t\|\bbeta\|_1 \right\}.$$
As a result, on the event $\cA_t$, we have
$$\frac{\bX_{[t],S}'\bX_{[t],S}}{t}[(\ltheta_t-\btheta^*)_S]=\frac{\bX_{[t],S}'\bepsilon_{[t]}}{t}-\lambda_t\bd{1}_{s_0},
$$
where $\bepsilon_{[t]}:=(\epsilon_1,\ldots,\epsilon_{t})'$. Thus, $\left\|\ltheta_t-\btheta^*\right\|_2^2 \wedge \xi  \ge I_t - II_{t}$, where
\begin{align*}
I_t:=\frac{\lambda_t^2}{2} \left\|\left(\frac{\bX_{[t],S}'\bX_{[t],S}}{t}\right)^{-1}\bd{1}_{s_0}\right\|_2^2 \wedge \xi,\quad
II_{t}:= \|(\bX_{[t],S}'\bX_{[t],S})^{-1}\bX_{[t],S}'\bepsilon_{[t]}\|_2^2.
\end{align*}
On the event $\cA_t$ (in particular, $\cA_{1,t}$) and due to \eqref{lasso:lower_lambda}, we have
\begin{align*}
    I_{t} \ge \frac{\lambda_t^2\|\bd{1}_{s_0}\|_2^2}{2\left[\lmax{\bX_{[t],S}'\bX_{[t],S}/t}\right]^2}\wedge \xi
    \geq \frac{s_0}{2(3L_0/2)^2} \left(\frac{16\sigma}{1-L_2^{-1}}\right)^2 \frac{2\log(d)}{t} \wedge \xi.
\end{align*}
Thus, there exists some $C_T' \geq C_T$   depending on $L_0,L_2, \sigma^2, \xi$ and some $C' > 0$ depending only on $L_0,L_2$, such that for $t \geq C_T' s_0^2\log(d)$, on the event $\cA_t$, we have $I_{t} \geq (C')^{-1}\sigma^2 s_0 \log(d)/t$.

Further, note that $\cA_{1,t}$ only depend on $\bX_{[t],S}$ and thus is dependent from $\bepsilon_{[t]}$. Thus,
\begin{align*}
\Exp_{\btheta^*}[II_t; \cA_{1,t}] = &\Exp_{\btheta^*}\left[
\text{trace}\left((\bX_{[t],S}'\bX_{[t],S})^{-1}\bX_{[t],S}'\bepsilon_{[t]} \bepsilon_{[t]}'\bX_{[t],S}(\bX_{[t],S}'\bX_{[t],S})^{-1} \right); \cA_{1,t}
\right]\\
= &\sigma^2\Exp_{\btheta^*}\left[
\text{trace}\left((\bX_{[t],S}'\bX_{[t],S})^{-1}\right); \cA_{1,t}
\right] \leq \sigma^2 \frac{s_0}{t} 2L_0.
\end{align*}
If $\log(d) \geq 20L_0 C'$, then $\Exp[II_t; \cA_{1,t}] \leq (10C')^{-1} \sigma^2 s_0\log(d)/t$.

To sum up, if $\log(d) \geq 20L_0 C'$, for $t \geq C_{T}'s_0^2 \log(d)$, we have
\begin{align*}
   & \Exp_{\btheta^*}\left(\left\|\ltheta_t-\btheta^*\right\|_2^2 \wedge \xi \right)\ge \Exp_{\btheta^*}\left[I_t - II_t ; \cA_t \right] \\
    \geq & (C')^{-1} \sigma^2 \frac{s_0\log(d)}{t} \times \Pro_{\btheta^*}(\cA_t) - (10C')^{-1} \frac{\sigma^2 s_0\log(d)}{t} \geq (10C')^{-1} \frac{\sigma^2 s_0\log(d)}{t},
\end{align*}
where for the last inequality, we use the fact that $\Pro_{\btheta^*}(\cA_t) \geq 1/5$.

Thus if $T \geq (C_{T}'s_0^2 \log(d))^{\kappa} := t_0^{\kappa}$ for some $k > 1$, we have
\begin{align*}
&\sum_{t=1}^T \Exp_{\btheta^*} (\|\ltheta_t-\btheta^*\|_2^2  \wedge \xi) \ge \Exp_{\btheta^*} \sum_{t=t_0}^T (\|\ltheta_t-\btheta^*\|_2^2 \wedge \xi) \\
\ge &(10C')^{-1}\sigma^2 s_0\log d\log \frac{T}{t_0}
\ge (10C')^{-1}\sigma^2 (1-1/\kappa)s_0 \log d\log T.
\end{align*}
The proof is complete.
\end{proof}

Now, we prove Lemma \ref{lemma:lasso_lower_events}. Recall the events $\cA_{i,t}$ for $i \in [3]$ in \eqref{def:lasso_lower_events}.

\begin{proof}[Proof of Lemma \ref{lemma:lasso_lower_events}]
 For $t \in [T]$, we define \begin{equation*}
\begin{aligned}
& \mathcal{A}_{4,t}=\left\{ \max_{j\in{S^c}}\left\| \left(\bX_{[t],S}'\bX_{[t],S}\right)^{-1} \bX_{[t],S}' \bX_{[t],j}\right\|_1    \le\frac{1}{2}(1+ L_2^{-1})    \right\},
\\&\mathcal{A}_{5,t}=\left\{\max_{j\in [d]}\ t^{-1}\|\bX_{[t],j}\|_2^2\le 2\right\}.
\end{aligned}
\end{equation*}

On the event $\cA_{1,t}$, we have
$$
\op{\left( \frac{\bX_{[t],S}'\bX_{[t],S}}{t}\right)^{-1}}_{\infty} \leq 2L_0 s_0.
$$
Due to \eqref{lasso:lower_lambda}, we have
$$
\delta := \frac{1-L_2^{-1}}{8 \sigma} \lambda_t - \sqrt{\frac{2\log(d-s_0)}{t}} \geq  \sqrt{\frac{2\log(d)}{t}}.
$$
Since the events $\cA_{1,t}, \cA_{4,t}$ and $ {\cA_{5,t}}$ only depend on $\bX_{[t]}$,
 by Corollary 7.22 in \cite{wainwright2019high}, 
\begin{align*}
    \Pro\left(\cA_{2,t} \cap \cA_{3,t} \;\vert\; \cA_{1,t} \cap \cA_{4,t} \cap {\cA_{5,t}} \right) \geq 1 - 4 \exp(- t\delta^2/2) \geq 1 - 4/d.
\end{align*}
Since $d \geq 8$, we have
\begin{align*}
    \Pro\left(\cA_{1,t} \cap \cA_{2,t} \cap \cA_{3,t}\right)
\geq (1-4/d) \times     \Pro\left(\cA_{1,t} \cap \cA_{4,t} \cap {\cA_{5,t}} \right) \geq 1/2 \times \Pro\left(\cA_{1,t} \cap \cA_{4,t} \cap {\cA_{5,t}} \right).
\end{align*}
In Steps 1-3, we show that there exists some constant $C_T > 0$, depending on $L_0,L_2$, such that if $t \geq C_T s_0^2\log(d)$, then
$$
\Pro\left(\cA_{i,t}^c \right) \leq 1/5, \quad \text{ for } i = 1,4,5.
$$
Then the proof is complete since this would imply $\Pro\left(\cA_{1,t} \cap \cA_{4,t} \cap {\cA_{5,t}} \right) \geq 2/5$.

\medskip
\noindent \underline{Step 1: the event $\cA_{1,t}$.} By Lemma \ref{lemma: lower eigenvalue whp}, there exists $C > 0$ depending on $L_0$ such that for all $t \geq C s_0$, we have $\Pro(\cA_{1,t}^c) \leq 1/10$.

\medskip
\noindent \underline{Step 2: the event $\cA_{5,t}$.} Let $j \in [d]$. Without loss of generality, assume $\bSigma_{j,j} > 0$. Due to Assumption \ref{assumption: bounded variance}, 
$$\Pro\left(t^{-1} {\|\bX_{[t],j}\|_2^2} > 2 \right) \le \Pro\left(
(t \bSigma_{j,j})^{-1} \|\bX_{[t],j}\|_2^2 - 1 \ge 1
\right).$$
Since the entries of $(\bSigma_{j,j})^{-1/2} \bX_{[t],j}$ are i.i.d. standard Gaussians $\mathcal{N}(0,1)$, by Example 2.11 in \cite{wainwright2019high} and the union bound, we have
$$
\Pro\left(\cA_{5,t}^c\right) \leq d e^{-t/8}.
$$
Thus if $t \geq 8\log(5d)$, then $\Pro\left(\cA_{5,t}^c\right)  \leq 1/5$.


\medskip
\noindent \underline{Step 3: the event $\cA_{4,t}$.}  For $t \in [T]$ and $j \in S^c$, we define the following $t$-dimensional vector:
$$
\bZ_{[t]}^{(j)} = \bX_{[t],j} - \bX_{[t],S} \bSigma_{S,S}^{-1} \bSigma_{S,j}.
$$
Since $\bX_1,\ldots,\bX_t$ are i.i.d. distributed as $\mathcal{N}(0,\bSigma)$, we have that
$\bZ_{[t]}^{(j)} = (Z_1^{(j)},\ldots,Z_t^{(j)})'$ are independent from $\bX_{[t],S}$, and further that $Z_1^{(j)},\ldots,Z^{(j)}_t$ are i.i.d. distributed as $\mathcal{N}(0,b_j^2)$, where 
$b_j^2 := \bSigma_{j,j}-\bSigma_{j,S}(\bSigma_{S,S})^{-1}\bSigma_{S,j}$. 

Fix some $j \in S^c$. Note that
\begin{align*}
    \left(\bX_{[t],S}'\bX_{[t],S}\right)^{-1}
    \bX_{[t],S}' \bX_{[t],j} = \bSigma_{S,S}^{-1} \bSigma_{S,j} + \left(\bX_{[t],S}'\bX_{[t],S}\right)^{-1}
    \bX_{[t],S}' \bZ_{[t]}^{(j)}.
\end{align*}
Let $\alpha := 2^{-1}(1-L_2^{-1})$. Due to Assumption \ref{assumption: multi linear cov matrix} and $\|\bv\|_{1} \leq \sqrt{s_0} \|\bv\|_2$ for any $\bv \in \bR^{s_0}$, we have $\widetilde{\cA}_{4,t} \subseteq \cA_{4,t}$, where
\begin{align*}
    \widetilde{\cA}_{4,t} := \bigcap_{j \in S^{c}} \widetilde{\cA}_{4,t,j}, \quad \text{ and } \quad \widetilde{\cA}_{4,t,j}:=   \left\{\left\|\left(\bX_{[t],S}'\bX_{[t],S}\right)^{-1}
    \bX_{[t],S}' \bZ_{[t]}^{(j)}\right\|_2^2 \leq \frac{1}{s_0} \alpha^2
    \right\}.
\end{align*}

For each $j \in S^c$, given $\bX_{[t], S}$, the conditional distribution of $ (\bX_{[t],S}'\bX_{[t],S})^{-1} \bX_{[t],S}' \bZ_{[t]}^{(j)}$ is the same as $b_j (\bX_{[t],S}'\bX_{[t],S})^{-1/2} \bW$, where $\bW$ follows the standard $s_0$-dimensional normal distribution $\mathcal{N}(0,  \mathbbm{I}_{s_0})$ and is independent from any other random variables. Since $\cA_{1,t}$ only involves $\bX_{[t],S}$, we have
\begin{align*}
&\Pro\left(\widetilde{\cA}_{4,t,j}^c \cap \mathcal{A}_{1,t}\right)
=\Exp\left[\Pro\left(
\left\|(\bX_{[t],S}'\bX_{[t],S})^{-1/2} \bW\right\|_2^2 > \frac{\alpha^2}{b_j^2 s_0} \; \vert \;
    \bX_{[t],S}
\right) ; \mathcal{A}_{1,t}\right] \\
\leq &\Exp\left[\Pro\left(
\left\|\bW\right\|_2^2 > \frac{\alpha^2 t}{2 L_0 b_j^2 s_0} \; \vert \;
    \bX_{[t],S}
\right) ; \mathcal{A}_{1,t}\right]
\leq \Pro\left(
\frac{1}{s_0}\left\|\bW\right\|_2^2 > \frac{\alpha^2 t}{2 L_0 b_j^2 s_0^2}
\right).
\end{align*}  
Since $b_j \leq 1$, by Equation (2.18) and Example 2.11 in \cite{wainwright2019high}, there exists some $C > 0$, depending only on $L_0, L_2$, such that if $t \geq C s_0^2 \log(d)$, we have
$$\Pro\left(\frac{1}{s_0}\left\|\bW\right\|_2^2-1 > \frac{\alpha^2 t-2 L_0 b_j^2 s_0^2}{2 L_0 b_j^2 s_0^2}\right)\leq 1/(10d).$$ 
Then by a union bound and Step 1, we have
\begin{align*}
   \Pro\left({\cA}_{4,t}^c \right) \leq \sum_{j \in S^c} \Pro\left(\widetilde{\cA}_{4,t,j}^c \cap \mathcal{A}_{1,t}\right) + \Pro\left(  \mathcal{A}_{1,t}^c\right) \leq 1/5.
\end{align*}
\end{proof}

\section{Proofs for the Bandit Problem - Upper Bounds}
In this Appendix, for $t \in [T]$, define
\begin{equation}
    \label{def:bepsilon}
\bepsilon_{[t]}:=(\epsilon_1,\ldots,\epsilon_{t})'.
\end{equation}
Recall that for each time $t \in [T]$ and arm $k \in [K]$, we define
\begin{equation}\label{eq: bandit design matrix}
\begin{aligned}
    &\bX_{[t]}^{(k)} := \left(\bX_1 \mathbb{I}\{A_1 = k\},\; \cdots,\; \bX_t \mathbb{I}\{A_t = k\}\right)', \\
    &\bY_{[t]}^{(k)} := \left(Y_1 \mathbb{I}\{A_1 = k\},\; \cdots,\; Y_t \mathbb{I}\{A_t = k\}\right)'.
\end{aligned}
\end{equation}
For each arm $k\in [K]$ and $h>0$, let 
\begin{equation}\label{eq: U_h_k}
    \mathcal{U}_h^{(k)} := \{\bd{x}\in \mR^d: (\btheta^{(k)})'\bd{x}> \max_{j\ne k}(\btheta^{(j)})'\bd{x}+h\}.
\end{equation}
which is the collection of contexts for which the potential reward of $k$-th arm exceeds that of all other arms by more than $h$. For $t \in [T]$ and $k \in [K]$, define
\begin{equation}\label{eq: hbSigma}
\begin{split}
\bSigma^{(k)} := \Exp\left[\bX_t\bX_t'\idf\{\bX_{t} \in \mathcal{U}_{0}^{(k)}\}\right], \quad 
    \hbSigma_t^{(k)}:=\frac{1}{t}\sum_{s=1}^{t}\bX_s'\bX_s\idf\{A_s=k\},
\end{split}
\end{equation}
where $\mathcal{U}_{0}^{(k)}$ is the collection of contexts for which the potential reward of $k$-th arm exceeds that of all other arms.

Further, recall that we set the regularization and threshold parameters as
\begin{equation} \label{app_eq: bandit_paras_est}
\lambda_t = 6m_X\sigma\sqrt{\frac{\log (dT)}{t}},\quad \text{ and } \quad \hdlambda_t = 28L_3\lambda_t,
\end{equation}
and for some integers $C_{\gamma_1}$ and $C_{\gamma_2}$, set the end times of Stage 1 and 2, respectively, as
\begin{equation} \label{app_eq: end_times}
    \gamma_1 = C_{\gamma_1} \lceil(\sigma^2\vee 1)\rceil s_0 \lceil \log (dT) \rceil, \quad \text{ and } \quad \gamma_2 = C_{\gamma_2} s_0^4 \gamma_1.
\end{equation}

Recall that in Stage 1 (i.e., from time $1$ to $\gamma_1$), arms are selected uniformly at random from $[K]$. By  Lemma \ref{lemma: Stage 1_restricted}, Lemma \ref{lemma: arm optimality sample}, Lemma \ref{lemma: prediction error controlled sample},  Lemma \ref{lemma: ell_2 bound of x epsilon} and Lemma \ref{lemma: margin condition}, there exist constants
\begin{equation}
    \label{def:star_constants}
    C^* \geq 1,  \text{ and } a^*, h^*, \ell_0^*, \ell_1^* > 0 \;\text{ depending only on }\;  K,  m_X, \alpha_X,  L_3, m_{\theta}, L_4,
\end{equation}
such that the following statements hold for $\gamma_1 \geq C^*s_0\log(dT)$ and $n \geq C^* s_0\log(dT)$:
\begin{enumerate}
    \item (Lemma \ref{lemma: Stage 1_restricted}) with probability  at least $1-2K/T^4$, for each arm $k \in [K]$, 
    \begin{align}\label{stage 2:aux_eq1}
    \bv'\left(\frac{1}{\gamma_1}\sum_{t=1}^{\gamma_1} \bX_t\bX_t'  \idf\{A_t = k\}\right) \bv \ge \frac{1}{4KL_3}\|\bv\|_2^2,\quad \text{ for all } 
\bv\in \mathcal{C}(s_0,3). 
\end{align} 
    \item (Lemma \ref{lemma: arm optimality sample}) with probability at least $1-2K/T^4$,  for each arm $k \in [K]$, 
    \begin{equation}\label{stage 2:aux_eq2}
\bv'\left(\frac{1}{n}\sum_{t=1}^{n} \bX_t\bX_t'\idf\{\bX_t\in \mathcal{U}_{h^*}^{(k)}\} \right) \bv \ge \ell_0^*\|\bv\|_2^2,\quad \text{ for all } 
\bv\in \mathcal{C}(s_0,3),    
\end{equation}
where we recall the definition of $\mathcal{U}_h^{(k)}$ in \eqref{eq: U_h_k}.

    \item (Lemma \ref{lemma: prediction error controlled sample}) 
    for any $\bu \in \bR^d$ with $\|\bu\|_2 \leq \ell_1^*$,   
    with probability at least $1-2/T^4$, 
        \begin{equation}\label{stage 2:aux_eq3}
\bv'\left(\frac{1}{n}\sum_{t=1}^{n} \bX_t\bX_t'\idf\{|\bu'\bX_t|\ge h^*/2\} \right) \bv \le \frac{K-1}{K} a^* \|\bv\|_2^2,\quad \text{ for all } 
\bv\in \mathcal{C}(s_0,3)   
\end{equation}
where we define
\begin{equation}\label{eq: a}
    a^*= \min\{(\ell_0^*/K),\; 1/(4K L_3)\}.  
\end{equation} 
\item (Lemma \ref{lemma: ell_2 bound of x epsilon}) for any $t \in [T]$, with probability  at least $1-K/T^2$, for each arm $k \in [K]$
\begin{align}\label{eq: x_bepsilon_S_k}
 \|(\bX_{[t]}^{(k)})_{S_k}'\bepsilon_{[t]}/t\|_2 \le C^*  m_X\sigma\sqrt{\frac{3s_0\log (dT)}{t}}.
\end{align}
where $(\bX_{[t]}^{(k)})_{S_k}'\bepsilon_{[t]} = 0$ if $S_k$ is empty.
\item   (Lemma \ref{lemma: margin condition}) for any $\bu \in \bR^{d}$ such that $\bu \neq \bd{0}_d$,
\begin{align}
    \label{eq: margin}
    \Pro(|\bu'\bX_1|\le \tau)\le 
    \frac{C^*}{\|\bu\|_2 } \tau, \text{ for all } \tau>0.
\end{align}

\end{enumerate}

Note that the events in \eqref{stage 2:aux_eq2}, \eqref{stage 2:aux_eq3}, and \eqref{eq: margin} depend only on the properties of the i.i.d.\ random vectors $\{\bX_t: t \in [T]\}$, and not on the arm selection mechanism. Moreover, the fourth property above, concerning the event in \eqref{eq: x_bepsilon_S_k}, holds regardless of the arm selection mechanism, while the first property pertains specifically to the first stage, in which arms are selected at random. In this Appendix, we fix the constants in \eqref{def:star_constants}.

\subsection{Overall strategy for proving Theorem \ref{thm:bandit_est_accuracy}} \label{subsec:discussion_bandit_est}
Theorem \ref{thm:bandit_est_accuracy} follows directly from Lemma \ref{lemma: RE and estimation error of stage 2} and Lemma \ref{stage3: main lemma} ahead. Here, we briefly outline the strategy behind the proofs of these two lemmas.

For Lemma \ref{lemma: RE and estimation error of stage 2} concerning Stage 2, we use induction to show that for each $m \in \{1, 2, \dots, \lfloor \gamma_2 / \gamma_1 \rfloor\}$, with probability at least $1 - (6m - 2)K / T^4$, the following hold:  
(i) for each $k \in [K]$, the matrix $\bX_{[m\gamma_1]}^{(k)}$ satisfies the restricted eigenvalue (R.E.) condition in Definition \ref{def: RE} with appropriate constants (denoted by the event $\mE_m^{\textup{RE}}$); and  
(ii) for each $k \in [K]$,  
\[
\| \widehat{\btheta}_{m\gamma_1}^{(k)} - \btheta^{(k)} \|_2^2 \leq C \frac{\sigma^2 s_0 \log(dT)}{m \gamma_1}
\]  
for some constant $C > 0$ (denoted by the event  $\mE_m^{\btheta}$). Note that R.E.-type conditions are commonly used to derive $\ell_2$-estimation error bounds for Lasso estimators \cite{wainwright2019high}.

Specifically, for the base case $m = 1$, the claim holds if $C_{\gamma_1}$ in \eqref{app_eq: end_times} is chosen sufficiently large. In this case, since arms are selected uniformly at random during Stage 1, the analysis involves i.i.d.~observations. 

Next, we discuss the inductive step. On the event $\mE_m^{\textup{RE}} \cap \mE_m^{\btheta}$, we have control over the restricted eigenvalues of $\bX_{[m\gamma_1]}^{(k)}$, $k \in [K]$, as well as the estimation error of $\widehat{\btheta}_{m\gamma_1}^{(k)}$, $k \in [K]$. For $t \in (m\gamma_1, (m+1)\gamma_1]$, by Definition \ref{def:three_stages}, the estimators at time $m\gamma_1$ are used in arm selections.  Thus, the control over estimation errors allows us to show that if the covariate vectors $\{\bX_t : t \in (m\gamma_1, (m+1)\gamma_1]\}$ do not fall ``too close'' to the boundaries, that is, $\cup_{i \neq j \in [K]} \{ \bx \in \bR^d : (\btheta^{(i)})'\bx = (\btheta^{(j)})'\bx \}$, then the selected arm $A_t$ is equal to the optimal one $A_t^* := \argmax_{k \in [K]} \bX_t'\btheta^{(k)}$. Note that the optimal arms $\{A_t^* : t \in [T]\}$ are i.i.d.~random variables. Further, as discussed following Assumption \ref{assumption: covariates_bandit}, we use the assumption that $\bX_t$ has a log-concave density and bounded eigenvalues to show that the probability that $\{\bX_t : t \in (m\gamma_1, (m+1)\gamma_1]\}$ are ``close'' to the boundaries is small.  
These points allow us to show that on the event $\mE_m^{\textup{RE}} \cap \mE_m^{\btheta}$, $\mE_{m+1}^{\textup{RE}}$ holds with high probability. Finally, on $\mE_{m+1}^{\textup{RE}}$, due to the choice of regularization parameter $\lambda_{(m+1)\gamma_1}$ in \eqref{app_eq: bandit_paras_est}, we can establish upper bound on the $\ell_2$ estimation error of $\hbtheta^{(k)}_{(m+1)\gamma_1}, k \in [K]$, and thus $\mE_{m+1}^{\btheta}$ holds with high probability. Detailed arguments are provided in Appendix \ref{sec: proof of thm: regret in stage 2}.

Finally, we discuss Lemma \ref{stage3: main lemma} concerning Stage 3. As in Lemma \ref{lemma: RE and estimation error of stage 2}, we use an inductive argument with similar intuitions as before. The key difference is that our goal is to establish \textit{instance-specific} upper bounds on the $\ell_2$ estimation errors of the OPT-Lasso estimators, which, in turn, depend on upper bounds for the $\ell_{\infty}$ estimation error of the initial Lasso estimators.  For the sequential estimation problem in Section~\ref{sec: multi linear}, the observations $\{(\bX_t, Y_t): t \in [T]\}$ are i.i.d., and $\bX_t$ follows a multivariate normal distribution, for which the recent $\ell_{\infty}$ bounds in \cite{bellec2022biasing} apply when $t = \Omega(s_0 \log(dT))$.  However, for the bandit problem, due to the dependency among observations, we resort to the classical $\ell_{\infty}$ bounds in Lemma 4.1 of \cite{van2016estimation} (see also Lemma~\ref{lemma: bounds of lasso}), and can only establish the required conditions for $t = \Omega(s_0^5 \log(dT))$. This explains the gap between Stage 1 and Stage 3. 
Moreover, this requires us to include additional events in the inductive argument, such as the $\ell_{\infty}$ distance between the design matrix for the $k$-th arm and $\bSigma^{(k)}$ in \eqref{eq: hbSigma}, for $k \in [K]$, as detailed in Appendix~\ref{sec: proof of thm: regret in stage 3}.

\subsection{Proof of Theorem \ref{thm:bandit_est_accuracy}, Part (i): Stage 2} \label{sec: proof of thm: regret in stage 2}
In this subsection, we focus on Stage 2 of the procedure defined in Definition \ref{def:three_stages} and establish part (i) of Theorem \ref{thm:bandit_est_accuracy}, which follows immediately from Lemma \ref{lemma: RE and estimation error of stage 2}. 

Since this subsection focuses on Stage 2, the end time $\gamma_2$ of Stage 2 can be arbitrary, provided that $\gamma_2 \geq \gamma_1$. Recall the definition of the restricted eigenvalue condition, as well as the set $\mathcal{C}(s,\kappa)$, in Definition \ref{def: RE}.

Recall the fixed constants  $C^* \geq 1$ and $a^*, h^*, \ell_0^*, \ell_1^* > 0$  in \eqref{def:star_constants}, which only depend only on $K,  m_X, \alpha_X,  L_3, m_{\theta}, L_4$, such that \eqref{stage 2:aux_eq1}-\eqref{eq: margin} hold.

\begin{lemma} \label{lemma: RE and estimation error of stage 2}
Suppose that Assumptions \ref{assumption: context}, \ref{assumption: bounded cov mat} and \ref{assumption:arm_parameters} hold. We set the regularization parameter $\lambda_t$ as in \eqref{app_eq: bandit_paras_est}, and the end time $\gamma_1$ of Stage 1 as in \eqref{app_eq: end_times}.
If the integer $C_{\gamma_1}$ in \eqref{app_eq: end_times} satisfies the following:
\begin{align}
    \label{eq: lambda_gamma_1}
     C_{\gamma_1} \ge \max\left\{C^*,\; \left(\frac{18m_X}{\ell_1^* a^*}\right)^2\right\},
\end{align}
then for each  $m \in \{1,2,\cdots,\lfloor \gamma_2/\gamma_1 \rfloor\}$, with probability at least $1-(6m-2)K/T^4$, we have
\begin{enumerate}[label=\roman*)]
    \item for each $k \in [K]$, $\bX_{[m\gamma_1]}^{(k)}$ satisfies $\textup{RE}(s_0,3,a^*)$ ;
    \item for each $k \in [K]$,
    $$
    \|\hbtheta_{m\gamma_1}^{(k)}-\btheta^{(k)}\|_2^2\le   \frac{(18m_{X})^2}{(a^*)^2}  \frac{\sigma^2 s_0\log (dT)}{m\gamma_1}.
    $$
\end{enumerate}
\end{lemma}

\begin{proof}
For integer $m=1, \cdots,\lfloor \gamma_2/\gamma_1\rfloor$, we define the  following events:  with $t = m\gamma_1$, 
\begin{align}\label{def:stage2_events}
\begin{split}
&  \mE_m^{\textup{RE}} := \bigcap_{k = 1}^{K} \left\{\bX_{[t]}^{(k)} \textup{ satisfies } \textup{RE}(s_0,3, a^*)\right\}, \\
&\mE_m^{\btheta} = \bigcap_{k = 1}^{K}\left\{\|\hbtheta_t^{(k)}-\btheta^{(k)}\|_2^2 \le \frac{9}{(a^*)^2}s_0\lambda_t^2\right\},    \\
&\mE_m^{\bX\bepsilon} = \bigcap_{k\in [K]} \left\{\|{(\bX_{[t]}^{(k)}})' \bepsilon_{[t]}/t\|_{\infty} 
\le \frac{\lambda_t}{2} \right\}.
\end{split}
\end{align}

The proof proceeds by induction to show that $\Pro(\mE_m^{\textup{RE}} \cap \mE_m^{\btheta}) \geq 1-(6m-2)K/T^4$, for $m=1,\ldots,\lfloor \gamma_2/\gamma_1\rfloor$. 

\medskip \noindent
\textbf{Base case: m = 1.} Due to \eqref{stage 2:aux_eq1} and the definition of $a^*$ in \eqref{eq: a}, we have $\Pro\left(\mE_1^{\textup{RE}}\right) \geq 1 - 2K/T^4$. Further, due to Lemma \ref{lemma:ell_infty bound of x times epsilon} and the definition of $\lambda_{m\gamma_1}$ in \eqref{app_eq: bandit_paras_est}, 
 we have $\Pro\left(\mE_1^{\bX\bepsilon}\right) \geq 1 - 2K/T^4$. Finally, on the event $\mE_1^{\textup{RE}} \cap \mE_1^{\bX\bepsilon}$,
due to the second inequality in Lemma \ref{lemma: bounds of lasso}, we have that the event $\mE_1^{\btheta}$ holds. Thus, by the union bound,
\begin{align*}
   \Pro(\mE_1^{\btheta}\cap\mE_1^{\textup{RE}})  &\ge \Pro( \mE_1^{\textup{RE}}\cap\mE_1^{\bX\bepsilon})   \ge 1- \frac{4K}{T^4}.
\end{align*}
The base case is established.

\medskip \noindent
\textbf{Inductive step.}  Now, we assume that for $1 \leq m < \lfloor \gamma_2/\gamma_1\rfloor$,  
$\Pro(\mE_m^{\btheta}\cap\mE_m^{\textup{RE}}) \ge 1- \frac{(6m-2)K}{T^4}$. Our goal is to show that 
$\Pro(\mE_{m+1}^{\btheta}\cap\mE_{m+1}^{\textup{RE}}) \ge 1- \frac{(6m+4)K}{T^4}$.
 
On the event $\mE_{m+1}^{\textup{RE}} \cap \mE_{m+1}^{\bX\bepsilon}$,
due to the second inequality in Lemma \ref{lemma: bounds of lasso}, we have that the event $\mE_{m+1}^{\btheta}$ holds. Thus, by the union bound,
\begin{align*}
\Pro(\mE_{m+1}^{\btheta}\cap\mE_{m+1}^{\textup{RE}})  \ge \Pro( \mE_{m+1}^{\textup{RE}}\cap\mE_{m+1}^{\bX\bepsilon})   \ge 1- \Pro( (\mE_{m+1}^{\textup{RE}})^c) - \Pro((\mE_{m+1}^{\bX\bepsilon})^c). 
\end{align*}
By Lemma \ref{lemma:ell_infty bound of x times epsilon} and the definition of $\lambda_{(m+1)\gamma_1}$ in \eqref{app_eq: bandit_paras_est}, 
 we have $\Pro\left((\mE_{m+1}^{\bX\bepsilon})^c\right) \leq 2K/T^4$. Thus, it suffices to show that 
 $\Pro( (\mE_{m+1}^{\textup{RE}})^c) \leq (6m+2)K/T^4$, on which we now focus. 

\medskip
For $m\gamma_1<t\le (m+1)\gamma_1$ and $k \in [K]$, define the event 
\begin{equation}\label{eq: stage 2 RE 0}
    \mathcal{B}_t^{(k)}= \left(\bigcap_{i\in [K]}\left\{ |(\hbtheta_{m\gamma_1}^{(i)}-\btheta^{(i)})'\bX_t|\le {h^*}/{2} \right\} \right) \bigcap \left\{\bX_t\in \mathcal{U}_{h^*}^{(k)} \right\}.
\end{equation}
Due to the definition of $\mathcal{U}_h^{(k)}$ in \eqref{eq: U_h_k}, for any $j \neq k \in [K]$,    on the event $\mathcal{B}_t^{(k)}$, 
\begin{align*}
     (\hbtheta_{m\gamma_1}^{(k)}-\hbtheta_{m\gamma_1}^{(j)})'\bX_t &\ge (\btheta^{(k)}- \btheta^{(j)})'\bX_t - |(\hbtheta_{m\gamma_1}^{(k)}-\btheta^{(k)})'\bX_t| - |(\hbtheta_{m\gamma_1}^{(j)}-\btheta^{(j)})'\bX_t|
     \\&> h^*-h^*/2-h^*/2=0.
\end{align*}
 which implies that $\mathcal{B}_t^{(k)}\subseteq \{A_t=k\}$, and thus for each $\bv\in \mC(s_0,3)$,
 \begin{align*}
     \frac{\|\bX_{[(m+1)\gamma_1]}^{(k)} \bv\|_2^2}{(m+1)\gamma_1} \geq \frac{m\gamma_1}{(m+1)\gamma_1}\frac{\|\bX_{[m\gamma_1]}^{(k)} \bv\|_2^2}{m\gamma_1} + \frac{\gamma_1}{(m+1)\gamma_1} \bv'\left(\frac{1}{\gamma_1}\sum_{t=m\gamma_1+1}^{(m+1)\gamma_1} \bX_t\bX_t'\idf{\{\mathcal{B}_t^{(k)}\}} \right) \bv.
 \end{align*}

Due to the definition of $\mE_{m}^{\textup{RE}}$ in \eqref{def:stage2_events} and $\mathcal{B}_t^{(k)}$ in \eqref{eq: stage 2 RE 0}, on the event $\mE_{m}^{\textup{RE}}$, we have that for each $k \in [K]$ and $\bv\in \mC(s_0,3)$,
\begin{equation}\label{eq: stage 2 RE 1}
\begin{aligned}
   \frac{\|\bX_{[(m+1)\gamma_1]}^{(k)} \bv\|_2^2}{(m+1)\gamma_1} \geq \frac{m}{m+1} a^*\|\bv\|_2^2 + \frac{1}{m+1} \left(I_{m+1}^{(k)}(\bv) - \sum_{i \in [K]} II_{m+1}^{(i)}(\bv) \right), 
\end{aligned}
\end{equation}
where we define 
\begin{equation*} 
\begin{aligned}
& \;\; I_{m+1}^{(k)}(\bv)  :=  \bv'\left(\frac{1}{\gamma_1}\sum_{t=m\gamma_1+1}^{(m+1)\gamma_1} \bX_t\bX_t'\idf\{\bX_t\in \mathcal{U}_{h^*}^{(k)}\} \right) \bv, \\
    &II_{m+1}^{(k)}(\bv) :=   \bv'\left(\frac{1}{\gamma_1}\sum_{t=m\gamma_1+1}^{(m+1)\gamma_1} \bX_t\bX_t'\idf\{|(\hbtheta_{m\gamma_1}^{(k)}-\btheta^{(k)})'\bX_t|\ge h^*/2\} \right) \bv.
\end{aligned}
\end{equation*}

Further, we define the following two events
\begin{align*}
 &\mathcal{A}_{m+1} := \bigcap_{k \in [K]} \left\{ I_{m+1}^{(k)}(\bv) \geq \ell_0^* \|\bv\|_2^2 \; \text{ for all } 
\bv\in \mathcal{C}(s_0,3)\right\},\\
& \mathcal{A}'_{m+1}=\bigcap_{k = 1}^{K}\left\{ II_{m+1}^{(k)}(\bv) \leq  \frac{K-1}{K} a^*\|v\|_2^2\; \text{ for all } 
\bv\in \mathcal{C}(s_0,3)
\right\}.
\end{align*}
By the definition of $a^*$ in \eqref{eq: a}, $\ell_0^* \geq a^* K$. Thus, due to \eqref{eq: stage 2 RE 1} and the definition $\mE_{m+1}^{\textup{RE}}$ in \eqref{def:stage2_events},
$$\mE_{m}^{\textup{RE}} \bigcap
\mathcal{A}_{m+1} \bigcap \mathcal{A}'_{m+1}
\;\subseteq\; \mE_{m+1}^{\textup{RE}} 
$$

Finally, due to \eqref{stage 2:aux_eq2}, we have that $\Pro\left(\mathcal{A}_{m+1}\right) \geq 1 -2K/T^4$. Further, on the event $\mE_m^{\btheta}$ defined in \eqref{def:stage2_events} and $\lambda_{m\gamma_1}$ in \eqref{app_eq: bandit_paras_est}, due to \eqref{eq: lambda_gamma_1}, for each $k \in [K]$,
\begin{align*}
  \|\hbtheta_{m\gamma_1}^{(k)}-\btheta^{(k)}\|_2^2\le \frac{9}{(a^*)^2} s_0 (6m_X \sigma)^2 \frac{\log(dT)}{m \gamma_1} \leq \frac{9}{(a^*)^2} \frac{(6m_X)^2}{C_{\gamma_1}} \leq (\ell_1^*)^2.
\end{align*}
Thus, since $\{\bX_t: m\gamma_1 < t \leq (m+1)\gamma_1\}$ are independent from $\cF_{m\gamma_1}$ and $\mE_m^{\btheta} \in \cF_{m\gamma_1}$, due to \eqref{stage 2:aux_eq3} with $\bu =\hbtheta_{m\gamma_1}^{(k)}-\btheta^{(k)}$, we have that almost surely
\begin{align*}
    \Pro\left(\mathcal{A}'_{m+1} \vert \cF_{m\gamma_1} \right) \idf\{\mE_m^{\btheta}\}  \geq (1 - 2K/T^4)\idf\{\mE_m^{\btheta}\}.
\end{align*}
Combining the above results, by the union bound, we have
\begin{align*}
    \Pro\left((\mE_{m+1}^{\textup{RE}})^c\right) 
    \leq &\Pro\left( (\mE_{m}^{\textup{RE}} \cap \mE_{m}^{\btheta})^c\right)   
    +\Pro\left((\mathcal{A}_{m+1})^c\right) 
    +\Pro\left((\mathcal{A}'_{m+1})^c\; \vert    \mE_{m}^{\btheta} \cap \mE_{m}^{\textup{RE}}\right)   \\
    \leq & (6m - 2)K/T^4 + 2K/T^4 + 2K/T^4 = (6m+2)K/T^4, 
\end{align*}
where we use the inductive hypothesis on $ \Pro\left(\mE_{m}^{\textup{RE}} \cap \mE_{m}^{\btheta}\right)$ in the second step. Then the inductive step is finished, and the proof is complete.
\end{proof}

\subsection{Additional results for Stage 2}
The following lemma concerns the property of 
$\bX_{[t]}^{(k)}$ in \eqref{eq: bandit design matrix}, the Lasso estimator $\hbtheta_{t}^{(k)}$ in Subsection \ref{subsec:three_stage_algo}, and $\hbSigma^{(k)}_{t}$ in \eqref{eq: hbSigma} at the end of Stage 2, that is, for $t = \gamma_2$. Recall the definition of $\mathcal{U}_{h}^{(k)}$ in \eqref{eq: U_h_k} and $\bSigma^{(k)}$ in \eqref{eq: hbSigma}.

Recall the fixed constants  $C^* \geq 1$ and $a^*, h^*, \ell_0^*, \ell_1^* > 0$  in \eqref{def:star_constants}, which only depend only on $K,  m_X, \alpha_X,  L_3, m_{\theta}, L_4$, such that \eqref{stage 2:aux_eq1}-\eqref{eq: margin} hold.
\begin{lemma}  \label{stage3: base case}
Suppose that Assumptions \ref{assumption: covariates_bandit} and \ref{assumption:arm_parameters} hold. We set the regularization parameter $\lambda_t$ as in \eqref{app_eq: bandit_paras_est}, and the end time $\gamma_1$ and $\gamma_2$ of Stage 1 and 2, respectively, as in \eqref{app_eq: end_times}.
If the integers $C_{\gamma_1}$ and $C_{\gamma_2}$ in \eqref{app_eq: end_times} satisfy the following:
\begin{align}
    \label{eq: aux_lambda_gamma_12}
     C_{\gamma_1} \ge \max\left\{C^*,\; \left(\frac{18m_X}{\ell_1^* a^*}\right)^2,2\right\}, \;\;
     C_{\gamma_2} \ge \max\left\{4, \frac{32 m_{X}^4}{(a^*)^2}, \left(\frac{192K L_4 C^* m_X^4}{(a^*)^2} \right)^4 \right\}.
\end{align}
then, with probability at least $1-17K/T^2$, the event $\tilde{\mE}_1^{\textup{RE}} \cap \tilde{\mE}_1^{\bSigma} \cap  \tilde{\mE}_1^{\textup{Lasso}}$ occurs, where
\begin{align*}
\begin{split}
&\tilde{\mE}_1^{\textup{RE}} := \bigcap_{k=1}^{K}\left\{\bX_{[\gamma_2]}^{(k)} \textup{ satisfies } \textup{RE}(s_0, 3, a^*)\right\},\\ 
&\tilde{\mE}_1^{\bSigma} := \bigcap_{k=1}^{K}\left\{\mn\hbSigma_{\gamma_2}^{(k)}-\bSigma^{(k)}\mn_{\textup{max}} \le  \frac{3 a^*}{s_0} \right\},\\ 
&\tilde{\mE}_1^{\textup{Lasso}} = \bigcap_{k = 1}^{K}\left\{\|\hbtheta_{\gamma_2}^{(k)} - \btheta^{(k)}\|_{\infty} \leq  28L_3 \lambda_{\gamma_2}\right\}.
\end{split}
\end{align*}
\end{lemma}

\begin{proof}
Recall the events $\mE_m^{\textup{RE}}$ and $\mE_m^{\bX\bepsilon}$ in \eqref{def:stage2_events} for $m = 1,\ldots,\gamma_2/\gamma_1$, where $\gamma_2$ is now a multiple of $\gamma_1$. Define the following times and ratios
\begin{align}\label{aux:times_ratios}
    \tgamma_1 := \lceil \sqrt{C_{\gamma_2}} \rceil s_0^3\gamma_1, \quad r := \frac{\tgamma_1}{\gamma_1}, \quad R:= \frac{\gamma_2}{\gamma_1}.
\end{align}
By Lemma \ref{lemma: RE and estimation error of stage 2} and 
Lemma \ref{lemma:ell_infty bound of x times epsilon}
for each $m=1,2,\cdots, R$,   we have
\begin{align}\label{aux:from_stage2}
\begin{split}
&\Pro\left(\bigcap_{k = 1}^{K}\left\{\bX_{[m\gamma_1]}^{(k)} \textup{ satisfies } \textup{RE}(s_0,3,a^*) \right\}\right) \ge 1-(6m-2)K/T^4,\\
&\Pro\left(\bigcap_{k = 1}^{K}\left\{\|(\bX_{[m\gamma_1]}^{(k)})'\bepsilon_{[m\gamma_1]}^{(k)}/(m\gamma_1)\|_{\infty} \leq \lambda_{m\gamma_1}/2\right\}\right)\ge 1-2K/T^4.
\end{split}
\end{align}

Note that $\tilde{\mE}_1^{\textup{RE}} = \mE_{R}^{\textup{RE}}$, and thus $ \Pro\left(\tilde{\mE}_1^{\textup{RE}}\right) \geq 1 - 6K/T^2.$ Further, applying Lemma \ref{lemma: bounds of lasso}
with $\bA = \bSigma^{(k)}$, and due to Assumption \ref{assumption: constraint on the covariance matrix 1}, on the event $\mE_{R}^{\bX\bepsilon}  \cap \tilde{\mE}_1^{\textup{RE}} \cap \tilde{\mE}_1^{\bSigma}$, for each $k \in [K]$,
$$
\|\hbtheta_{\gamma_2}^{(k)} - \btheta^{(k)}\|_{\infty} \leq L_3\left(4 + \frac{8s_0}{a^*}  \times \frac{3a^*}{s_0}  \right) \lambda_{\gamma_2} = 28 L_3 \lambda_{\gamma_2}.
$$
That is, $\mE_{R}^{\bX\bepsilon} \cap \tilde{\mE}_1^{\textup{RE}} \cap \tilde{\mE}_1^{\bSigma}\; \subseteq \;  \tilde{\mE}_1^{\textup{Lasso}}$.
Thus, by the union bound, we have
\begin{align*}
   \Pro\left( \tilde{\mE}_1^{\textup{RE}} \cap \tilde{\mE}_1^{\bSigma} \cap  \tilde{\mE}_1^{\textup{Lasso}}\right) \geq 1 - \Pro((\mE_{R}^{\bX\bepsilon})^c) - \Pro((\tilde{\mE}_1^{\textup{RE}})^c) - \Pro((\tilde{\mE}_1^{\bSigma})^c) \geq 1 - \frac{8K}{T^2} - \Pro((\tilde{\mE}_1^{\bSigma})^c).
\end{align*}
As a result, it suffices to show that $\Pro((\tilde{\mE}_1^{\bSigma})^c) \leq 9K/T^2$, on which we now focus.

\medskip

Recall the definition of $\hbSigma_{\gamma_2}^{(k)}$ and $\bSigma^{(k)}$ in \eqref{eq: hbSigma}, and $\mathcal{U}_h^{(k)}$ in \eqref{eq: U_h_k}. By the triangle inequality,
\begin{align}\label{aux:based_case_strategy}
    \op{\hbSigma_{\gamma_2}^{(k)}-\bSigma^{(k)}}_{\textup{max}}&\le  I^{(k)} + II^{(k)} + III^{(k)},\;\; \text{ for } k \in [K],
\end{align}
where we define
\begin{align*}
&I^{(k)} := \frac{\tgamma_1}{\gamma_2} \op{\hbSigma_{\tgamma_1}^{(k)} -\bSigma^{(k)} }_{\textup{max}},\;\;
    II^{(k)} :=\op{\frac{1}{\gamma_2}\sum_{s=\tgamma_1+1}^{\gamma_2} \left(\bX_s\bX_s'\idf\{\bX_s \in\mathcal{U}_0^{(k)}\}-\bSigma^{(k)}\right)}_{\textup{max}},\\
    &III^{(k)} := \op{\frac{1}{\gamma_2}\sum_{s=\tgamma_1+1}^{\gamma_2}\bX_s\bX_s'\left(\idf\{A_s=k\}-\idf\{\bX_s \in \mathcal{U}_0^{(k)}\}\right)}_{\textup{max}}.
\end{align*}

For the term $I^{(k)}$, due to \eqref{eq: aux_lambda_gamma_12}, 
$\tgamma_1/\gamma_2 \leq a^*/(2m_X^2s_0)$, and thus,
by Assumption \ref{assumption: context},  
\begin{equation}\label{eq: gamma2_aux1}
I^{(k)} 
\le \frac{2\tgamma_1 m_X^2}{\gamma_2} \le \frac{a^*}{s_0},\;\; \text{ for } k \in [K].
\end{equation}

Next, we consider the term $II^{(k)}$. Due to \eqref{eq: aux_lambda_gamma_12}, we have $$
\gamma_2 - \tgamma_1 \geq \gamma_2/2 \geq (16m_X^4/(a^*)^2) s_0^5 \log(dT).
$$ 
Thus,  by  Lemma \ref{lemma: event D_2},
\begin{equation}\label{eq: gamma2_aux2}
   \Pro\left( \bigcap_{k \in [K]} \left\{ II^{(k)}  \le \frac{a^*}{s_0}\right\} \right) \leq 1 -\frac{2K}{T^2}.
\end{equation}

Furthermore, we focus on the term $III^{(k)}$. For each $t \in [T]$ and  $k\in[K]$, define
\begin{align*}
    &Z_t^{(k)}:=|\idf\{A_t=k\}-\idf\{(\btheta^{(k)})'\bX_t>\max_{\ell \neq k}((\btheta^{(\ell)})'\bX_t)\}|\in \{0,1\}, \\
        & \mathcal{B}_t^{(k)} = \bigcap_{\ell \neq k}\left\{|(\btheta^{(k)}-\btheta^{(\ell)})'\bX_t| > \frac{16m_X}{a^*} s_0\lambda_{\tgamma_1}\right\},
\end{align*}
and further
\begin{equation*}
    \mathcal{A} = \bigcap_{m=r}^{R}\bigcap_{k = 1}^{K}\biggl\{|(\hbtheta_{m\gamma_1}^{(k)}-\btheta^{(k)})'\bX_t|\le \frac{8m_X}{a^*}s_0\lambda_{m\gamma_1} \text{ for all } m\gamma_1< t \le (m+1)\gamma_1 \biggr\}.
\end{equation*}

On the event $\mathcal{A}$, for each $k \in[K]$, $m \in \{r,\cdots,R\}$ and $t \in (m\gamma_1, (m+1)\gamma_1]$, we have that for $\ell \in [K]$ and $\ell \neq k$,
\begin{align*}
&|(\hbtheta_{m\gamma_1}^{(k)}-\hbtheta_{m\gamma_1}^{(\ell)})'\bX_t-(\btheta^{(k)}-\btheta^{(\ell)})'\bX_t|
 \le \frac{16m_X}{a^*}s_0\lambda_{\tgamma_1},
\end{align*}
which implies that on the event $\mathcal{A}\cap \mathcal{B}_t^{(k)}$,
$$
\{A_t = k \} \quad \text{ if and only if} \quad
\bX_t'\btheta^{(k)}>\max_{\ell \neq k}(\bX_t'\btheta^{(\ell)}),
$$
and thus
$$
Z_t^{(k)} \idf\{\cA\}\; = \;  Z_t^{(k)} \idf\{\cA\} \idf\{(\mathcal{B}_t^{(k)})^c\} \; \leq \; \idf\{(\mathcal{B}_t^{(k)})^c\}.
$$
Therefore, due to Assumption \ref{assumption: context},
\begin{align}\label{aux:III_gamma2}
    III^{(k)} \idf\{\cA\} \leq  m_X^2 \frac{1}{\gamma_2} \sum_{t=1}^{\gamma_2} \idf\{(\mathcal{B}_t^{(k)})^c\}.
\end{align}

By Lemma \ref{lemma: bounds of lasso}, for each $m=1,\cdots,R$, on the event $\mE_m^{\textup{RE}} \cap \mE_m^{\bX\bepsilon}$,  we have that 
\begin{equation*}
    \|\hbtheta_{m\gamma_1}^{(k)}-\btheta^{(k)}\|_1\le \frac{8s_0\lambda_{m \gamma_1}}{a^*}, \quad \text{ for each } k \in [K].
\end{equation*}
Since $|(\hbtheta_{m\gamma_1}^{(k)}-\btheta^{(k)})'\bX_t|\le \|\bX_{t}\|_{\infty}\|\hbtheta_{m\gamma_1}^{(k)}-\btheta^{(k)}\|_1$, due to Assumption \ref{assumption: context}, 
\begin{align*}
    \bigcap_{m=r}^{R} \left(\mE_m^{\textup{RE}} \cap \mE_m^{\bX\bepsilon} \right) \; \subseteq \; \mathcal{A},
\end{align*}
which, in view of \eqref{aux:from_stage2}  and due to the union bound, implies that
$$\Pro\left(\mathcal{A}\right) \geq
1 - \sum_{m=r}^{R} \Pro\left( (\mE_m^{\textup{RE}} \cap \mE_m^{\bX\bepsilon})^c \right) \geq 
1 - \frac{6K}{T^2}.
$$

Further, due to Assumption \ref{assumption: linear_independent}, $\|\btheta^{(k)}-\btheta^{(\ell)}\|_2 \geq L_4^{-1}$. 
Due to \eqref{eq: margin} and the union bound, we have that for $t \in [T]$ and $k \in [K]$
\begin{equation*}  
\Pro((\mathcal{B}_t^{(k)})^{c})\le K \times C^* L_4 \frac{16m_X}{a^*} s_0  \lambda_{\tgamma_1}.
 \end{equation*}
 By the definition of $\lambda_{\tgamma_1}$ in \eqref{app_eq: bandit_paras_est} and $\tgamma_1$ in \eqref{aux:times_ratios}, 
 $\lambda_{\tgamma_1} \leq 6m_X/(C_{\gamma_2}^{1/4} s_0^2)$. Then due to \eqref{eq: aux_lambda_gamma_12}, we have
 $\Pro((\mathcal{B}_t^{(k)})^{c}) \leq a^*/(2m_X^2s_0)$ for $t \in [T]$ and $k \in [K]$. By Lemma \ref{lemma: concentration inequality of bernoulli} with $\delta = a^*/(2m_X^2s_0)$, we have that for $k \in [K]$,
\begin{equation*}
\Pro\left( \frac{1}{\gamma_2}\sum_{t=1}^{\gamma_2}\idf\{(\mathcal{B}_t^{(k)})^{c}\}\ge \frac{a^*}{m_X^2 s_0}   \right)\leq \exp\left(- \frac{3a^*}{16m_X^2s_0} \gamma_2\right) \leq \frac{1}{T^2},
\end{equation*}
where the last inequality is due to \eqref{eq: aux_lambda_gamma_12}. As a result, due to \eqref{aux:III_gamma2}, since $\Pro(\cA) \geq 1-6K/T^2$,  
\begin{align}\label{eq: gamma2_aux3}
    \Pro\left(\bigcap_{k \in [K]}  \left\{III^{(k)} \leq \frac{a^*}{s_0} \right\}\right) \geq 1-\frac{7K}{T^2}.
\end{align}

Finally, combining inequalities in \eqref{eq: gamma2_aux1}, \eqref{eq: gamma2_aux2}, and \eqref{eq: gamma2_aux3} on the three terms, in view of \eqref{aux:based_case_strategy},  
\begin{align*}
    \Pro\left(\bigcap_{k \in [K]}  \left\{\op{\hbSigma_{\gamma_2}^{(k)}-\bSigma^{(k)}}_{\textup{max}} \le \frac{3a^*}{s_0} \right\} \right) \geq 1 - \frac{9K}{T^2}.
\end{align*}
The proof is complete.
\end{proof}

\subsection{Proof of Theorem \ref{thm:bandit_est_accuracy}, Part (ii): Stage 3} \label{sec: proof of thm: regret in stage 3}
In this subsection, we focus on Stage 3 of the procedure defined in Definition \ref{def:three_stages} and establish part (ii) of Theorem \ref{thm:bandit_est_accuracy}, which follows immediately from Lemma \ref{stage3: main lemma}.

Recall the definition of $\bepsilon_{[t]}$ in \eqref{def:bepsilon} and $\mathcal{U}_h^{(k)}$ in \eqref{eq: U_h_k}.  Recall 
$\bX_{[t]}^{(k)}$ and $\bY_{[t]}^{(k)}$ in \eqref{eq: bandit design matrix}, the Lasso and OPT-Lasso estimators $\hbtheta_{t}^{(k)}$ and $\tbtheta_t^{(k)}$ in Subsection \ref{subsec:three_stage_algo}, and $\hbSigma^{(k)}_{t}$ and $\bSigma^{(k)}$ in \eqref{eq: hbSigma}.

Recall the fixed constants  $C^* \geq 1$ and $a^*, h^*, \ell_0^*, \ell_1^* > 0$  in \eqref{def:star_constants}, which only depend only on $K,  m_X, \alpha_X,  L_3, m_{\theta}, L_4$, such that \eqref{stage 2:aux_eq1}-\eqref{eq: margin} hold.
\begin{lemma}  \label{stage3: main lemma}
Suppose that Assumptions \ref{assumption: covariates_bandit} and \ref{assumption:arm_parameters} hold. We set the regularization parameter $\lambda_t,\hdlambda_t$ as in \eqref{app_eq: bandit_paras_est}, and the end times $\gamma_1$ and $\gamma_2$ of Stage 1 and 2, respectively, as in \eqref{app_eq: end_times}.
If the integers $C_{\gamma_1}$ and $C_{\gamma_2}$ in \eqref{app_eq: end_times} satisfy conditions \eqref{eq: aux_lambda_gamma_12} and \eqref{eq: aux_lambda_gamma_12_generalm}.
Then, for each $m = 1,2,\ldots, \lfloor T/\gamma_2 \rfloor$, there exists an event $\mathcal{A}_{m\gamma_2}\in \cF_{m\gamma_2}$ such that $\Pro(\mathcal{A}_{m\gamma_2})\ge 1-17K/T$, and
\begin{equation*}
    \Exp\left[\|\tbtheta_{m\gamma_2}^{(k)}-\btheta^{(k)} \|_2^2 \idf{\{\mathcal{A}_{m\gamma_2}\}} \right] \le 7L_3^2 \sum_{j\in S_k} (\btheta^{(k)}_j)^2 \idf\left\{|\btheta^{(k)}_j| \le 2\hdlambda_{m\gamma_2} \right\} + 32L_3^2 m_X^2 \sigma^2 \frac{ s_0}{m\gamma_2}.
\end{equation*}
\end{lemma}

\begin{proof}
Fix some $m \in \{1,\ldots, \lfloor T/\gamma_2 \rfloor\}$, and denote by $t := m\gamma_2$. Define
\begin{align*}
\mathcal{A}_{t} :=  \bigcap_{k=1}^{K}\left\{\mn\hbSigma_{t}^{(k)}-\bSigma^{(k)}\mn_{\textup{max}} \le  \frac{3 a^*}{s_0} \right\}  \bigcap \bigcap_{k = 1}^{K}\left\{\|\hbtheta_{t}^{(k)} - \btheta^{(k)}\|_{\infty} \leq  \hdlambda_{t}\right\}.
\end{align*}
By Lemma \ref{stage3: main lemma_second}, we have  $\Pro(\mathcal{A}_{t})\ge 1-17K/T$ since $m \leq T$ in that context.

Since $3a^* \leq 1/(2L_3)$ due to \eqref{eq: a}, by Lemma \ref{lemma: subgaussian lower eigenvalue whp} and Assumption \ref{assumption: bounded cov mat}, on the event $\mathcal{A}_{t}$, the following holds:
\begin{equation*}
 \frac{L_3^{-1}}{2} \le \lmin{(\hbSigma_{t}^{(k)})_{S_k, S_k}} \le \lmax{(\hbSigma_{t}^{(k)})_{S_k, S_k}} \le \frac{3L_3}{2} \; \text{ for } k \in [K]
\end{equation*}
 where the convention is that $\lmin{(\hbSigma_{t}^{(k)})_{S_k, S_k}} = \lmax{(\hbSigma_{t}^{(k)})_{S_k, S_k}} = 1$ if $S_k$ is empty. As a result, due to Lemma \ref{thm: ols post l2 estimation error}, on the event $\mathcal{A}_{t}$, for $k \in [K]$,
\begin{align*}
    \|\tbtheta_{t}^{(k)}-\btheta^{(k)}\|_2^2 \leq \left(6L_3^2  + 1\right) \sum_{j\in S_k} (\btheta^{(k)}_j)^2 \idf\left\{|\btheta^{(k)}_j| \le 2\hdlambda_t \right\} + 8L_3^2 \left\|\frac{(\bX_{[t]}^{(k)})_{S_k}'\bepsilon_{[t]}}{t}\right\|_2^2,
 \end{align*}
 where $(\bX_{[t]}^{(k)})_{S_k}'\bepsilon_{[t]}$ is zero if $S_k$ is empty. By Lemma \ref{lemma: ell_2 bound of x epsilon}, we have
 \begin{align*}
\Exp\left[\left\|\frac{1}{t}(\bX_{[t]}^{(k)})_{S_k}'\bepsilon_{[t]}\right\|_2^2\right] \leq \frac{4(m_X\sigma)^2 s_0}{t}.
 \end{align*}
 Then the proof is complete.
\end{proof}

\begin{lemma}  \label{stage3: main lemma_second}
Suppose that Assumptions \ref{assumption: covariates_bandit} and \ref{assumption:arm_parameters} hold. We set the regularization parameter $\lambda_t,\hdlambda_t$ as in \eqref{app_eq: bandit_paras_est}, and the end time $\gamma_1$ and $\gamma_2$ of Stage 1 and 2, respectively, as in \eqref{app_eq: end_times}.
If the integers $C_{\gamma_1}$ and $C_{\gamma_2}$ in \eqref{app_eq: end_times} satisfy condition \eqref{eq: aux_lambda_gamma_12} and
\begin{align}
    \label{eq: aux_lambda_gamma_12_generalm}
     C_{\gamma_2} \geq \max\left\{\left(\frac{4000KL_3^2L_4 m_X^4(C^*)^2}{a^*}\right)^2, \; \left( \frac{30L_3^2 m_X C^*}{\ell_1^*}\right)^2 \right\}.
\end{align}
then, for each $m = 1,2,\ldots, \lfloor T/\gamma_2 \rfloor$,  with probability at least $1-17mK/T^2$, the event $\tilde{\mE}_m^{\textup{RE}} \cap \tilde{\mE}_m^{\bSigma} \cap  \tilde{\mE}_m^{\textup{Lasso}}$ occurs, where
\begin{align} \label{eq:aux_gamma2_events_generalm}
\begin{split}
&\tilde{\mE}_m^{\textup{RE}} := \bigcap_{k=1}^{K}\left\{\bX_{[m\gamma_2]}^{(k)} \textup{ satisfies } \textup{RE}(s_0, 3, a^*)\right\},\\ 
&\tilde{\mE}_m^{\bSigma} := \bigcap_{k=1}^{K}\left\{\mn\hbSigma_{m\gamma_2}^{(k)}-\bSigma^{(k)}\mn_{\textup{max}} \le  \frac{3 a^*}{s_0} \right\},\\ 
&\tilde{\mE}_m^{\textup{Lasso}} = \bigcap_{k = 1}^{K}\left\{\|\hbtheta_{m\gamma_2}^{(k)} - \btheta^{(k)}\|_{\infty} \leq  \hdlambda_{m\gamma_2}\right\}.
\end{split}
\end{align}
\end{lemma}

\begin{proof}
For $k \in [K]$ and $t \in [T]$, define
\begin{align*} 
    \tS_t^{(k)} := \{j \in [d]: |(\hbtheta_{t}^{(k)})_j| > \hdlambda_n\},
\end{align*}
which is the estimated support by thresholding the Lasso estimator; see Definition \ref{def: OPT-Lasso}. 
For $m=1,2,\cdots, \lfloor T/\gamma_2 \rfloor$, we define the following events: with $t = m\gamma_2$,
\begin{align}\label{def:stage3_more_events}
\begin{split}
       & \tilde{\mE}_m^{\bX\bepsilon} := \bigcap_{k\in [K]} \left\{\frac{\left\|({\bX_{[t]}^{(k)}})' \bepsilon_{[t]}\right\|_{\infty}}{t} \le \frac{\lambda_t}{2},\;\;
\frac{\left\|(\bX_{[t]}^{(k)})_{S_k}'\bepsilon_{[t]}\right\|_2 }{t}\le C^*m_X\sigma\sqrt{\frac{3s_0\log (dT)}{t}}\right\},  \\
   & \tilde{\mE}_m^{\textup{OPT}} := \bigcap_{k = 1}^{K}\left\{\|\tbtheta_t^{(k)}-\btheta^{(k)}\|_2 \le \frac{a^*}{4KC^*L_4 m_X^3 s_0^2} \wedge \ell_1^*, \;\;\; \tS_t^{(k)}\subseteq S_k \right\},
      \end{split} 
\end{align}
where $(\bX_{[t]}^{(k)})_{S_k}'\bepsilon_{[t]} = 0$ if $S_k$ is empty.

By Lemma \ref{lemma:ell_infty bound of x times epsilon} and \ref{lemma: ell_2 bound of x epsilon} (see also \eqref{eq: x_bepsilon_S_k}), due to the union bound, for $m = 1,\ldots,\lfloor T/\gamma_2 \rfloor$,
\begin{equation}\label{eq: stage 3 E infty prob}
    \Pro(\tilde{\mE}_{m}^{\bX\bepsilon})\ge 1- \frac{2K}{T^4} - \frac{K}{T^2} \geq 1- \frac{3K}{T^2}.
\end{equation}

We proceed by induction. The base case is established in Lemma  \ref{stage3: base case}. Now,   assume that for some $1\leq m < \lfloor T/\gamma_2 \rfloor$, we have 
\begin{align}
    \label{stage 3:inductive_hypothesis}
\Pro\left(\tilde{\mE}_m^{\textup{RE}} \cap \tilde{\mE}_m^{\bSigma} \cap  \tilde{\mE}_m^{\textup{Lasso}}\right) \geq 1-\frac{17mK}{T^2},
\end{align}
and we show that $\Pro\left(\tilde{\mE}_{m+1}^{\textup{RE}} \cap \tilde{\mE}_{m+1}^{\bSigma} \cap  \tilde{\mE}_{m+1}^{\textup{Lasso}}\right) \geq 1-{17(m+1)K}/{T^2}$.

\medskip
\noindent \underline{\textbf{Step 1:} we deal with the event 
$\tilde{\mE}_{m}^{\textup{OPT}}$.}

On the event $\tilde{\mE}_{m}^{\bSigma}$, due to Lemma \ref{lemma: subgaussian lower eigenvalue whp} and since $3a^* \leq 1/(2L_3)$ due to \eqref{eq: a}, we have  that $\tilde{\mE}_{m}^{\textup{eigen}}$ holds, where
\begin{equation*}
    \   \tilde{\mE}_m^{\textup{eigen}} := \bigcap_{k = 1}^{K}\left\{\frac{L_3^{-1}}{2} \le \lmin{(\hbSigma_{m\gamma_2}^{(k)})_{S_k, S_k}} \le \lmax{(\hbSigma_{m\gamma_2}^{(k)})_{S_k, S_k}} \le \frac{3L_3}{2}\right\}, 
\end{equation*}
 and the convention is that $\lmin{(\hbSigma_{m\gamma_2}^{(k)})_{S_k, S_k}} = \lmax{(\hbSigma_{m\gamma_2}^{(k)})_{S_k, S_k}} = 1$ if $S_k$ is empty.

By Lemma \ref{lemma: threshold Lasso_ell_infty bound bandit}, on the event $\tilde{\mE}_{m}^{\textup{Lasso}}$, we have
\begin{align*}
\tS_{m\gamma_2}^{(k)}  \subseteq S_k, \text{ for each } k \in [K].
\end{align*}
Further, on the event $\tilde{\mE}_{m}^{\textup{Lasso}} \cap \tilde{\mE}_{m}^{\textup{eigen}}$, by Lemma \ref{thm: ols post l2 estimation error}, for each $k \in [K]$, we have
\begin{align*}
    \|\tbtheta_{m\gamma_2}^{(k)}-\btheta^{(k)}\|_2^2 \leq \left(6L_3^2  + 1\right) s_0 \left(2\hdlambda_{m\gamma_2}\right)^2 + 8L_3^2 \left\|\frac{(\bX_{[m\gamma_2]}^{(k)})_{S_k}'\bepsilon_{[m\gamma_2]}}{m\gamma_2}\right\|_2^2.
 \end{align*}

Due to the definition of $\hdlambda_t$ in \eqref{app_eq: bandit_paras_est}, on the event $\tilde{\mE}_{m}^{\textup{Lasso}} \cap \tilde{\mE}_{m}^{\textup{eigen}} \cap \tilde{\mE}_{m}^{\bX\bepsilon}$, for each $k \in [K]$,
\begin{align*}
    \|\tbtheta_{m\gamma_2}^{(k)}-\btheta^{(k)}\|_2^2 &\leq 28L_3^2 s_0 (28L_3)^2 (6m_X \sigma)^2 \frac{\log(dT)}{C_{\gamma_2} \sigma^2 s_0^5  \log(dT)}  \\
    &+ 8L_3^2  (C^* m_X \sigma)^2 \frac{3s_0\log(dT)}{ C_{\gamma_2} \sigma^2 s_0^5  \log(dT)},
 \end{align*}
 which, due to \eqref{eq: aux_lambda_gamma_12_generalm}, implies that
 \begin{align*}
    \|\tbtheta_{m\gamma_2}^{(k)}-\btheta^{(k)}\|_2  \leq  \frac{a^*}{4KC^*L_4 m_X^3 s_0^2} \wedge  \ell_1^*, \text{ for each } k \in[K].
 \end{align*}
 Since $\tilde{\mE}_{m}^{\bSigma} \subseteq \tilde{\mE}_{m}^{\textup{eigen}}$, we have
\begin{align}\label{aux:stage3_OPT}
   \tilde{\mE}_{m}^{\textup{Lasso}} \cap \tilde{\mE}_{m}^{\bSigma}  \cap \tilde{\mE}_{m}^{\bX\bepsilon} \subseteq \tilde{\mE}_{m}^{\textup{OPT}}.
\end{align}

\medskip

\noindent \underline{\textbf{Step 2:} we deal with the events 
$\tilde{\mE}_{m+1}^{\textup{Lasso}}$.}

Applying Lemma \ref{lemma: bounds of lasso}
with $\bA = \bSigma^{(k)}$, and due to Assumption \ref{assumption: constraint on the covariance matrix 1}, on the event $\tilde{\mE}_{m+1}^{\textup{RE}} \cap \tilde{\mE}_{m+1}^{\bSigma} \cap \tilde{\mE}_{m+1}^{\bX\bepsilon}$, for each $k \in [K]$,
$$
\|\hbtheta_{(m+1)\gamma_2}^{(k)} - \btheta^{(k)}\|_{\infty} \leq L_3\left(4 + \frac{8s_0}{a^*}  \times \frac{3a^*}{s_0}  \right) \lambda_{(m+1)\gamma_2} = \hdlambda_{m\gamma_2},
$$
which implies that
\begin{align}
    \label{aux:stage3_Lasso_event}
\tilde{\mE}_{m+1}^{\textup{RE}} \cap \tilde{\mE}_{m+1}^{\bSigma} \cap \tilde{\mE}_{m+1}^{\bX\bepsilon} \; \subseteq \;
      \tilde{\mE}_{m+1}^{\textup{Lasso}}.
\end{align}

\medskip
\noindent \underline{\textbf{Step 3:} we deal with the events 
$\tilde{\mE}_{m+1}^{\textup{RE}}$.} 

By an argument similar to that in the proof of Lemma \ref{lemma: RE and estimation error of stage 2} (see \eqref{eq: stage 2 RE 1}), with $\gamma_1$ and $\hbtheta_t^{(k)}$ replaced by $\gamma_2$ and $\tbtheta_t^{(k)}$, respectively, we have that on the event $\tilde{\mE}_{m}^{\textup{RE}}$, for each $k \in [K]$ and $\bv\in \mC(s_0,3)$,
\begin{equation*}
\begin{aligned}
   \frac{\|\bX_{[(m+1)\gamma_2]}^{(k)} \bv\|_2^2}{(m+1)\gamma_2} \geq \frac{m}{m+1} a^*\|\bv\|_2^2 + \frac{1}{m+1} \left(\widetilde{I}_{m+1}^{(k)}(\bv) - \sum_{i \in [K]} \widetilde{II}_{m+1}^{(i)}(\bv) \right), 
\end{aligned}
\end{equation*}
where we define 
\begin{equation*} 
\begin{aligned}
\begin{aligned}
& \;\; \widetilde{I}_{m+1}^{(k)}(\bv)  :=  \bv'\left(\frac{1}{\gamma_2}\sum_{t=m\gamma_2+1}^{(m+1)\gamma_2} \bX_t\bX_t'\idf\{\bX_t\in \mathcal{U}_{h^*}^{(k)}\} \right) \bv, \\
    &\widetilde{II}_{m+1}^{(k)}(\bv) :=   \bv'\left(\frac{1}{\gamma_2}\sum_{t=m\gamma_2+1}^{(m+1)\gamma_2} \bX_t\bX_t'\idf\{|(\tbtheta_{m\gamma_2}^{(k)}-\btheta^{(k)})'\bX_t|\ge h^*/2\} \right) \bv.
\end{aligned}
\end{aligned}
\end{equation*}

Further, we define the following two events
\begin{align*}
 &\widetilde{\mathcal{A}}_{m+1} := \bigcap_{k \in [K]} \left\{ \widetilde{I}_{m+1}^{(k)}(\bv) \geq \ell_0^* \|\bv\|_2^2 \; \text{ for all } 
\bv\in \mathcal{C}(s_0,3)\right\},\\
& \widetilde{\mathcal{A}}'_{m+1}=\bigcap_{k = 1}^{K}\left\{ \widetilde{II}_{m+1}^{(k)}(\bv) \leq  \frac{K-1}{K} a^*\|v\|_2^2\; \text{ for all } 
\bv\in \mathcal{C}(s_0,3)
\right\}.
\end{align*}

Due to \eqref{stage 2:aux_eq2} and since $\gamma_2 \geq \gamma_1$, we have that $\Pro\left(\tilde{\mathcal{A}}_{m+1}\right) \geq 1 -2K/T^4$. Further, on the event $\tilde{\mE}_{m}^{\textup{OPT}}$ defined in \eqref{def:stage3_more_events}, for each $k \in [K]$, $\|\tbtheta_{m\gamma_2}^{(k)}-\btheta^{(k)}\|_2 \leq  \ell_1^*$. Thus, since $\{\bX_t: m\gamma_2 < t \leq (m+1)\gamma_2\}$ are independent from $\cF_{m\gamma_2}$ and $\tilde{\mE}_m^{\textup{OPT}} \in \cF_{m\gamma_2}$, due to \eqref{stage 2:aux_eq3} with $\bu =\tbtheta_{m\gamma_2}^{(k)}-\btheta^{(k)}$, we have that almost surely,
\begin{align*}
    \Pro\left(\tilde{\mathcal{A}}'_{m+1} \vert \cF_{m\gamma_2} \right) \idf\{\tilde{\mE}_{m}^{\textup{OPT}}\}  \geq (1 - 2K/T^4)\idf\{\tilde{\mE}_{m}^{\textup{OPT}}\}.
\end{align*}
By the definition of $a^*$ in \eqref{eq: a}, $\ell_0^* \geq a^* K$. Thus, due to \eqref{eq: stage 2 RE 1} and the definition $\tilde{\mE}_{m+1}^{\textup{RE}}$ in \eqref{def:stage2_events},
\begin{equation}
    \label{aux:stage3_RE}
\tilde{\mE}_{m}^{\textup{RE}} \bigcap
\tilde{\mathcal{A}}_{m+1} \bigcap \tilde{\mathcal{A}}'_{m+1}
\;\subseteq\; \tilde{\mE}_{m+1}^{\textup{RE}},
\end{equation}
which, due to \eqref{aux:stage3_OPT} and the union bound, implies that
\begin{align}
        \label{aux:stage3_step3}
        \begin{split}
            &\Pro\left(( \tilde{\mE}_{m+1}^{\textup{RE}})^c\; \cap \; (\tilde{\mE}_{m}^{\textup{RE}} \cap 
            \tilde{\mE}_{m}^{\bSigma} \cap
            \tilde{\mE}_{m}^{\textup{Lasso}} \cap    \tilde{\mE}_{m}^{\bX\bepsilon})\right) \\
\leq & \Pro\left(\tilde{\mathcal{A}}_{m+1}^c\right) + \Pro\left((\tilde{\mathcal{A}}'_{m+1})^c\; \cap \;{\tilde{\mE}_{m}^{\textup{OPT}}} \right) \leq 4K/T^2.
        \end{split}
\end{align}

\medskip
\noindent \underline{\textbf{Step 4:} we deal with the events 
$\tilde{\mE}_{m+1}^{\bSigma}$.}  

Recall the definition of $\hbSigma_{t}^{(k)}$ and $\bSigma^{(k)}$ in \eqref{eq: hbSigma}, and $\mathcal{U}_h^{(k)}$ in \eqref{eq: U_h_k}. By the triangle inequality,
\begin{align}\label{aux:stage3_case_strategy}
    \op{\hbSigma_{(m+1)\gamma_2}^{(k)}-\bSigma^{(k)}}_{\textup{max}}&\le  \overline{I}^{(k)} + \overline{II}^{(k)} + \overline{III}^{(k)},\;\; \text{ for } k \in [K],
\end{align}
where we define $\overline{I}^{(k)} := \frac{m}{m+1} \op{\hbSigma_{m\gamma_2}^{(k)} -\bSigma^{(k)} }_{\textup{max}}$, and
\begin{align*}
&
    \overline{II}^{(k)} :=\frac{1}{m+1}\op{\frac{1}{\gamma_2}\sum_{s=m \gamma_2+1}^{(m+1)\gamma_2} \left(\bX_s\bX_s'\idf\{\bX_s \in\mathcal{U}_0^{(k)}\}-\bSigma^{(k)}\right)}_{\textup{max}},\\
    &\overline{III}^{(k)} := \frac{1}{m+1} \op{\frac{1}{\gamma_2}\sum_{s=m \gamma_2+1}^{(m+1)\gamma_2}\bX_s\bX_s'\left(\idf\{A_s=k\}-\idf\{\bX_s \in \mathcal{U}_0^{(k)}\}\right)}_{\textup{max}}.
\end{align*}

On the event $\tilde{\mE}_m^{\bSigma}$ defined in \eqref{eq:aux_gamma2_events_generalm}, for each, $k\in[K]$
\begin{equation} \label{aux:stage3_I}
  \overline{I}^{(k)} := \frac{m}{m+1} \op{\hbSigma_{m\gamma_2}^{(k)} -\bSigma^{(k)} }_{\textup{max}} \leq   \frac{m}{m+1}  \frac{3 a^*}{s_0}, \quad \text{ for } k \in [K].
\end{equation}

Next, we consider the term $\overline{II}^{(k)}$. Due to \eqref{eq: aux_lambda_gamma_12}, $\gamma_2  \geq (16m_X^4/(a^*)^2) s_0^5 \log(dT)$. Thus,  by  Lemma \ref{lemma: event D_2},
\begin{equation}\label{aux:stage3_II}
   \Pro\left( \bigcap_{k \in [K]} \left\{ \overline{II}^{(k)}  \le \frac{a^*}{s_0}\right\} \right) \geq 1 -\frac{2K}{T^2}.
\end{equation}

Furthermore, we focus on the term $\overline{III}^{(k)}$. For each $t \in [T]$ and  $k\in[K]$, define
\begin{align*}
    &Z_t^{(k)}:=|\idf\{A_t=k\}-\idf\{(\btheta^{(k)})'\bX_t>\max_{\ell \neq k}((\btheta^{(\ell)})'\bX_t)\}|\in \{0,1\},\\
     & \tilde{\mathcal{B}}_t^{(k)} := \bigcap_{\ell \neq k}\left\{|(\btheta^{(k)}-\btheta^{(\ell)})'\bX_t| > \frac{a^*}{ 2KC^*L_4 m_X^2 s_0}\right\},\\
&     \tilde{\mathcal{A}} := \bigcap_{k\in [K]} \bigcap_{t = m\gamma_2 + 1}^{(m+1)\gamma_2}\left\{|(\tbtheta_{m\gamma_2}^{(k)}-\btheta^{(k)})'\bX_t| \le \frac{a^*}{ 4KC^*L_4 m_X^2 s_0}\right\},
\end{align*}
On the event $\tilde{\mathcal{A}}$, for $t \in (m\gamma_2+1,(m+1)\gamma_2]$ and $k \in [K]$, we have 
\begin{align*}
&|(\hbtheta_{m\gamma_1}^{(k)}-\hbtheta_{m\gamma_1}^{(\ell)})'\bX_t-(\btheta^{(k)}-\btheta^{(\ell)})'\bX_t|
 \le  \frac{a^*}{ 2KC^*L_4 m_X^2 s_0}.
\end{align*}
Therefore on the event $\tilde{\mathcal{A}}\cap \tilde{\mathcal{B}}_t^{(k)}$, $\{A_t = k\}$ if and only if
$(\btheta^{(k)})'\bX_t>\max_{\ell \neq k}((\btheta^{(\ell)})'\bX_t$, which implies that
\begin{equation}\label{aux:stage3_Z_t}
    Z_t^{(k)} \idf\{\tilde{\cA}\} = Z_t^{(k)} \idf\{\tilde{\cA}\} \idf\{(\tilde{\mathcal{B}}_t^{(k)})^c\} \leq \idf\{(\tilde{\mathcal{B}}_t^{(k)})^c\}.
\end{equation}

Further, due to Assumption \ref{assumption: linear_independent}, $\|\btheta^{(k)}-\btheta^{(\ell)}\|_2 \geq L_4^{-1}$. 
Due to \eqref{eq: margin} and the union bound, we have
\begin{equation*}  
\Pro((\tilde{\mathcal{B}}_t^{(k)})^c)\le K \times C^* L_4  \frac{a^*}{ 2KC^*L_4 m_X^2 s_0} = \frac{a^*}{2m_X^2 s_0}.
 \end{equation*}
 By Lemma \ref{lemma: concentration inequality of bernoulli} with $\delta = a^*/(2m_X^2s_0)$ and due to \eqref{eq: aux_lambda_gamma_12}, we have
\begin{equation*}
\Pro\left( \frac{1}{\gamma_2}\sum_{s=m \gamma_2+1}^{(m+1)\gamma_2} \idf{\{(\tilde{\mathcal{B}}_t^{(k)})^c\}}\ge \frac{a^*}{m_X^2 s_0}   \right)\leq \exp\left(- \frac{3a^*}{16m_X^2s_0} \gamma_2\right) \leq \frac{1}{T^2}.
\end{equation*}
Thus, due to \eqref{aux:stage3_Z_t} and Assumption \ref{assumption: context}, for each $k \in [K]$,
\begin{equation}\label{aux:stage3_III}
    \Pro\left( \overline{III}^{(k)} \idf\{\tilde{\cA}\} \ge \frac{a^*}{ s_0} \right) \leq \Pro\left( \frac{1}{\gamma_2}\sum_{s=m \gamma_2+1}^{(m+1)\gamma_2} \idf{\{(\tilde{\mathcal{B}}_t^{(k)})^c\}}\ge \frac{a^*}{m_X^2 s_0}   \right)\leq   \frac{1}{T^2}. 
\end{equation}

Finally, we consider the event $\tilde{\cA}$.
For $k\in[K]$ and $m\gamma_2< t \le (m+1)\gamma_2$,  due to Assumption \ref{assumption: context}, on the event $\tilde{\mE}_m^{\textup{OPT}}$ defined in \eqref{def:stage3_more_events} (in particular, $\tS_{m\gamma_2}^{(k)}\subseteq S_k$ ), we have
\begin{align*}
   & |(\tbtheta_{m\gamma_2}^{(k)}-\btheta^{(k)})'\bX_t| = |(\tbtheta_{m\gamma_2}^{(k)}-\btheta^{(k)})_{S_k}' (\bX_t)_{S_k}| \le \|(\bX_t)_{S_k}\|_2 \|(\tbtheta_{m\gamma_2}^{(k)}-\btheta^{(k)})_{S_k} \|_2
    \\&= \|(\bX_t)_{S_k}\|_2 \|\tbtheta_{m\gamma_2}^{(k)}-\btheta^{(k)} \|_2
    \le \sqrt{s_0} m_X \frac{a^*}{4KC^*L_4 m_X^3 s_0^2}\le   \frac{a^*}{4KC^*L_4 m_X^2 s_0}.
\end{align*}
Thus,   $ \tilde{\mE}_m^{\textup{OPT}} \subseteq \tilde{\mathcal{A}}$.  Due to \eqref{aux:stage3_OPT}, we have
$$
\tilde{\mE}_{m}^{\textup{Lasso}} \cap \tilde{\mE}_{m}^{\bSigma}  \cap \tilde{\mE}_{m}^{\bX\bepsilon}  \subseteq \tilde{\cA}
$$

Due to \eqref{aux:stage3_I},  \eqref{aux:stage3_II} and  \eqref{aux:stage3_III}. in view of \eqref{aux:stage3_case_strategy}, we have
\begin{align}\label{aux:stage3_step4}
\begin{split}
        &\Pro\left( (\tilde{\mE}_{m+1}^{\bSigma})^c\; \cap \;  (\tilde{\mE}_{m}^{\textup{Lasso}} \cap \tilde{\mE}_{m}^{\bSigma}  \cap \tilde{\mE}_{m}^{\bX\bepsilon} )  \right) \\
    \leq
    &\Pro\left( \bigcup_{k \in [K]} \left\{ \overline{II}^{(k)}  \geq \frac{a^*}{s_0}\right\} \right) + \Pro\left(\bigcup_{k \in [K]} \left\{\overline{III}^{(k)} \idf\{\tilde{\cA}\} \geq\frac{a^*}{ s_0}\right\}\right) \leq \frac{3K}{T^2}.
    \end{split}
 \end{align}

\medskip
\noindent
\underline{\textbf{Final step:} combining the above steps.}
Due to \eqref{aux:stage3_Lasso_event}, we have
\begin{align*}
(\tilde{\mE}_{m+1}^{\textup{RE}}\cap\tilde{\mE}_{m+1}^{\bSigma}\cap \tilde{\mE}_{m+1}^{\textup{Lasso}})^c\;    \subseteq  \;(\tilde{\mE}_{m+1}^{\textup{RE}})^c\; \cup \; (\tilde{\mE}_{m+1}^{\bSigma})^c\; \cup\; (\tilde{\mE}_{m+1}^{\bX\bepsilon})^c.
\end{align*}
By the union bound, $\Pro\left((\tilde{\mE}_{m+1}^{\textup{RE}}\cap\tilde{\mE}_{m+1}^{\bSigma}\cap \tilde{\mE}_{m+1}^{\textup{Lasso}})^c\right)$ is upper bounded by
\begin{align*}
&    \Pro\left((\tilde{\mE}_{m}^{\textup{RE}}\cap\tilde{\mE}_{m}^{\bSigma}\cap \tilde{\mE}_{m}^{\textup{Lasso}})^c\right) + \Pro\left((\tilde{\mE}_{m}^{\bX\bepsilon})^c\right)+
    \Pro\left((\tilde{\mE}_{m+1}^{\bX\bepsilon})^c\right)+ \\
    & \Pro\left((\tilde{\mE}_{m+1}^{\textup{RE}})^c\; \cap \; (\tilde{\mE}_{m}^{\textup{RE}}\cap\tilde{\mE}_{m}^{\bSigma}\cap \tilde{\mE}_{m}^{\textup{Lasso}} \cap \tilde{\mE}_{m}^{\bX\bepsilon}) \right)
    +  \Pro\left((\tilde{\mE}_{m+1}^{\bSigma})^c\; \cap \; (\tilde{\mE}_{m}^{\bSigma}\cap \tilde{\mE}_{m}^{\textup{Lasso}}\cap \tilde{\mE}_{m}^{\bX\bepsilon}) \right).
\end{align*}
By the inductive hypothesis in \eqref{stage 3:inductive_hypothesis}, and inequalities in \eqref{eq: stage 3 E infty prob}, \eqref{aux:stage3_step3} and \eqref{aux:stage3_step4}, we have
\begin{align*}
\Pro\left((\tilde{\mE}_{m+1}^{\textup{RE}}\cap\tilde{\mE}_{m+1}^{\bSigma}\cap \tilde{\mE}_{m+1}^{\btheta})^c\right)
    \leq \frac{17mK}{T^2} +     \frac{6K}{T^2} + \frac{4K}{T^2} + \frac{3K}{T^2} \leq \frac{17(m+1)K}{T^2},
\end{align*}
which finishes the inductive step, and the proof is complete.
\end{proof}

\subsection{Supporting lemmas}
Recall that for each $k\in[K]$, $S_k$ is the support of $\btheta^{(k)}$, and $s_0 = \max_{k \in [K]} |S_k| \geq 1$. Further, recall the notations in Subsection \ref{subsec:bandit_minimax}, and the definition of $ \mathcal{C}(s, \kappa)$ in Definition \ref{def: RE}. 
The following lemmas either do not depend on the arm selection mechanism or hold for any such mechanism.

\begin{lemma} \label{lemma: Stage 1_restricted}
Let $\{B_t, t \in [n]\}$ be a sequence of i.i.d.~Bernoulli random variables,  independent from any other random variables, with $\Exp[B_1] = 1/K$.
Suppose that Assumptions \ref{assumption: context}  and \ref{assumption: bounded cov mat} hold.   Then, there exists a constant  $C >0$ depending only on $\alpha_X, K,  L_3$, such that for any $n\ge Cs_0\log (dT)$, with probability at least $1-2/T^4$, we have
\begin{align*}
    \bv'\left(\frac{1}{n}\sum_{t=1}^{n} B_t\bX_t\bX_t'  \right) \bv \ge \frac{1}{4KL_3}\|\bv\|_2^2,\quad \text{ for all } 
\bv\in \mathcal{C}(s_0,3). 
\end{align*} 
\end{lemma}
\begin{proof} 
Due to Assumption \ref{assumption: context}, for any $\bu \in \bR^{d}$ with $\|\bu\|_2=1$, we have $\|\bu' \bX_t B_t\|_{\psi_2} \leq \alpha_X$ for $t \in [n]$. Further, due to independence,
we have $\Exp\left[\bX_t\bX_t'B_t\right] = \Exp[\bX_t\bX_t'] \Exp[B_t] = \bSigma/K,
$
which, due to Assumption \ref{assumption: bounded cov mat}, implies that
$$
\frac{1}{K L_3}\le \lmin{\Exp[\bX_t\bX_t'B_t]}\le \lmax{\Exp[\bX_t\bX_t'B_t]} \le \frac{L_3}{K}
$$ Applying 
Lemma \ref{lemma: RE between population and design}  to the random matrix $(B_1\bX_1,\ldots, B_t\bX_t)'$, the proof is complete.
\end{proof}

Recall the definition of $\mathcal{U}_h^{(k)}$ in \eqref{eq: U_h_k}.

\begin{lemma}\label{lemma: arm optimality sample} 
Suppose that Assumptions \ref{assumption: context}, \ref{assumption: bounded cov mat} and \ref{assumption:arm_parameters} hold.   Then there exist constants $h,\ell_0>0$ depending only on $K,  L_3, m_{\theta}, L_4$ and $C >0$ depending only on $\alpha_X, K,  L_3, m_{\theta}, L_4$, such that for any $n\ge Cs_0\log d$, with probability at least $1-2Ke^{-n/C}$, we have
\begin{equation*}
\bv'\left(\frac{1}{n}\sum_{t=1}^{n} \bX_t\bX_t'\idf\{\bX_t\in \mathcal{U}_{h}^{(k)}\} \right) \bv \ge \ell_0\|\bv\|_2^2,\quad \text{ for all } 
\bv\in \mathcal{C}(s_0,3) \text{ and } k \in [K].    
\end{equation*}

\end{lemma}
\begin{proof}
By Lemma \ref{lemma: arm optimality} and due to Assumption \ref{assumption: bounded cov mat}, there exist constants $h,\ell_0>0$, which depend only on $K, L_3, m_{\theta}, L_4$, such that for $t \in [T]$ and $k \in [K]$,
$$
\ell_0 \leq 
\lmin{\Exp\left[\bX_t\bX_t'\idf\{\bX_t\in\mathcal{U}_{h}^{(k)}\} \right]} \leq 
\lmax{\Exp\left[\bX_t\bX_t'\idf\{\bX_t\in\mathcal{U}_{h}^{(k)}\} \right]} \leq L_3.
$$
The proof is then complete by applying Lemma \ref{lemma: RE between population and design} to
the random matrix $(\bX_1\idf\{\bX_1\in \mathcal{U}_{h}^{(k)}\},\ldots, \bX_t\idf\{\bX_t\in \mathcal{U}_{h}^{(k)}\})'$ for each $k \in [K]$ and the union bound.
\end{proof}

\begin{lemma}\label{lemma: prediction error controlled sample}
Suppose that Assumptions \ref{assumption: context}, \ref{assumption: bounded cov mat} and \ref{assumption:arm_parameters} hold.  Then for any constants $h, \ell_2>0$, there exist a constant $\ell_1>0$ depending only on $h,\ell_2,L_3$ and a constant $C>0$ depending only on $\ell_2, \alpha_X,L_3$, such that for all $n\ge Cs_0\log d$ and  $\bu\in \mR^d$ with $\|\bu\|_2\le \ell_1$,  with probability at least $1-2e^{-n/C}$, we have  
$$
\bv'\left(\frac{1}{n}\sum_{t=1}^{n} \bX_t\bX_t'\idf\{|\bu'\bX_t|\ge h/2\} \right) \bv \le \ell_2\|\bv\|_2^2,\quad \text{ for all } 
\bv\in \mathcal{C}(s_0,3).  
$$
\end{lemma}
\begin{proof} 
Let $\tilde{\ell}_2 > 0$ be the unique real number such that $\ell_2 = (9/4)[\tilde{\ell}_2 + 3(\tilde{\ell}_2)^2]$, and define $\tau := \tilde{\ell}_2/\sqrt{L_3}$. For $t \in [n]$ and $\bu \in \bR^{d}$, define $\bZ_{t,\bu} := \bX_t(\idf\{|\bu'\bX_t|\ge h/2\}+\tau)$. Then
\begin{equation}\label{eq: arm optimality sample 1} 
    \bZ_{t,\bu}\bZ_{t,\bu}' = (1+2\tau)\bX_t\bX_t'\idf\{|\bu'\bX_t|\ge h/2\} + \tau^2\bX_t\bX_t'.
\end{equation}
By Assumption \ref{assumption: bounded cov mat} we have
$$\lmin{\Exp[\bZ_{t,\bu}\bZ_{t,\bu}']}\ge \tau^2\lmin{\Exp[\bX_t\bX_t']}\ge \tau^2 L_3^{-1}.$$
By Lemma \ref{lemma: log_concave_eig_max}, there exists a constant $\ell_1 > 0$ depending only on $\ell_2,h,L_3$ such that
\begin{align*}
    \lmax{\Exp[\bX_t\bX_t'\idf\left\{|\bu'\bX_t|\ge h/2 \right\}]} \le \tilde{\ell}_2,\; \text{ for all } \bu \in \mR^d \text{ such that } \|\bu\|_2\le \ell_1,
\end{align*}
which, together with Assumption \ref{assumption: bounded cov mat}, implies that for $\bu \in \mR^d \text{ such that } \|\bu\|_2\le \ell_1$, we have
\begin{align*}
    \lmax{\Exp[\bZ_{t,\bu}\bZ_{t,\bu}']}\le  (1+2\tau) \tilde{\ell}_2 +\tau^2 L_3 \leq (4/9)\ell_2.
\end{align*}

Due to Assumption \ref{assumption: context},
$\|\bd{\omega}'\bZ_{t,\bu}\|_{\psi_2} \leq (1+\tau)\alpha_X$ for any $\bd{\omega} \in \bR^d$ with $\|\bd{\omega}\|_2 = 1$. Then, by Lemma \ref{lemma: RE between population and design}, there exists some constant $C > 0$, depending only on $\alpha_X, \ell_2, L_3$, such that for all $n\ge Cs_0\log d$ and  $\bu\in \mR^d$ with $\|\bu\|_2\le \ell_1$, with probability at least $1-2e^{-n/C}$,  we have
$$
\bv'\left(\frac{1}{n}\sum_{t=1}^n\bZ_{t,\bu}\bZ_{t,\bu}'\right)\bv \le \frac{9}{4}\frac{4\ell_2}{9}\|\bv\|_2^2 = \ell_2\|\bv\|_2^2, \;\;\text{ for any }\bv\in\mathcal{C}(s_0,3).
$$
The proof is then finished in view of \eqref{eq: arm optimality sample 1}.
\end{proof}

Recall the definition of $\bSigma^{(k)}$ in \eqref{eq: hbSigma}.

\begin{lemma}\label{lemma: event D_2}
Let $\tau > 0$ be arbitrary. Suppose  Assumption \ref{assumption: context} holds. If 
$$n \geq \frac{16m_X^4}{\tau^2} s_0^2 \log(dT),$$ 
then with probability at least $1-2K/T^2$, for each $k \in [K]$, 
$$\op{\frac{1}{n}\sum_{t=1}^{n} \left(\bX_t\bX_t'\idf\{(\btheta^{(k)})'\bX_t>\max_{\ell \neq k}((\btheta^{(\ell)})'\bX_t)\}-\bSigma^{(k)}\right) }_{\textup{max}} \le \frac{\tau}{s_0}.$$
\end{lemma}
\begin{proof}
Note that by Assumption \ref{assumption: context} we have 
$$\op{\bX_t\bX_t'\idf\{(\btheta^{(k)})'\bX_t>\max_{\ell \neq k}((\btheta^{(\ell)})'\bX_t)\}-\bSigma^{(k)}}_{\textup{max}} \le 2m_X^2.$$
Therefore, by Lemma \ref{lemma: max inequality of bounded mat}, if $n \geq (16m_X^4/\tau^2) s_0^2 \log(dT)$, then  
\begin{equation*}
   \Pro\left(  \op{\frac{1}{n}\sum_{t=1}^{n} (\bX_t\bX_t'\idf\{(\btheta^{(k)})'\bX_t>\max_{\ell \neq k}((\btheta^{(\ell)})'\bX_t)\}-\bSigma^{(k)}) }_{\textup{max}} 
     \ge \frac{\tau}{s_0} \right) \leq \frac{2}{T^2}.
\end{equation*}
Then the proof is complete by the union bound.
\end{proof}

Recall the definition of $\bX_{[t]}^{(k)}$ in \eqref{eq: bandit design matrix} and $\bepsilon_{[t]}$ in \eqref{def:bepsilon}. Recall that we say a random variable $Z$ is $\alpha$-sub-Gaussian if $\Exp[\exp(\tau Z)] \leq \exp(\alpha^2\tau^2/2)$ for $\tau \in \bR$, and that $\epsilon_1$ is $\sigma$-sub-Gaussian. 

\begin{lemma} \label{lemma:ell_infty bound of x times epsilon}
Suppose that $\|\bX_1\|_{\infty}\le m_X$ and let $t \in [T]$. Then for each arm $k\in[K]$ and any $\delta > 0$,
$$\Pro\left(\|({\bX_{[t]}^{(k)}})' \bepsilon_{[t]}/t\|_{\infty} \ge \delta \right) \le 2d\exp\left\{-\frac{t\delta^2}{2(m_X\sigma)^2} \right\}.$$
Consequently,
\begin{equation*}
    \Pro\left(\bigcap_{k\in [K]}\left\{\|{(\bX_{[t]}^{(k)}})' \bepsilon_{[t]}/t\|_{\infty} 
    \le 3m_X\sigma\sqrt{\frac{\log (dT)}{t}} \right\} \right) \ge 1-\frac{2K}{T^4}.
\end{equation*}
\end{lemma}
\begin{proof}
Fix $k \in [K]$, $j \in [d]$, and $t \in [T]$. By definition,
$$
(({\bX_{[t]}^{(k)}})' \bepsilon_{[t]})_j = \sum_{s=1}^t (\bX_s)_j\idf\{A_s=k\}\epsilon_s.
$$ 
Clearly, $(({\bX_{[t]}^{(k)}})' \bepsilon_{[t]})_j \in \cF_{t}$. Further, for $s \in [1,t]$, since $\epsilon_s$ is independent from $\bX_s$ and $\cF_{s-1}$, and $\epsilon_s$  is $\sigma$-sub-Gaussian, we have that,  for any $\tau > 0$, almost surely, 
\begin{align*}
&
\Exp\left[\epsilon_s \mid \cF_{s-1}\cup\{\bX_s\}\right] = 0. \\
&
\Exp\left[\exp\left(\tau(\bX_s)_j\idf\{A_s=k\}\epsilon_s\right)\mid \cF_{s-1} \cup \{\bX_s\}\right]
\leq 
\exp\left\{\frac{1}{2}\left(\tau(\bX_s)_j\idf\{A_s=k\}\right)^2\sigma^2 \right\}.
\end{align*}
Thus, $\{(\bX_s)_j\idf\{A_s=k\}\epsilon_s: 1\le s\le t\}$ is an $\{\cF_s\}$-martingale difference sequence. Further, since $\|\bX\|_{\infty} \leq m_X$, we have that for $\tau > 0$,
\begin{align*}
    \Exp\left[\exp\left(\tau(\bX_s)_j\idf\{A_s=k\}\epsilon_s\right) \mid \cF_{s-1}\right]  \le e^{\tau^2 m_X^2 \sigma^2/2}.
\end{align*}

Finally, by Theorem 2.19 in \cite{wainwright2019high}, 
\begin{align*}
\Pro\left(|(({\bX_{[t]}^{(k)}})' \bepsilon_{[t]})_j| \ge t\delta\right) \leq 2\exp\left\{-\frac{t\delta^2}{2(m_X\sigma)^2} \right\}.
\end{align*}
Then the first claim follows due to the union bound. The second claim follows immediately from the first claim and the union bound.
\end{proof}

\begin{lemma} \label{lemma: ell_2 bound of x epsilon}
Suppose that $\|\bX_1\|_{\infty}\le m_X$ and let $t \in [T]$. Fix $k \in [K]$, and assume $S_k$ is non-empty. Then 
\begin{align*}
     \Exp\left[ \|(\bX_{[t]}^{(k)})_{S_k}'\bepsilon_{[t]}/t \|_2^2 \right] \le  4(m_X\sigma)^2s_0/t.
\end{align*}
Further, there exists an absolute constant $C>0$, such that for any $\tau>0$,
\begin{align*}
\Pro\left(\|(\bX_{[t]}^{(k)})_{S_k}'\bepsilon_{[t]}/t\|_2 \ge C  m_X\sigma\sqrt{\frac{s_0\log (s_0) + s_0\tau}{t}} \right) \le 2e^{-\tau}.
\end{align*}
In particular, if $S_k$ is non-empty for all $k \in [K]$, then
$$\Pro\left(\bigcap_{k = 1}^{K}\left\{\|(\bX_{[t]}^{(k)})_{S_k}'\bepsilon_{[t]}/t\|_2 \le C  m_X\sigma\sqrt{\frac{3s_0\log (dT)}{t}} \right\} \right)\geq 1- \frac{K}{T^2}.$$
\end{lemma}

\begin{proof}
Note that
$(\bX_{[t]}^{(k)})_{S_k}'\bepsilon_{[t]} = \sum_{s=1}^{t}(\bX_s)_{S_k}\idf\{A_s=k\}\epsilon_s$.

\medskip
\noindent\textbf{Step 1. Expectation bound.} Applying the same reasoning as in the proof of Lemma \ref{lemma:ell_infty bound of x times epsilon}, we have that the sequence $\{(\bX_s)_{S_k}\idf\{A_s=k\}\epsilon_s: s\in [t]\}$ forms a vector-valued $\{\cF_s\}$-martingale difference sequence. Thus, we have
\begin{align*}
   \Exp\left[\|(\bX_{[t]}^{(k)})_{S_k}'\bepsilon_{[t]} \|_2^2\right] \leq   \Exp[\epsilon_1^2] \sum_{s\in [t]}\Exp[(\bX_s)_{S_k}'(\bX_s)_{S_k}].
\end{align*}
Since $\epsilon_1$ is $\sigma$-sub-Gaussian, we have $\Exp[\epsilon_1^2] \leq 4\sigma^2$.
The proof then is complete since $\|\bX_1\|_{\infty} \leq m_X$.

\medskip
\noindent
\textbf{Step 2. Probability bounds.} Applying the same reasoning as in the proof of Lemma \ref{lemma:ell_infty bound of x times epsilon}, we have that the sequence $\{(\bX_s)_{S_k}\idf\{A_s=k\}\epsilon_s: s\in [t]\}$ forms a vector-valued $\{\cF_s\}$-martingale difference sequence, and that
for any $\tau>0$, 
\begin{align*}
    \Pro(\|(\bX_s)_{S_k}\idf\{A_s=k\}\epsilon_s\|_2 \ge \tau \mid \cF_{s-1}) \le 2\exp\left\{-\frac{\tau^2}{2(m_X \sigma \sqrt{s_0})^2} \right\}.
\end{align*}
Then, the first probability bound follows immediately from Corollary 7 in \cite{jin2019short} with $\delta = 2e^{-\tau}$.
The second probability bound is finished by setting $\tau = 2\log (dT)$ and the union bound.
\end{proof}

\section{Proofs for the Bandit Problem - Lower Bounds}

\subsection{Proof of Theorem \ref{thm: regret lower bound}} \label{sec: proof of cumulative regret, lower bound}
In this subsection, we prove Theorem \ref{thm: regret lower bound}, establishing a lower bound on the cumulative regret over all permissible rules for the case $K = 2$. 

For technical analysis, we consider the following equivalent problem formulation. For $t \in [T]$, define
\begin{align}\label{def:eqvilent}
    Y_{t}^{(k)} = \bX_{t}'\btheta^{(k)}  + \epsilon_{t}^{(k)},
\end{align}
where $\{\epsilon_{t}^{(k)}:t \in [T], k \in [2]\}$ are i.i.d. with the normal distribution $\mathcal{N}(0,\sigma^2)$. We may view $Y_{t}^{(1)}$ and $Y_{t}^{(2)}$ as the potential rewards, and when an arm $A_t \in [K]$ is selected, then $Y_t = Y_t^{(A_t)}$ is observed.

Compared to the formulation in Section \ref{sec: problem formulation}, the model in \eqref{def:eqvilent} assigns independent noise to each arm. However, since only one arm is pulled at each time, the two formulations are equivalent: for any admissible policy $\{\pi_t : t \in [T]\}$ as defined in Section \ref{sec: problem formulation}, the joint distribution of the observations $\{(\bX_t, \bY_t, A_t) : t \in [T]\}$, as well as the regret sequence $\{r_t : t \in [T]\}$, are identical under both formulations.

Recall the two families of distributions on $\bR^{d}$: $\omega_1(\cdot)$ in Definition \ref{def: omega_1} and $\omega_2(\cdot)$ in Definition \ref{def: omega_2}, and  the definition of $\widetilde{\bTheta}_{d}(s_0,L)$ prior to Theorem \ref{thm: regret lower bound}. The proof strategy follows a similar approach to that in Section 6 of \cite{song2022truncated}, but with different prior distributions.

Recall that in Theorem \ref{thm: regret lower bound}, we assume that for some constant $\kappa > 1$, we have
\begin{equation}\label{app_eq: regret lower bound 1}
3 \leq s_0 \leq (d+2)/3, \quad  d \geq s_0^{\kappa}, \quad 
T \geq \left( 2 (C_{*} \vee \sigma^2 \vee 1) s_0^5 \log(dT)\right)^{\kappa}.
\end{equation}

\begin{proof}[Proof of Theorem \ref{thm: regret lower bound}]
In this proof, we consider the equivalent formulation given in \eqref{def:eqvilent}. Fix any permissible rule and denote its arm selections by $\{A_t \in [2] : t \in [T]\}$ and its instantaneous regrets by $\{r_t: t \in [T]\}$.

\medskip
\noindent \underline{Step 1. Lower bound the supremum risk by the average risk.} 
Let $\tilde{\omega}_1$ and $\tilde{\omega}_2$ be two distributions on $\bR^{d+d}$, both supported on $\widetilde{\bTheta}_{d}(s_0,L)$, defined as follows.

\begin{enumerate}[label=\roman*)]
    \item If a random vector 
 $(\btheta^{(1)}, \btheta^{(2)}) \in \bR^{d+d}$ has the distribution given by $\tilde{\omega}_1$, then 
\begin{align}
    &\btheta^{(1)}=\bd{0}_d, \quad \btheta^{(2)} \text{ has the distribution } \omega_1(s_0, 2/L) \text{ in Definition } \ref{def: omega_1}. \label{eq: instance 1}
\end{align}
By Definition \ref{def: omega_1}, under $\tilde{\omega}_1$, almost surely (a.s.), $\|\btheta^{(1)} - \btheta^{(2)}\|_2 \geq 1/L$ and $\|\btheta^{(2)}\|_2 \leq 2/L \leq 2L$.
Thus, $\tilde{\omega}_1$ is supported on $\widetilde{\bTheta}_{d}(s_0,L)$, which is defined prior to Theorem \ref{thm: regret lower bound}.

\item If a random vector $(\btheta^{(1)}, \btheta^{(2)}) \in \bR^{d+d}$ has the distribution given by $\tilde{\omega}_2$, then 
\begin{align}
    \btheta^{(1)}=\bd{0}_d, \quad \btheta^{(2)} \text{ has the distribution } \omega_2(s_0,1/L,\delta_{*}) \text{ in Definition } \ref{def: omega_2}, \label{eq: instance 2}
\end{align}
where we define
\begin{align*}
  \delta_{*} := \frac{\sigma^2}{16 L_3 T} \frac{s_0-1}{2} \log\left(\frac{d-s_0}{(s_0-1)/2}  \right).
\end{align*}
Due to Lemma \ref{lemma: packing set}, under $\tilde{\omega}_2$, almost surely (a.s.), $\|\btheta^{(1)}-\btheta^{(2)}\|_2 \geq 1/L$ and $\|\btheta^{(2)}\|_2 \leq \sqrt{1/L^2 + 2 \delta_*^2}\le 2L$, where the last inequality is because we have $2\delta_*^2\le 3L^2$ due to \eqref{app_eq: regret lower bound 1}. Thus,  $\tilde{\omega}_2$ is supported $\widetilde{\bTheta}_{d}(s_0,L)$.
\end{enumerate}

Now, let $\tilde{\omega}$ be the mixture of $\tilde{\omega}_1$ and $\tilde{\omega}_2$ with equal weights. This mixture is also supported on $\widetilde{\bTheta}_{d}(s_0,L)$. It is important to note that $\tilde{\omega}$ does not depend on $t$. Then,
\begin{align*}
   & 2\sup_{(\btheta^{(1)}, \btheta^{(2)}) \in \widetilde{\bTheta}_{d}(s_0,L)}  \Exp_{\btheta^{(1)}, \btheta^{(2)}}\left[\sum_{t=\tilde{\gamma}}^T r_t \right] \geq 2\Exp_{(\btheta^{(1)}, \btheta^{(2)}) \sim \tilde{\omega}}\left[\sum_{t=\tilde{\gamma}}^{T}  r_t  \right] \\
   \geq & \Exp_{(\btheta^{(1)}, \btheta^{(2)}) \sim \tilde{\omega}_1}\left[\sum_{t=\tilde{\gamma}}^{T}  r_t  \right]  + \Exp_{(\btheta^{(1)}, \btheta^{(2)}) \sim \tilde{\omega}_2}\left[\sum_{t=\tilde{\gamma}}^{T}  r_t  \right],
\end{align*}
where $\Exp_{(\btheta^{(1)}, \btheta^{(2)}) \sim \tilde{\omega}}$ means that the parameters $(\btheta^{(1)}, \btheta^{(2)})$ is treated as a random vector with the prior distribution $\tilde{\omega}$, and for simplicity, we write it as $\Exp_{\tilde{\omega}}$ below."

\medskip
\noindent \underline{Step 2. Lower bound  expected regrets by $\ell_2$-estimation error.} Let $(\btheta^{(1)}, \btheta^{(2)})$ be a random vector
with the prior distribution $\tilde{\omega}$. Note that a.s., $\btheta^{(1)} = \bd{0}_d$. Then, for $t \in [T]$, since $r_t = \max_{k \in [2]} \bX_t'\left(\btheta^{(k)} - \btheta^{(A_t)}\right)$,  a.s.,
\begin{align*}
    r_t = (\btheta^{(2)})'\bX_t\idf\{(\btheta^{(2)})'\bX_t > 0,\; A_t = 1\} - (\btheta^{(2)})'\bX_t\idf\{(\btheta^{(2)})'\bX_t < 0,\; A_t = 2\}.
\end{align*}
Recall that $\cF_t = \sigma(\{(\bX_s, A_s, Y_s), s\le t\})$ is the $\sigma$-algebra generated by the observations up to time $t$. Then, the optimal decision rule $A_t^*$, which minimizes the expected regret $\Exp_{\tilde{\omega}}[r_t]$, can be achieved by the following: $A_t^* = 1$ if and only if
\begin{align*}
    &\Exp[(\btheta^{(2)})'\bX_t\idf\{(\btheta^{(2)})'\bX_t \ge 0\}\mid \cF_{t-1}\cup\bX_t]\\
    \leq \; & -\Exp[(\btheta^{(2)})'\bX_t\idf\{(\btheta^{(2)})'\bX_t < 0\}\mid \cF_{t-1}\cup\bX_t],
\end{align*}
or equivalently, $A_t^* = 1$ if and only if $\Exp\left[\btheta^{(2)}\mid \cF_{t-1}\cup\bX_t\right]'\bX_t\le 0$. Since $\bX_{t}$ and $\btheta^{(2)}$ are independent, we have
 \begin{align*}
     A_t^* = \begin{cases}
         1, \; &\text{ if } (\hbTheta_{t-1}^{(2)})'\bX_t \leq 0 \\
         2, \; & \text{ otherwise } \\  
     \end{cases}, \quad \text{ where }\;\; \hbTheta_{t-1}^{(2)} :=  \Exp\left[\btheta^{(2)}\mid \cF_{t-1}\right].
 \end{align*}

Thus, for any fixed rule $\{A_t: t \in [T]\}$, by the law of total expectation we have for $t \in [T]$
\begin{align}\label{aux:bandit_lower_rt}
\begin{split}
\Exp_{\tilde{\omega}}[r_t] &\geq   \Exp_{\tilde{\omega}}[(\btheta^{(2)})'\bX_t\idf\{(\btheta^{(2)})'\bX_t > 0\;,\;(\hbTheta_{t-1}^{(2)})'\bX_t\le 0\}
\\&\;\;\qquad -(\btheta^{(2)})'\bX_t\idf\{(\btheta^{(2)})'\bX_t \leq 0\;,\;(\hbTheta_{t-1}^{(2)})'\bX_t > 0\}]
\\&=\Exp_{\tilde{\omega}}\left[|(\btheta^{(2)})'\bX_t|\idf\{\textup{sgn}((\btheta^{(2)})'\bX_t)\neq\textup{sgn}((\hbTheta_{t-1}^{(2)})'\bX_t)\}\right],
\end{split}
 \end{align}
where for $\tau \in \bR$, we define $\text{sgn}(\tau) := \idf\{\tau > 0\} - \idf\{\tau < 0\}$.

For $t \in [T]$, applying Lemma \ref{lemma: estErr_to_regret} to $\bX_t$ conditional on $\cF_{t-1}$, and due to Assumptions \ref{assumption: context} and  \ref{assumption: bounded cov mat}, we have that there exists a constant $C>0$ depending on $L_3$, such that, a.s.,
\begin{align*}
&\Exp_{\tilde{\omega}}\left[|(\btheta^{(2)})'\bX_t|\idf\{\textup{sgn}((\btheta^{(2)})'\bX_t)\neq\textup{sgn}((\hbTheta_{t-1}^{(2)})'\bX_t)\} \; \vert \cF_{t-1} \right] \ge \\
&C^{-1}\|\btheta^{(2)}\|_2\left\|\frac{\btheta^{(2)}}{\|\btheta^{(2)}\|_2}- \frac{\hbTheta_{t-1}^{(2)}}{\|\hbTheta_{t-1}^{(2)}\|_2}\right\|_2^2 \ge \frac{1}{C L}\left\|\frac{\btheta^{(2)}}{\|\btheta^{(2)}\|_2}- \frac{\hbTheta_{t-1}^{(2)}}{\|\hbTheta_{t-1}^{(2)}\|_2}\right\|_2^2,
\end{align*}
where we use  the fact that $\|\btheta^{(2)}\|_2 \geq 1/L$ a.s. under $\tilde{\omega}$ and that $\hbTheta_{t-1}^{(2)}/\|\hbTheta_{t-1}^{(2)}\|_2$ is interpreted as zero if the denominator is zero .

\medskip
\noindent \underline{Step 3. Lower bound the $\ell_2$ estimation errors.} 
For $t \in [T]$, we denote by 
$$
\mathcal{H}_t := \sigma(\{(\bX_s, Y_s^{(1)}, Y_s^{(2)}: s \in [t]\})
$$ 
all potential observations up to time $t$. Note that $\cF_{t-1} \subseteq {\mathcal{H}}_{t-1}$, and thus
 $\hbTheta_{t-1}^{(2)}/\|\hbTheta_{t-1}^{(2)}\|_2 \in \cF_{t-1} \subset {\mathcal{H}}_{t-1}$, which, in view of \eqref{aux:bandit_lower_rt}, implies that 
$$
\Exp_{\tilde{\omega}}\left[r_t \right] \ge \frac{1}{CL} \inf_{\hat{\psi}_{t-1} \in  {\mathcal{H}}_{t-1}} \Exp_{\tilde{\omega}}\left(\left\|\hat{\psi}_{t-1}-\frac{\btheta^{(2)}}{\|\btheta^{(2)}\|_2} \right\|_2^2 \right),
$$
where the infimum is taken over all $\mathbb{R}^d$-valued, $\mathcal{H}_{t-1}$-measurable random vectors.
Further, for $t \in [T]$,   since $Y_s^{(1)} = \epsilon_s^{(1)}$ for $s \in [t]$, we have $\{Y_s^{(1)}:  s \in [t]\}$ are independent from $\tilde{\mathcal{H}}_t := \sigma(\{(\bX_s,   Y_s^{(2)}:  s \in [t]\})$ and $\btheta^{(2)}$. Thus, the above infimum can be taken 
over all $\mathbb{R}^d$-valued, \emph{$\tilde{\mathcal{H}}_{t-1}$}-measurable random vectors.

Finally, by Lemma \ref{lemma: K known} and Lemma \ref{lemma: k unknown} ,  for some constant $C > 0$, depending only on $L$ and $L_3$, we have
\begin{align*}
    \inf_{\hat{\psi}_{t-1} \in  {\tilde{\mathcal{H}}}_{t-1}} \Exp_{\tilde{\omega}}\left(\left\|\hat{\psi}_{t-1}-\frac{\btheta^{(2)}}{\|\btheta^{(2)}\|_2} \right\|_2^2 \right) \geq C^{-1}  \left(\frac{\sigma^2 s_0}{t+\sigma^2s_0} + \frac{\sigma^2}{T}s_0\log\frac{d}{s_0} \right),
\end{align*}
which implies that
\begin{align*}
\Exp_{\tilde{\omega}}\left[\sum_{t=\tilde{\gamma}}^{T} r_t \right] \geq C^{-1} \sigma^2 s_0 \left(\log\left( \frac{T + \sigma^2 s_0}{ \tilde{\gamma} +1+ \sigma^2 s_0} \right) + \frac{T - \tilde{\gamma}}{T} \log\left( \frac{d}{s_0}\right) \right).
\end{align*}
The proof is then complete due to \eqref{app_eq: regret lower bound 1}.
\end{proof}

\subsection{Proof of Theorem \ref{thm: bandit lower bound of lasso}}\label{sec: proof of thm: bandit lower bound of lasso}
In Theorem \ref{thm: bandit lower bound of lasso}, $K = 2$, and   for $(\btheta^{(1)}, \btheta^{(2)}) \in \widetilde{\bTheta}_{d}(s_0,L)$, Assumption \ref{assumption:arm_parameters} holds with $m_{\theta} = 2L$ and $L_4 = L$. Further, we consider the specific case
$$\bSigma^{(1)} = \bSigma^{(2)} = 2^{-1} \mathbb{I}_{d \times d},$$
where $\mathbb{I}_{d \times d}$ is the $d$-by-$d$ identity matrix. Then,
 Assumptions \ref{assumption: bounded cov mat} and \ref{assumption: constraint on the covariance matrix 1} hold with $L_3=2$.

Recall the fixed constants  $C^* \geq 1$ and $a^*, h^*, \ell_0^*, \ell_1^* > 0$  in \eqref{def:star_constants}, which only depend only on $K=2,  m_X, \alpha_X,  L_3 = 2, m_{\theta} = 2L, L_4 = L$, such that \eqref{stage 2:aux_eq1}-\eqref{eq: margin} hold. 

\begin{lemma}  \label{lemma:LassoBandit}
Consider the case $K =2$ and set $\gamma_2 = T$ in Definition \ref{def:three_stages}, that is, without Stage 3. 
Suppose that Assumption  \ref{assumption: context}  holds, that $\bSigma^{(1)} = \bSigma^{(2)} = 2^{-1} \mathbb{I}_{d \times d}$, and that  $(\btheta^{(1)}, \btheta^{(2)}) \in \widetilde{\bTheta}_{d}(s_0,L)$ for some $L > 1$ and the supports $S_1$ and $S_2$ are both non-empty. Set the regularization parameter $\lambda_t$   and the end time $\gamma_1$ of Stage 1 
as follows:
\begin{align}
    \label{modified:choice}
\gamma_1 = C_{\gamma_1} \lceil(\sigma^2\vee 1)\rceil s_0 \lceil \log (dT) \rceil, \quad
\lambda_t = 48 m_X\sigma\sqrt{\frac{\log (dT)}{t}}.
\end{align}
Suppose that the integer   $C_{\gamma_1}$ in \eqref{modified:choice} satisfies the following:
\begin{align}
    \label{eq: lassobandit_lambda_gamma_12}
     C_{\gamma_1} \ge \max\left\{C^*,\; \left(\frac{18m_X}{\ell_1^* a^*}\right)^2,2\right\}.
\end{align}
Let $M_*$ be the smallest integer satisfying:  
\begin{align}\label{def: M_star}
     M_* \geq  \max\left\{4, \frac{32 m_{X}^4}{(a^*)^2}, \left(\frac{3072 L  C^* m_X^4}{(a^*)^2} \right)^4 \right\}.
\end{align}
Then,  for $m = M_* s_0^4, \ldots, \lfloor T/\gamma_1\rfloor$, with probability at least $1-38/T^2$,  the event $\tL_m^{\textup{RE}} \cap \tL_m^{\bSigma} \cap  \tL_m^{\textup{Lasso}} \cap \tL_m^{\textup{Inc}}$ occurs, where, denoting $t = m\gamma_1$,
\begin{align} \label{LassoBandit:events}
\begin{split}
&\tL_m^{\textup{RE}} := \bigcap_{k=1}^{2}\left\{\bX_{[t]}^{(k)} \textup{ satisfies } \textup{RE}(s_0, 3, a^*)\right\}, \\
& \tL_m^{\bSigma} := \bigcap_{k=1}^{2}\left\{\mn\hbSigma_{t}^{(k)}-\bSigma^{(k)}\mn_{\textup{max}} \le  \frac{3 a^*}{s_0} \right\},\\ 
&\tL_m^{\textup{Lasso}} = \bigcap_{k = 1}^{2}\left\{\|\hbtheta_{t}^{(k)} - \btheta^{(k)}\|_{\infty} \leq  56\lambda_{t}\right\},\\
&\tL_m^{\textup{Inc}} =  \bigcap_{k=1}^{2}\left\{ \hat{S}_t^{(k)} \; \subseteq \; S_k 
\right\},
\end{split}
\end{align}
and $\hat{S}_t^{(k)} := \{j \in [d]: (\hbtheta_{t}^{(k)})_j \neq 0\}$ is the support of the Lasso estimator $\hbtheta_{t}^{(k)}$
\end{lemma}
\begin{proof}
By the same argument as in Lemma \ref{stage3: base case}, but with a slight different choice of $\lambda_t$, we have that for each $m = M_* s_0^4,\ldots, \lfloor T/\gamma_1 \rfloor$, with probability at least $1-34/T^2$, we have 
$\tL_m^{\textup{RE}} \cap \tL_m^{\bSigma} \cap  \tL_m^{\textup{Lasso}}$ holds. As a result, due to the union bound, it suffices to show that 
\begin{align*}
\Pro\left((\tL_m^{\textup{Inc}} )^c\; \cap \; \tL_m^{\bSigma}\right) \leq 4/T^2.
\end{align*}

Fix $m \in \{M_*s_0^4,\ldots,\lfloor T/\gamma_1 \rfloor\}$, and denote by $t = m \gamma_1$. Since $\bSigma^{(1)} = \bSigma^{(2)} = 2^{-1}\mathbb{I}_{d\times d} = 2^{-1} \bSigma$,  Assumptions  \ref{assumption: bounded cov mat} and \ref{assumption: constraint on the covariance matrix 1} hold with $L_3 = 2$. Thus, on the event $\tL_m^{\bSigma}$, due to Lemma \ref{lemma: subgaussian lower eigenvalue whp} and the definition of $a^*$ in \eqref{eq: a},  for $k \in [2]$,
\begin{align}\label{aux:eign}
   \frac{1}{4} \le \lmin{(\hbSigma_{t}^{(k)})_{S_k, S_k}} \le \lmax{(\hbSigma_{t}^{(k)})_{S_k, S_k}} \le 3,
\end{align}
which verifies condition (7.43a) in \cite{wainwright2019high}.

Thus, for $k \in [2]$ and  $j \not\in S_k$, on the event $\tL_m^{\bSigma}$,
\begin{align*}
&\left\|\left( (\bX_{[t], S_k}^{(k)})' \bX_{[t], S_k}^{(k)} \right)^{-1} (\bX_{[t], S_k}^{(k)})' \bX_{[t], j}^{(k)} \right\|_2
\leq 4 \left\|t^{-1}(\bX_{[t], S_k}^{(k)})' \bX_{[t], j}^{(k)} \right\|_2 \\
\leq &4 \sqrt{s_0 \times ((3a^*)/s_0)^2} = 12a^*/\sqrt{s_0}.
\end{align*}
Due to \eqref{eq: a}, $a^* \leq 1/16$. By Cauchy–Schwarz inequality, for $k \in [2]$, on the event $\tL_m^{\bSigma}$, 
\begin{align}\label{aux:mutual_inco}
  \max_{j \not\in S_k}  \left\|\left( (\bX_{[t], S_k}^{(k)})' \bX_{[t], S_k}^{(k)} \right)^{-1} (\bX_{[t], S_k}^{(k)})' \bX_{[t], j}^{(k)} \right\|_1 \leq \sqrt{s_0} \times 12a^*/\sqrt{s_0} \leq 3/4,
\end{align}
which verifies condition (7.43b) in \cite{wainwright2019high}, and also implies that 
\begin{align}\label{aux:wain}
\begin{split}
 &\max_{j \not\in S_k}  \left| ( \bX_{[t], j}^{(k)})'\bX_{[t], S_k}^{(k)} \left( (\bX_{[t], S_k}^{(k)})' \bX_{[t], S_k}^{(k)} \right)^{-1} (\bX_{[t], S_k}^{(k)})' \bepsilon_{[t]}/t  
\right| \\
\leq & 3/4 \times \|(\bX_{[t], S_k}^{(k)})' \bepsilon_{[t]}/t  \|_{\infty}.
\end{split}
\end{align}
Define $\mathcal{A}$ to be the following event:
\begin{align*}
\bigcap_{k \in [2]} & \left\{ \max_{j \not\in S_k}  \left| ( \bX_{[t], j}^{(k)})' \left(\mathbb{I}_{|S_k| \times |S_k|} - \bX_{[t], S_k}^{(k)} \left( (\bX_{[t], S_k}^{(k)})' \bX_{[t], S_k}^{(k)} \right)^{-1} (\bX_{[t], S_k}^{(k)})'\right) \bepsilon_{[t]}/t  
\right|  
    \right. \\
    &\left. \quad \leq 6m_X\sigma \sqrt{\frac{\log(dT)}{t}} = \lambda_t/8 \right\},
\end{align*}
where we recall the choice of $\lambda_t$ in \eqref{modified:choice}. Note that on the event $\cA$, condition (7.44) in \cite{wainwright2019high} holds. Further, by Lemma \ref{lemma:ell_infty bound of x times epsilon}, in view of \eqref{aux:wain}, we have
\begin{equation*}
    \Pro\left( \mathcal{A}^c \cap \tL_m^{\bSigma}\right) \leq  {4}/{T^4}.
\end{equation*}
 
By Theorem 7.21 in \cite{wainwright2019high}, when events \eqref{aux:eign}, \eqref{aux:mutual_inco}, and $\cA$ all hold, we have $\tL_m^{\textup{Inc}}$ holds. Thus,
$\Pro\left((\tL_m^{\textup{Inc}})^c\; \cap\; \tL_m^{\bSigma}\right)\le \Pro\left( \mathcal{A}^c \cap \tL_m^{\bSigma}\right) \leq 4/T^2$. The proof is then complete.
\end{proof}

Define the following function: for $z_1,z_2 \in (-1/12,1/12)$,
\begin{align}
    \label{def:elementary_func}
    f(z_1, z_2) = \left\|\frac{1}{\sqrt{1 + (-1/12)^2}}\begin{bmatrix}
        1  \\
        -\frac{1}{12} 
    \end{bmatrix} - \frac{1}{\sqrt{(1-z_1)^2 + (-1/12+z_2)^2}}\begin{bmatrix}
        1 - z_1 \\
        -\frac{1}{12} + z_2
    \end{bmatrix}  \right\|_2^2.
\end{align}

\begin{lemma}
\label{lemma:elementary}
There exists some absolute constant $c_* > 0$ such that for any $0 < \delta < c_*$, 
\begin{align*}
   \inf_{(z_1, z_2) \in [\delta/2, 5\delta]^2}\; f(z_1, z_2) \geq c_* \delta^2.
\end{align*}
\end{lemma}
\begin{proof}
By elementary calculation, the gradient vector of $f$ at $(0,0)$ is the zero vector. 

Let $\bu = (1,1/12)'$ and $\bv = (-1/12,1)'$. Again by elementary calculation,  $\bu$ and $\bv$ are the eigen-vectors of the Hessian matrix $ \nabla^2 f(0,0)$ corresponding to eigen-values $0$ and $2/(1+(1/12)^2)$, respectively.  
For $(z_1, z_2) \in [\delta/2, 5\delta]^2$, we have
$$
(z_1,z_2) \bv  \geq - \frac{5\delta}{12} + \frac{\delta}{2}  =\frac{\delta}{12} > 0.
$$
Then the proof is complete by Taylor's theorem.
 \end{proof}

In the following proof, denote by $c_*$ a constant such that Lemma \ref{lemma:elementary} holds.

\begin{proof}[Proof of Theorem \ref{thm: bandit lower bound of lasso}]
In this proof, we fix the following pair of arm parameters:
\begin{align*}
    \btheta^{(1)} := (1, \bd{0}_{d-1}')' \text{ with } S_1 = \{1\}, \quad \text{ and } \quad
    \btheta^{(2)} := (0, 1/12, \bd{0}_{d-2}')', \text{ with } S_2 = \{2\},
\end{align*}
where $\bd{0}_{\ell}$ is the all-zero vector with length $\ell$. Clearly, $(\btheta^{(1)},\btheta^{(2)}) \in  \widetilde{\bTheta}_{d}(s_0,L)$ for any $s_0 \geq 1$ and $L \geq 1$.

Recall that $M_*$ is defined in \eqref{def: M_star} and that $\gamma_1 = C_{\gamma_1}  \lceil(\sigma^2\vee 1)\rceil s_0 \lceil \log (dT) \rceil$ in \eqref{modified:choice}. We assume that $T \geq M_* s_0^4\gamma_1$ and that the integer $C_{\gamma_1}$ satisfies \eqref{eq: lassobandit_lambda_gamma_12}.

\medskip
\noindent \underline{\textbf{Step 1.} Control the estimation error.} In this step,
let $m \in \{M_* s_0^4,\ldots, \lfloor T/\gamma_1\rfloor\}$. 
By lemma \ref{lemma:LassoBandit},   since $T \geq 64$,   with probability at least $0.99$, $\tL_m^{\textup{RE}} \cap \tL_m^{\bSigma} \cap  \tL_m^{\textup{Lasso}} \cap \tL_m^{\textup{Inc}}$ holds. 
Further, applying Lemma \ref{lemma:ell_infty bound of x times epsilon} to the first two covariates, we have $\Pro\left( \tL_m^{\bX\bepsilon}
 \right) \geq 0.41$, where 
\begin{align*}
\tL_m^{\bX\bepsilon} := \bigcap_{j \in [2]}
 \left\{ 
 \left|(\bX_{[t],j})' \bepsilon_{[t]}/t \right|
 \leq \frac{2m_X\sigma}{\sqrt{t}}  \right\}, \quad \text{ and } t = m\gamma_1.
\end{align*}
Thus, by the union bound, we have
\begin{align}\label{aux:lasso_bandit_overall}
    \Pro\left(\tL_m^{\textup{all}} \right) \geq 0.4, \quad \text{ where }\;\;
    \tL_m^{\textup{all}} := \tL_m^{\textup{RE}} \cap \tL_m^{\bSigma} \cap  \tL_m^{\textup{Lasso}} \cap \tL_m^{\textup{Inc}} \cap \tL_m^{\bX\bepsilon}.
\end{align}
Due to the definition of $M_*$ in \eqref{def: M_star} and since $m \geq M_* s_0^4$, we have $56\lambda_{m\gamma_1} \leq 1/24$.
By a similar argument as in the proof of Theorem \ref{thm: lower bound of lasso}, on the event $\tL_m^{\textup{Lasso}} \cap \tL_m^{\textup{Inc}}$, we must have
\begin{align}\label{aux:support_recov_lassobandit}
    \hat{S}_{m\gamma_1}^{(1)} = S_1 = \{1\}, \quad    \text{ and }  \quad \hat{S}_{m\gamma_1}^{(2)} = S_2 = \{2\},
\end{align}
and further for $k \in [2]$,
\begin{align*}
\frac{(\bX_{[t],k}^{(k)})'\bX_{[t],k}^{(k)}}{t}[(\hbtheta_t^{(k)}-\btheta^{(k)})_k]=\frac{(\bX_{[t],k}^{(k)})'\bepsilon_{[t]}}{t} - \lambda_t, \quad \text{ where } t = m\gamma_1.
\end{align*}
For the left hand side, on the event $\tL_m^{\bSigma}$, since $a^* \leq 1/16$ due to \eqref{eq: a} and $\bSigma^{(1)} = \bSigma^{(2)} = 2^{-1} \mathbb{I}_{d \times d}$, we have
\begin{align*}
    5/16 \leq (\bX_{[t],k}^{(k)})'\bX_{[t],k}^{(k)}/t \leq  11/16,\quad \text{ for } k \in [2].
\end{align*}
For the right-hand side, on the event $\tL_m^{\bX\bepsilon}$, due to the definition of $\lambda_{t}$ in \eqref{modified:choice}, we have that  
\begin{align*}
    -\frac{3\lambda_t}{2}  \leq {(\bX_{[k],k}^{(k)})'\bepsilon_{[t]}}/{t} - \lambda_t  
\leq  -\frac{\lambda_t}{2},\quad \text{ for } k \in [2] \text{ and } t = m\gamma_1.
\end{align*}
Thus, on the event $\tL_m^{\textup{all}}$, we have that for $t = m\gamma_1$, 
\begin{align}\label{aux:lassobandit_est_error}
\begin{split}
    -\frac{24}{5} \lambda_t \leq \hbtheta_{t,1}^{(1)} - 1 \leq -\frac{8}{11} \lambda_t, \quad \text{ and } \quad       -\frac{24}{5} \lambda_t \leq \hbtheta_{t,2}^{(2)} - \frac{1}{12} \leq -\frac{8}{11} \lambda_t.
    \end{split}
\end{align}

\medskip
\noindent \underline{\textbf{Step 2.} Lower bound the instantaneous regret.} In this step, let $m \in \{1,\ldots, \lfloor T/\gamma_1\rfloor\}$. By arguments similar to those in the proof of Theorem \ref{thm: regret over horizon} (see Section \ref{sec: proof of thm: regret over horizon}), for  each $t \in (m\gamma_1, (m+1)\gamma_1]$, we have
\begin{align*}
    r_t = \left|\bu'\bX_t\right| \idf\{
    \textup{sgn}(\bu'\bX_t) \neq   \textup{sgn}(\hat{\bu}_m'\bX_t) 
    \}
\end{align*}
where  $\textup{sgn}(\tau) := \idf\{\tau > 0\} - \idf\{\tau < 0\}$ for $\tau \in \bR$, and
\begin{align*}
    \bu := \btheta^{(1)} - \btheta^{(2)}, \quad 
    \hat{\bu}_m := \hbtheta_{m\gamma_1}^{(1)} - \hbtheta_{m\gamma_1}^{(2)}.
\end{align*}
Due to \eqref{aux:support_recov_lassobandit}, we have that $\bu_j =  \hat{\bu}_{m,j} = 0$ for $j > 2$, and
\begin{align*}
\bu_{[2]} = (1,-1/12)', \quad
\hat{\bu}_{m,[2]} =  (\hbtheta_{m\gamma_1,1}^{(1)}, -\hbtheta_{m\gamma_1,2}^{(2)})'.
\end{align*}

Due to Assumption \ref{assumption: context} and since $\bSigma = \mathbb{I}_{d \times d}$, applying Lemma \ref{lemma: estErr_to_regret} to $\bX_{t}$ conditional on $\cF_{m\gamma_1}$, there exists an absolute constant $C >0$ such that that for each $t \in (m\gamma_1, (m+1)\gamma_1]$, almost surely (a.s.),
\begin{align*}
   \Exp\left[r_t \vert \cF_{m\gamma_1}\right]  
   \geq C^{-1} \|\bu_{[2]}\|_2 \left\|\frac{\bu_{[2]}}{\|\bu_{[2]}\|_2} - \frac{\hat{\bu}_{m,[2]}}{\|\hat{\bu}_{m,[2]}\|_2} \right\|_2^2,
\end{align*}
where $\hat{\bu}_m/{\|\hat{\bu}_m\|_2}$ is interpreted as zero if the denominator is zero.

Now, if
\begin{align}
    \label{def:m another requirement}
    m \geq \tilde{\gamma} := \max\left\{M_* s_0^4,  \left(\frac{48 m_X}{c_*}\right)^2\right\},
\end{align}
then $\lambda_{m\gamma_1} \leq c^*$. By the definition of $\btheta^{(1)},\btheta^{(2)}$, and due to \eqref{aux:support_recov_lassobandit}, for each $t \in (m\gamma_1, (m+1)\gamma_1]$, a.s.,
\begin{align*}
     \Exp\left[r_t \vert \cF_{m\gamma_1}\right]   \idf\{\tL_m^{\textup{all}}\}  
   \geq C^{-1} \frac{\sqrt{145}}{12}   f\left(1-\hbtheta_{m\gamma_1,1}^{(1)},\; 1/12-\hbtheta_{m\gamma_1,2}^{(2)}\right)  \idf\{\tL_m^{\textup{all}}\},
\end{align*}
where $f$ is defined in \eqref{def:elementary_func} and $C$ is an absolute constant. Since $\lambda_{m\gamma_1} \leq c^*$ and  \eqref{aux:lassobandit_est_error} holds, by Lemma \ref{lemma:elementary}, we have that for each $t \in (m\gamma_1, (m+1)\gamma_1]$, a.s.,
\begin{align}\label{aux:lassobandit_instant_regret}
     \Exp\left[r_t \vert \cF_{m\gamma_1}\right]   \idf\{\tL_m^{\textup{all}}\}  
   \geq C^{-1} \frac{\sqrt{145}}{12} c_*\lambda_{m\gamma_1}^2 \idf\{\tL_m^{\textup{all}}\},
\end{align}
which, due to \eqref{aux:lasso_bandit_overall}, implies that
\begin{align*}
  \Exp\left[r_t \right]  \geq 0.4 \frac{\sqrt{145}}{12} C^{-1}  c_* \lambda_{m\gamma_1}^2.
\end{align*}

\medskip
\noindent \underline{\textbf{Step 3.} Lower bound the cumulative regret.} Due to the above inequality, if $T \geq 2\tilde{\gamma} \gamma_1$ where $\tilde{\gamma}$ is defined in \eqref{def:m another requirement}, then
\begin{align*}
   \sum_{t = \tilde{\gamma} \gamma_1}^{T} \Exp\left[r_t \right]   = \sum_{m = \tilde{\gamma}}^{\lfloor T/\gamma_1 \rfloor} \sum_{t = m\gamma_1+1}^{(m+1) \gamma_1 \wedge T} \Exp\left[r_t \right]  \geq   C^{-1}   \sum_{m = \tilde{\gamma}}^{\lfloor T/\gamma_1 \rfloor - 1} 
   \gamma_1 \lambda_{m\gamma_1}^2,
\end{align*}
where $C > 0$ is an absolute constant. Then by the definition of $\lambda_{t}$ in \eqref{modified:choice}, we have
\begin{align*}
   \sum_{t = \tilde{\gamma} \gamma_1}^{T} \Exp\left[r_t \right]  \geq   C^{-1}  \sigma^2 \sum_{m = \tilde{\gamma}}^{\lfloor T/\gamma_1 \rfloor - 1} 
    \frac{\log(dT)}{m} \geq C^{-1} \sigma^2  \log\left(\frac{\lfloor T/\gamma_1 \rfloor}{\tilde{\gamma} + 1}\right)(\log(d) + \log(T)),
\end{align*}
where $C > 0$ is a constant that only depends on $m_X$. To summarize, there exists some constant $C^* > 0$ that only depends on $m_X$, such that if 
$$
C_{\gamma_1} \geq C^*, \quad \text{ and } \quad
T  \geq (2C_{\gamma_1} C^*  s_0^5\log(dT) )^\kappa,
$$
then 
\begin{align*}
   \sum_{t = C^* s_0^5 \log(dT)}^{T} \Exp\left[r_t \right]  \geq    C^{-1}  \sigma^2 (1-1/\kappa) \log\left(T\right)(\log(d) + \log(T)).
\end{align*}
\end{proof}

\section{Auxiliary Results on Log-concave Densities} \label{sec: log-concave}
This appendix collects several results concerning log-concave densities. Let $\bZ \in \bR^{d}$ be a random vector with a log-concave density $f_{\bZ}$. Assume that $\Exp[\bZ] = 0$, and let $\bSigma_{\bZ} := \textup{Cov}(\bZ) = \Exp[\bZ\bZ']$. 
Further, assume that for some constant $L > 1$,
\begin{align}\label{app:log_assumption}
    L^{-1}\le \lmin{\bSigma_{\bZ} }\le \lmax{\bSigma_{\bZ} }\le L.
\end{align}
For $\tau \in \bR$, define $\text{sgn}(\tau) := \idf\{\tau > 0\} - \idf\{\tau < 0\}$.

\begin{lemma}\label{lemma: estErr_to_regret}
Assume \eqref{app:log_assumption} holds. Let $\bu,\bv \in \bR^{d}$ be two vectors. Assume $\bu \neq \bd{0}_d$. Then there exists a constant $C>0$ depending only on $L$, such that
$$ \Exp\left[{|\bu'\bZ|} \idf\{\textup{sgn}(\bu'\bZ)\neq \textup{sgn}(\bv'\bZ) \} \right] \le \frac{C\|\bu-\bv\|_2^2}{\|\bu\|_2}.$$
If, in addition, $\bv \neq \bd{0}_d$, then
$$C^{-1}{\|\bu\|_2} \left\|\frac{\bu}{\|\bu\|_2}-\frac{\bv}{\|\bv\|_2} \right\|_2^2\le \Exp\left[{|\bu'\bZ|} \idf\{\textup{sgn}(\bu'\bZ)\neq \textup{sgn}(\bv'\bZ) \} \right].  $$
\end{lemma}
\begin{proof}
Note that $\Exp\left[{|\bu'\bZ|} \right]  \leq \sqrt{\bu'\bSigma_{\bZ} \bu}\leq \sqrt{L} \|\bu\|$. Thus, if $\bv = \bd{0}_d$,  the upper bounds holds with $C = \sqrt{L}$. 
Now, we focus on the case that $\bv \neq \bd{0}_d$. By applying Lemma 3 in \cite{song2022truncated} with  $\bu/\|\bu\|_2$ and $\bv/\|\bv\|_2$, we have
$$
C^{-1}\left\|\frac{\bu}{\|\bu\|_2}-\frac{\bv}{\|\bv\|_2} \right\|_2^2\le \Exp\left[\frac{|\bu'\bZ|}{\|\bu\|_2} \idf\{\textup{sgn}(\bu'\bZ)\neq \textup{sgn}(\bv'\bZ) \} \right] \le C\left\|\frac{\bu}{\|\bu\|_2}-\frac{\bv}{\|\bv\|_2} \right\|_2^2.
$$
The lower bound follows immediately. Moreover, together with Lemma 36 in \cite{song2022truncated}, the upper bound also holds.
\end{proof}

\begin{lemma}\label{lemma:density_props}
Assume \eqref{app:log_assumption} holds. Let $m \geq 1$ be an integer, and $\{\bu_1, \cdots \bu_m\} \in \mR^d$ be an arbitrary set of orthonormal vectors (i.e. 
$\|\bu_i\|_2=1$ and $\bu_i'\bu_j = 0$ for all $i\neq j\in[m]$), and denote by $f_{\bu_1,\cdots,\bu_m}$ the joint density of $(\bu_1'\bZ,\cdots,\bu_m'\bZ)$. Then there exists a constant $C > 0$ depending only on $m,L$, such that 
for all $\max_{i\in[m]}|\tau_i|\le C^{-1}$, we have
$f_{\bu_1,\cdots,\bu_m}(\tau_1,\cdots,\tau_m) \ge C^{-1}$.
\end{lemma}
\begin{proof}
    The proof is essentially the same as that for  Lemma 34(iii) in \cite{song2022truncated}, and thus omitted.
\end{proof}
\begin{remark}
We emphasize that the constant $C$ in Lemma \ref{lemma:density_props} does not depend on $d$, but only on $m$ and $L$.
\end{remark}

\begin{lemma}[Margin Condition] \label{lemma: margin condition} 
Assume \eqref{app:log_assumption} holds. Then there exists a constant $C>0$ depending only on $L$, such that for any non-zero $\bu \in \bR^{d}\setminus \{\bd{0}_d\}$, we have 
$$
\Pro(|\bu'\bZ|\le \tau)\le C \|\bu\|_2^{-1} \tau, \text{ for all } \tau>0.
$$
\end{lemma}
\begin{proof} 
It is due to Lemma 34(i) in \cite{song2022truncated}.
\end{proof}

\begin{lemma}\label{lemma: log_concave_eig_max}
Assume \eqref{app:log_assumption} holds.  Then for any constants $h, \ell_2>0$, there exists a constant $\ell_1>0$ depending only on $\ell_2,h,L$, such that  
$$
\lmax{\Exp[\bZ\bZ'\idf\left\{|\bu'\bZ|\ge h \right\}]} \le \ell_2,\; \text{ for all } \bu \in \mR^d \text{ such that } \|u\|_2\le \ell_1.
$$
\end{lemma}
\begin{proof} 
The proof is essentially the same as the arguments in Appendix D.1 of \cite{song2022truncated}, combined with Lemma 34(ii) therein. Thus, it is omitted.
\end{proof}

For each arm $k\in [K]$ and $h>0$, recall in \eqref{eq: U_h_k} that $\mathcal{U}_h^{(k)} = \{\bd{x}\in \mR^d: (\btheta^{(k)})'\bd{x}> \max_{j\ne k}(\btheta^{(j)})'\bd{x}+h\}$, whose definition only involves the arm parameters.

\begin{lemma}\label{lemma: arm optimality}
Suppose that \eqref{app:log_assumption} holds, and that Assumption \ref{assumption:arm_parameters} holds. Then there exist constants $h>0$ and $\ell_0 > 0$  depending only on $L,K,m_{\theta}, L_4$, such that  for each arm $k \in [K]$,
$$
\lmin{\Exp[\bZ\bZ'\idf\{\bZ\in \mathcal{U}_h^{(k)}\}]}\ge \ell_0.
$$
\end{lemma}
 \begin{proof} 
 Fix an arbitrary vector $\bv \in \bR^{d}$ such that $\|\bv\|_2 = 1$.  Without loss of generality, we consider the case $k = K$, and   show that for some $h,\ell_0 >0$, we have $\Exp[(\bv'\bZ)^2\idf\{\bZ\in \mathcal{U}_h^{(K)}\}] \geq \ell_0$. 
 Let $\bu_{i,j} = \btheta^{(i)}-\btheta^{(j)}$ for $i,j\in[K]$. Then by definition,
 \begin{align*}
    \mathcal{U}_h^{(K)} = \{\bd{x}\in \mR^d: (\bu_{K,j})'\bd{x}>  h, \text{ for } j =1,\ldots,K-1\}. 
 \end{align*}

We apply the  Gram–Schmidt process to transform $\{\bu_{K,1},\ldots,\bu_{K,K-1},\bv\}$ into a set of \textit{orthonormal} vectors $\{\tbu_1 ,\cdots,\tbu_{K-1} , \omega\}$. Specifically, 
we set $\gamma_1 := \|\bu_{K,1}\|_2$ and $\tbu_1  = \bu_{K,1}/\gamma_1$. Further,  for $j \in [2,K-1]$, define
\begin{align*}
   \tbu_j :=   \frac{\bu_{K,j} - \sum_{i=1}^{j-1} (\bu_{K,j}'\tbu_i) \tbu_i}{\gamma_{j,j}}, \;\;\text{ where } \gamma_{j,j}:= \|\bu_{K,j} - \sum_{i=1}^{j-1}(\bu_{K,j}'\tbu_i) \tbu_i\|_2
\end{align*}
Due to Assumption \ref{assumption: linear_independent}, this is well defined as $\gamma_{j,j} \geq L_4^{-1}$ for $1 \leq j \leq K-1$. Finally, if $\bv$ is in the linear span of $\{\bu_{K,1},\ldots,\bu_{K,K-1}\}$, let $\gamma_{K,K} = 0$ and $\bd{\omega}$ be any unit vector such that $\bd{\omega}'\tbu_j^{(k)}=0$ for all $j=1,\cdots,K-1$; otherwise, define
\begin{align*}
    \bd{\omega} := \frac{\bv  - \sum_{i=1}^{K-1} (\bv '\tbu_i) \tbu_i}{\gamma_{K,K}}, \;\;\text{ where } \gamma_{K,K} 
     := \|\bv  - \sum_{i=1}^{K-1}(\bv '\tbu_i) \tbu_i\|_2
\end{align*}
Then by definition, we have that for $1 \leq j \leq K-1$,
$$
\bu_{K,j} = \gamma_{j,j} \tbu_j + \sum_{i=1}^{j-1}  \gamma_{j,i} \tbu_i, \quad 
\bv = \gamma_{K,K} \bd{\omega} + \sum_{i=1}^{K-1}\gamma_{K,i} \tbu_i
$$
where $\gamma_{j,i} := \bu_{K,j}'\tbu_i$ and $\gamma_{K,i} = \bv '\tbu_i$. Due to Assumption \ref{assumption:arm_parameters} and the subsequent remark,  
\begin{align}\label{aux:coeffs}
    \gamma_{j,j} \geq L_4^{-1}, \quad |\gamma_{j,i}| \leq m_{\theta}, \; \text{ for } 1 \leq i < j \leq K-1, \quad \sum_{i=1}^{K} \gamma_{K,i}^2 = 1.
\end{align}

Now, define $(\bW_1,\ldots, \bW_K) := (\tbu_1'\bZ,\ldots,\tbu_{K-1}'\bZ, \bd{\omega}'\bZ)$,  and then by definition
\begin{align*}
 \{\bZ \in \mathcal{U}_h^{(K)}\}& = \bigcap_{j=1}^{K-1} \left\{ \bW_{j} > \gamma_{j,j}^{-1} (h - \sum_{i=1}^{j-1} \gamma_{j,i} \bW_i) \right\}  \\
 &\supseteq \bigcap_{j=1}^{K-1} \left\{ \bW_{j} > L_4 (h + \sum_{i=1}^{j-1} m_{\theta} |\bW_i|)\right\}.
\end{align*}
By elementary induction, we then have $\{(\bW_1,\ldots,\bW_{K-1}) \in \mathcal{A}\} \subseteq \{\bZ \in \mathcal{U}_h^{(K)}\} $, where $\mathcal{A}$ is defined as follows:
\begin{align*}
    \left\{(\tau_1,\ldots,\tau_{K-1})\in \bR^{K-1}:  (4m_{\theta}L_4)^{i-1}L_4 h \le \tau_i \le 2(4m_{\theta}L_4)^{i-1}L_4 h, \;\;\text{ for }  i \in [K-1]\right\}.
\end{align*}
As a result, due to the expression of $\bv$, we have 
\begin{align*}
    \Exp\left[(\bv'\bZ)^2\idf\{\bZ\in \mathcal{U}_h^{(K)}\}\right]  
    \geq \Exp\left[\left( \sum_{i=1}^{K} \gamma_{K,i} \bW_i\right)^2\idf\{(\bW_1,\ldots,\bW_{K-1}) \in \mathcal{A}\}\right].
\end{align*}

Finally, denote by $f$ the joint density of $(\bW_1,\ldots,\bW_K)'$. By Lemma \ref{lemma:density_props}, there exists a constant $C_0>0$ depending only on $K,L$, such that 
$f(\tau_1,\cdots,\tau_K)\ge C_0^{-1}$
for all $(\tau_1,\ldots,\tau_K) \in \bR^{K}$ such that $\max_{i\in[K]}|\tau_i|\le C_0^{-1}$. 
Due to the definition of $\mathcal{A}$, if $h \leq 1/\left( (4m_{\theta} L_4)^{K} \right)$, then
\begin{align*}
    &\Exp\left[(\bv'\bZ)^2\idf\{\bZ\in \mathcal{U}_h^{(K)}\}\right] 
    \geq C_0^{-1} \int_{\mathcal{A} \times [-C_0^{-1}, C_0^{-1}]} \left( \sum_{i=1}^{K} \gamma_{K,i} \tau_i\right)^2 d\tau_1\ldots d\tau_K.
\end{align*}

Denote by $r$ the Lebesgure area of $\mathcal{A}'$, where $\mathcal{A}' := \mathcal{A} \times [-C_0^{-1}, C_0^{-1}]$. Further, let $(\xi_1,\ldots,\xi_K)$ be a random vector uniformly distributed on $\mathcal{A}'$. By definition, $\xi_1,\ldots,\xi_K$ are independent, and then due to \eqref{aux:coeffs},
\begin{align*}
    &\Exp\left[(\bv'\bZ)^2\idf\{\bZ\in \mathcal{U}_h^{(K)}\}\right] 
    \geq r C_0^{-1} \text{Var} \left(\sum_{i=1}^{K} \gamma_{K,i} \xi_i\right) \ge rC_0^{-1} \min_{1 \leq i \leq K} \text{Var}(\xi_i).
\end{align*}
Then the proof is complete by letting 
$\ell_0 := r C_0^{-1} \min\{L_4^2 h^2/12, \;1/(48C_0^2)\}$.
\end{proof}

\section{Some Fundamental Results}
\begin{lemma} \label{lemma: subgaussian lower eigenvalue whp}
Let $\bA, \bd{B}\in \mR^{d\times d}$ be two symmetric matrices, and  $S \subseteq [d]$ with $s:=|S| > 0$.  Assume that there exist  $L,C>0$ such that 
$L^{-1}\le \lmin{\bA}\le\lmax{\bA}\le L$ and $\mn\bd{B}-\bA\mn_{\textup{max}}\le C$.
Then 
$$
L^{-1}-C\times s\le \lmin{\bd{B}_{S,S}} \le \lmax{\bd{B}_{S,S}}\le L+C \times s.
$$
\end{lemma}

\begin{proof}
Let $\bv \in \bR^{s}$ such that $\|\bv\|_2 = 1$. By Cauchy-Schwartz inequality,
\begin{align*}
    \left|\bv'[(\bd{B}-\bA)_{S,S}] \bv\right| \le \mn\bd{B}-\bA\mn_{\textup{max}}\sum_{i\in S}\sum_{j\in S}|\bv_i\bv_j| \le C (s\|\bv\|_2^2) = C \times s.
\end{align*}
Due to the assumption,  $L^{-1}\le \bv' \bA_{S,S} \bv \le L$, which implies that $L^{-1} - C \times s\le \bv' \bd{B}_{S,S} \bv \le L + C \times s$. The proof is then complete.
\end{proof}

\begin{lemma}\label{lemma: concentration inequality of bernoulli}
Let $\{B_i: i\in[n]\}$ be a sequence of independent Bernoulli random variables. Assume that for some $\delta \in [0,1]$, $\Exp[B_i] \leq \delta$ for $i \in [n]$. Then
$$
\Pro\left(\frac{1}{n}\sum_{i=1}^n B_i\ge 2\delta \right)\le \exp\left\{-\frac{3}{8}n\delta\right\}.
$$
\end{lemma}
\begin{proof}
By Bernstein's inequality (Theorem 2.8.4 in \cite{vershynin2018high}),
\begin{align*} 
&\Pro\left(\sum_{i=1}^n \left(B_i-\delta \right)\ge n\delta \right)  \le  \Pro\left(\sum_{i=1}^n (B_i-\Exp[B_i]) \ge n\delta \right)  \leq \exp\left\{-\frac{(n\delta)^2/2}{n\delta+(n\delta)/3} \right\}.
\end{align*}
The proof is complete.
\end{proof}

\begin{lemma} \label{lemma: max inequality of bounded mat}
Let $\{\bA_{\ell} \in\mR^{m\times n}: \ell \in[N]\}$ be a sequence of independent random matrices. Assume that for a constant $L>0$,  $\op{\bA_{\ell}}_{\textup{max}}\le L$ for all $\ell \in[N]$. Then, for any $\epsilon>0$ we have
$$\Pro\left(\op{\sum_{\ell=1}^N (\bA_\ell -\Exp[\bA_{\ell}]) }_{\textup{max}}\ge \epsilon \right) \le 2mn\exp\left\{-\frac{\epsilon^2}{2 L^2 N} \right\}.$$
\end{lemma}
\begin{proof} Fix $i \in [m]$ and $j \in [n]$. 
By the Hoeffding bound (see Proposition 2.5 in \cite{wainwright2019high}), we have
\begin{align*}
    \Pro\left(\left|\sum_{\ell=1}^N ((\bA_{\ell})_{i,j} -\Exp[(\bA_{\ell})_{i,j}]) \right| \ge \epsilon \right) \le 2 \exp\left\{-\frac{\epsilon^2}{2 L^2 N} \right\}.
\end{align*}
The proof is then complete by the union bound.
\end{proof}

\subsection{Restricted Eigenvalues}
Let $\bA \in \mR^{n\times d}$ be a matrix, where $d \geq 2$. Following Definition 1 in \cite{pmlr-v23-rudelson12}, for $1 \leq s \leq d$ and $\kappa > 0$, we define
\begin{equation}\label{eq: K_RE}
    \frac{1}{K(s, \kappa, \bA)} := \min_{J \subseteq [d], |J|\le s}\min_{\|\bv_{J^c}\|_1 \le \kappa\|\bv_J\|_1} \frac{\|\bA v\|_2}{\|v_J\|_2}, \;\;\text{ for all } \bv\in\mR^d\backslash\{\bd{0}_d\}.
\end{equation}
This quantity is introduced to enable the application of results from \cite{pmlr-v23-rudelson12}, and is closely related to the restricted eigenvalue condition in Definition \ref{def: RE}, as demonstrated in the following lemma.

\begin{lemma} \label{lemma: zhou lemma 12}
Let $a>0$ and $\bA\in \mR^{n\times d}$. Further, let  $1 \leq s \leq d$ and $\kappa > 0$.

\begin{enumerate}[label=\roman*)]
\item If 
$1/{K(s, \kappa, \bA/\sqrt{n})} \ge a$,
then $\bA$ satisfies $\textup{RE}(s,\kappa,a^2/(1+\kappa))$.

\item If $\bA$ satisfies $\textup{RE}(s,\kappa, a)$, then $1/{K(s, \kappa, \bA/\sqrt{n})} \ge \sqrt{a}$.
\end{enumerate}
\end{lemma}


\begin{proof}
Recall the definition of $\mathcal{C}(s,k)$ in Definition \ref{def: RE}. We begin with the first claim. Let $\bv \in \mathcal{C}(s,k)$, and denote by $T_0$ the locations of the $s$ largest coefficients of $\bv$ in absolute values. Then by definition, 
$\|\bv_{T_0^c}\|_1 \le \kappa \|\bv_{T_0}\|_1$. Due to \eqref{eq: K_RE}, 
\begin{align*}
   \frac{\|\bA \bv\|_2^2}{n\|\bv_{T_0}\|_2^2} \geq  \frac{1}{[K(s,\kappa,\bA/\sqrt{n})]^2}  \ge a^2. 
\end{align*}
By Lemma 12 of \cite{pmlr-v23-rudelson12}, we have 
$\|\bv\|_2^2 \leq (1+\kappa) \|\bv_{T_0}\|_2^2$, which implies that $n^{-1}\|\bA \bv\|_2^2 \geq (a^2/(1+\kappa)) \|\bv\|_2^2$. Thus, $\bA$ satisfies $\textup{RE}(s,\kappa,a^2/(1+\kappa))$.

Now, we focus on the second claim. Let $\bv \in\mR^d\backslash\{\bd{0}_d\}$, and $J \subseteq [d]$ such that $|J| \leq s$ and $\|\bv_{J^c}\|_1 \le \kappa\|\bv_J\|_1$. By definition, $\bv \in \mathcal{C}(s,k)$. As a result,
\begin{align*}
   a \leq  \frac{\|\bA \bv\|_2^2}{n\|\bv\|_2^2} \leq \frac{\|(\bA/\sqrt{n}) \bv\|_2^2}{\|\bv_J\|_2^2}.
\end{align*}
Since the above holds for any such pair $\bv,J$, we have $1/{K(s, \kappa, \bA/\sqrt{n})} \ge \sqrt{a}$. The proof is complete. 
\end{proof}

Recall the $\psi_2$ norm defined before Assumption \ref{assumption: covariates_bandit}, and $\mC(s,\kappa)$  in Definition \ref{def: RE}.

\begin{lemma}\label{lemma: RE between population and design} 
Let $\bZ = (\bZ_1,\cdots,\bZ_n)'\in \mR^{n\times d}$
be a random matrix. Let $1 \leq s \leq d$, $\kappa, \alpha_Z >0$ and $0 < a< b$. Assume that $\bZ_1,\ldots,\bZ_n$ are i.i.d.~$\bR^d$-random vectors, and that for each vector $\bu \in \bR^{d}$ with $\|\bu\|_2 = 1$, $\|\bu'\bZ_1\|_{\psi_2} \leq \alpha_Z$. Further, let 
$\bA := \Exp[\bZ_1\bZ_1']$, and assume that 
$a\le \lmin{\bA} \le \lmax{\bA} \le b$. 
Then, there exists a constant $C>0$ depending only on $\alpha_Z,\kappa,a,b$, such that for $n\ge Cs\log d$, with probability at least $1-2e^{-n/C}$, 
$$\frac{a}{4}\|\bv\|_2^2\le  \bv'\left(\frac{1}{n}\sum_{i=1}^{n} \bZ_i \bZ_i'\right) \bv\le \frac{9b}{4}\|\bv\|_2^2,\quad \text{for all } \bv\in \mC(s,\kappa).$$
\end{lemma}

\begin{proof}
Since $a\le \lmin{\bA}$, by Definition \ref{def: RE}, 
$n^{1/2} \bA^{1/2}$ satisfies $\textup{RE}(s,3\kappa, a)$ for any $s \geq 1,\kappa>0$. By Lemma \ref{lemma: zhou lemma 12}, $1/K(s,3\kappa,A^{1/2}) \geq \sqrt{a}$, where $K(\cdot)$ is defined in \eqref{eq: K_RE}.

Now, define $\bd{\Psi}_i := \bA^{-1/2}\bZ_i$ for $i\in[n]$, and $\bd{\Psi}=(\bd{\Psi}_1,\cdots,\bd{\Psi}_n)'=\bZ \bA^{-1/2}$. Then, $\bZ = \bd{\Psi}\bA^{1/2} $. First, since  $\Exp[\bd{\Psi}_1 \bd{\Psi}_1']$ is the $d$-by-$d$ identity matrix,  $\bd{\Psi}_1$ is isotropic as defined in Definition 5 of \cite{pmlr-v23-rudelson12}. Second, for any vector $\bu \in \bR^{d}$ with $\|\bu\|_2 = 1$,
\begin{align*}
   \|\bu' \bd{\Psi}_1\|_{\psi_2} =    \|(\bA^{-1/2}\bu)' \bZ_1\|_{\psi_2} \leq \|\bA^{-1/2}\bu\|_{2} \alpha_{Z}.
\end{align*}
Since $a \leq \lmin{\bA}$, we have $\|\bu' \bd{\Psi}_1\|_{\psi_2} \leq a^{-1/2} \alpha_Z$. 

Finally, we apply Theorem 15 in \cite{pmlr-v23-rudelson12} with $\delta=1/2$ and $p = q = d$ to obtain that there exists some constant $C>0$ depending only on $\alpha_Z, \kappa,a,b$ such that if $n \geq Cs\log(d)$, with probability at least $1-2e^{-n/C}$,
\begin{align*}
    \frac{1}{4} \leq \frac{\|\bd{\Psi}\bA^{1/2} \bv\|_2^2}{n \|\bA^{1/2} \bv\|_2^2} \leq \frac{9}{4}.
\end{align*}
Since  
$a\le \lmin{\bA} \le \lmax{\bA} \le b$, we have $a \|\bv\|_2^2\leq \|\bA^{1/2} \bv\|_2^2 \leq b \|\bv\|_2^2$. The proof is then complete since $\bZ = \bd{\Psi}\bA^{1/2}$.
\end{proof}

\section{Additional Simulation Results}
\label{app: more_simulations}
\underline{\textbf{Sequential Estimation.}}
We present additional simulation results for the sequential estimation setting, focusing on scenarios (a), (b), (c), and (d) as described in Section~\ref{sec: HSLR simulation}.
Figure \ref{fig: HSLR plot 1} displays the running cumulative estimation error from time $T/10$ to time $t \in [T/10, T]$, comparing the performance of OPT-Lasso, Lasso, and an oracle estimator that has access to the true support and applies ordinary least squares on it.
Two different combinations of the tuning parameters $C_0$ and $C_0^{\textup{hard}}$ are considered.
Table~\ref{fig: HSLR sensitivity 1} reports a sensitivity analysis showing the cumulative estimation error from time $T/10$ to time $T$ across these parameter combinations.

\medskip

\underline{\textbf{Bandit Setup.}} We present additional simulation results for the bandit setup, focusing on scenarios 
(a), (b), (c), (d), (g) and (h) in Section \ref{sec: bandit simulation}. In Figure \ref{fig: bandit plot 1}, we report the cumulative regret of the three-stage algorithm and the ``Two-stage with Lasso'' algorithm up to time $t$ for $t \in (\gamma_2, T]$.
In Table \ref{table: sensi bandit 1}, we present a sensitivity analysis of the three-stage algorithm to the tuning parameters $C_0$ and $C_0^{\textup{hard}}$. 

\begin{figure}[t!]
\centering
\includegraphics[width=14cm]{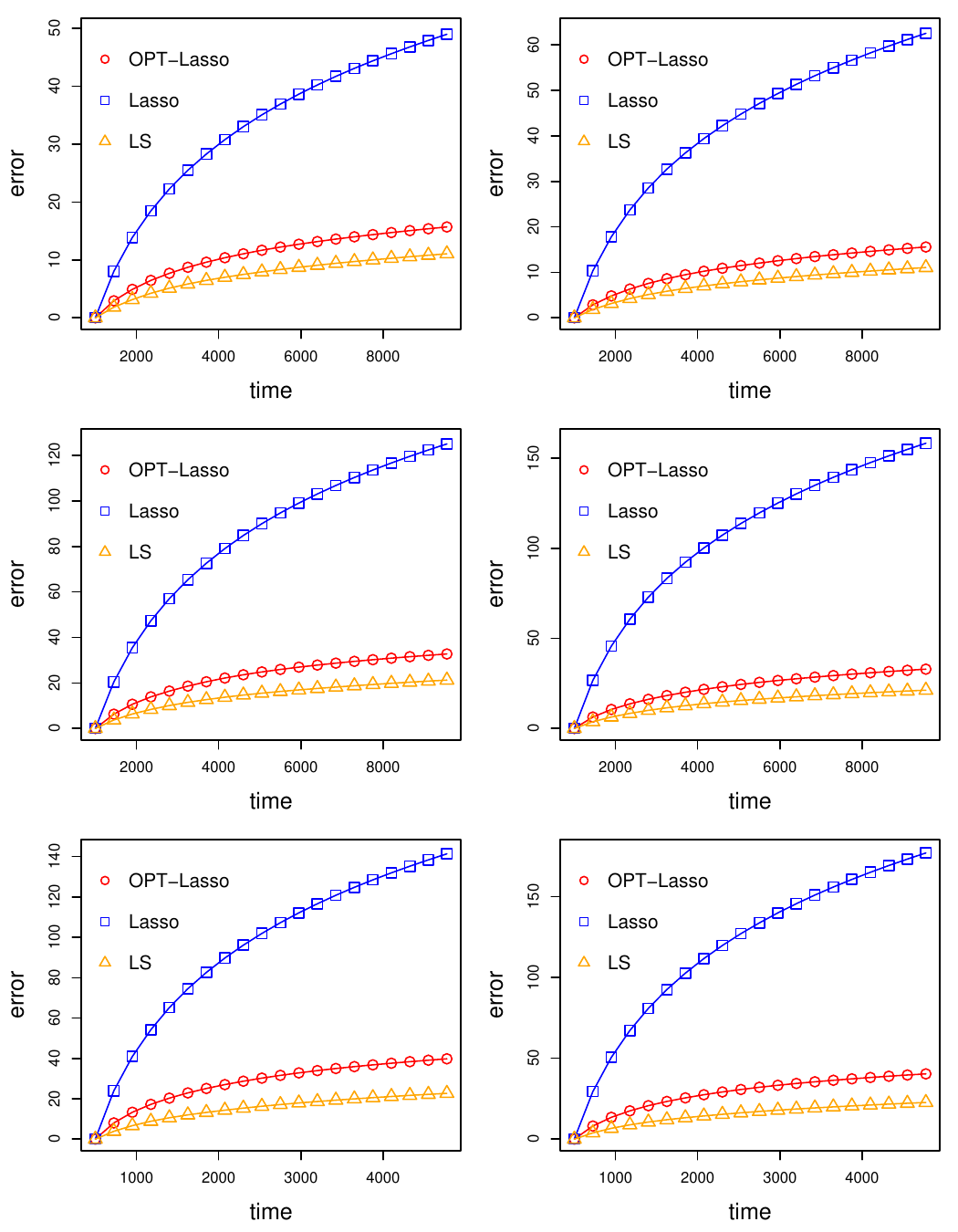}
\caption{The $x$-axis represents time, and the $y$-axis shows the cumulative estimation error from time $T/10$ to  time  $t \in [T/10, T]$.
The three rows correspond to scenarios (a), (b), and (d), respectively, while the first and second columns correspond to the parameter settings $(C_0 = 0.8,\ C_0^{\textup{hard}} = 0.6)$ and $(C_0 = 1,\ C_0^{\textup{hard}} = 0.4)$, respectively.
Each plot displays the running cumulative estimation error of OPT-Lasso, Lasso, and an oracle estimator (denoted as “LS”), which has access to the true support.
}
\label{fig: HSLR plot 1}
\end{figure} 

\begin{table}[t!]
\caption{
Under four scenarios, 
for different combinations of $C_0$ and $C_0^{\textup{hard}}$, we report the cumulative estimation error from time $T/10$ to $T$ for OPT-Lasso and Lasso. The combinations with the smallest cumulative errors are marked in bold.
} \label{fig: HSLR sensitivity 1}
\begin{tabular}{c|ccccc|c}
\hline
(a)  &      &        & \text{OPT-Lasso}      &        &            & \text{Lasso} \\ \hline
$C_0\;\backslash \; C_0^{\textup{hard}}$    & 0.2         & 0.4         & 0.6        & 0.8        & 2.0           & -      \\ \hline
0.4 & 176.8 $\pm$ 7.7 & 135.3 $\pm$ 2.1 & 100.0 $\pm$ 1.1  & 69.7 $\pm$ 1.0   & 16.1 $\pm$ 0.8  & 64.2 $\pm$ 1.4 \\ 
0.6 & 99.0 $\pm$ 1.1    & 56.7 $\pm$ 1.3  & 31.7 $\pm$ 0.6 & 20.1 $\pm$ 0.4 & 23.9 $\pm$ 1.4  & 45.7 $\pm$ 1.1 \\
0.8 & 46.3 $\pm$ 1.5  & 22.7 $\pm$ 0.5  & \textbf{16.0 $\pm$ 0.4}   & \textbf{15.9 $\pm$ 0.7} & 41.0 $\pm$ 2.3    & 52.5 $\pm$ 1.5 \\
1.0   & 23.1 $\pm$ 0.8  & \textbf{15.9 $\pm$ 0.3}  & 17.2 $\pm$ 0.4 & 21.9 $\pm$ 1.1 & 70.0 $\pm$ 3.9    & 64.6 $\pm$ 1.4 \\
1.2 & 16.7 $\pm$ 0.4  & 16.9 $\pm$ 0.4  & 20.4 $\pm$ 1.0   & 26.1 $\pm$ 0.8 & 102.4 $\pm$ 6.1 & 88.5 $\pm$ 1.9 \\   
\end{tabular}

\medskip

\begin{tabular}{c|ccccc|c}
\hline
(b)  &      &        & \text{OPT-Lasso}      &        &            & \text{Lasso} \\ \hline
$C_0\;\backslash \; C_0^{\textup{hard}}$    & 0.2         & 0.4         & 0.6         & 0.8         & 2.0          & -       \\ \hline
0.4 & 761.6 $\pm$ 4.9  & 482.0 $\pm$ 26.4  & 356.3 $\pm$ 9.6 & 224.2 $\pm$ 4.9 & 33.6 $\pm$ 0.9 & 203.9 $\pm$ 2.6 \\
0.6 & 322.9 $\pm$ 5.4 & 156.5 $\pm$ 2.6 & 74.7 $\pm$ 1.2  & 42.1 $\pm$ 1.0    & 53.2 $\pm$ 1.7 & 124.4 $\pm$ 1.4 \\
0.8 & 115.7 $\pm$ 1.6 & 47.4 $\pm$ 1.0    & \textbf{33.3 $\pm$ 0.9}  & 35.3 $\pm$ 0.8  & 94.4 $\pm$ 4.9 & 120.7 $\pm$ 2.2 \\
1.0   & 46.3 $\pm$ 1.0    & \textbf{33.5 $\pm$ 0.7}  & 39.6 $\pm$ 1.0    & 47.1 $\pm$ 1.2  & 171.2 $\pm$ 6.0  & 162.4 $\pm$ 2.1 \\
1.2 & \textbf{34.2 $\pm$ 0.9}  & 37.4 $\pm$ 0.9  & 50.6 $\pm$ 1.7  & 65.8 $\pm$ 1.8  & 272.9 $\pm$ 13.0 & 222.0 $\pm$ 4.0    \\  
\end{tabular}

\medskip

\begin{tabular}{c|ccccc|c}
\hline
(c)  &      &        & \text{OPT-Lasso}      &        &            & \text{Lasso} \\ \hline
$C_0\;\backslash \; C_0^{\textup{hard}}$    & 0.2          & 0.4         & 0.6        & 0.8         & 2.0           & -       \\ \hline
0.4 & 1108.6 $\pm$ 2.6 & 794.9 $\pm$ 2.3 & 515.1 $\pm$ 2.0  & 308.5 $\pm$ 1.7 & 23.8 $\pm$ 1.0    & 258.9 $\pm$ 1.8 \\
0.6 & 468.6 $\pm$ 2.2  & 207.8 $\pm$ 2.1 & 79.6 $\pm$ 0.8 & 33.3 $\pm$ 0.5  & 36.3 $\pm$ 1.8  & 113.8 $\pm$ 1.4 \\
0.8 & 144.3 $\pm$ 1.1  & 40.8 $\pm$ 0.5  & \textbf{21.1 $\pm$ 0.4} & 21.7 $\pm$ 0.4  & 68.3 $\pm$ 2.1  & 79.9 $\pm$ 1.1  \\
1.0   & 38.6 $\pm$ 0.5   & \textbf{20.6 $\pm$ 0.4}  & 24.2 $\pm$ 0.5 & 30.4 $\pm$ 0.6  & 121.3 $\pm$ 6.4 & 94.3 $\pm$ 1.9  \\
1.2 & \textbf{20.4 $\pm$ 0.4}   & 23.8 $\pm$ 0.5  & 31.9 $\pm$ 0.7 & 42.2 $\pm$ 0.9  & 196.1 $\pm$ 5.5 & 118.7 $\pm$ 1.4 \\  
\end{tabular}

\medskip

\begin{tabular}{c|ccccc|c}
\hline
(d)  &      &        & \text{OPT-Lasso}      &        &            & \text{Lasso} \\ \hline
$C_0\;\backslash \; C_0^{\textup{hard}}$    & 0.2          & 0.4         & 0.6         & 0.8         & 2.0           & -      \\ \hline
0.4 & 1122.1 $\pm$ 3.2 & 815.6 $\pm$ 2.8 & 534.3 $\pm$ 2.1 & 316.8 $\pm$ 1.6 & 41.6 $\pm$ 0.8  & 294.1 $\pm$ 1.5 \\
0.6 & 487.8 $\pm$ 1.9  & 225.8 $\pm$ 1.4 & 97.3 $\pm$ 1.3  & 49.9 $\pm$ 0.6  & 73.7 $\pm$ 1.2  & 157.7 $\pm$ 0.9 \\
0.8 & 162.1 $\pm$ 1.3  & 60.3 $\pm$ 1.0    & \textbf{40.5 $\pm$ 0.8}  & 42.9 $\pm$ 0.6  & 136.5 $\pm$ 2.2 & 144.9 $\pm$ 1.0   \\
1.0   & 57.4 $\pm$ 0.6   & \textbf{41.0 $\pm$ 0.6}    & 47.6 $\pm$ 0.7  & 61.4 $\pm$ 1.4  & 239.0 $\pm$ 3.5   & 182.0 $\pm$ 1.2   \\
1.2 & \textbf{40.2 $\pm$ 0.7}   & 48.5 $\pm$ 1.0    & 66.3 $\pm$ 1.1  & 87.0 $\pm$ 1.4    & 403.0 $\pm$ 5.1   & 241.6 $\pm$ 1.4  
\end{tabular}
\end{table}

\begin{figure}[!t]
\centering
\includegraphics[width=\textwidth]{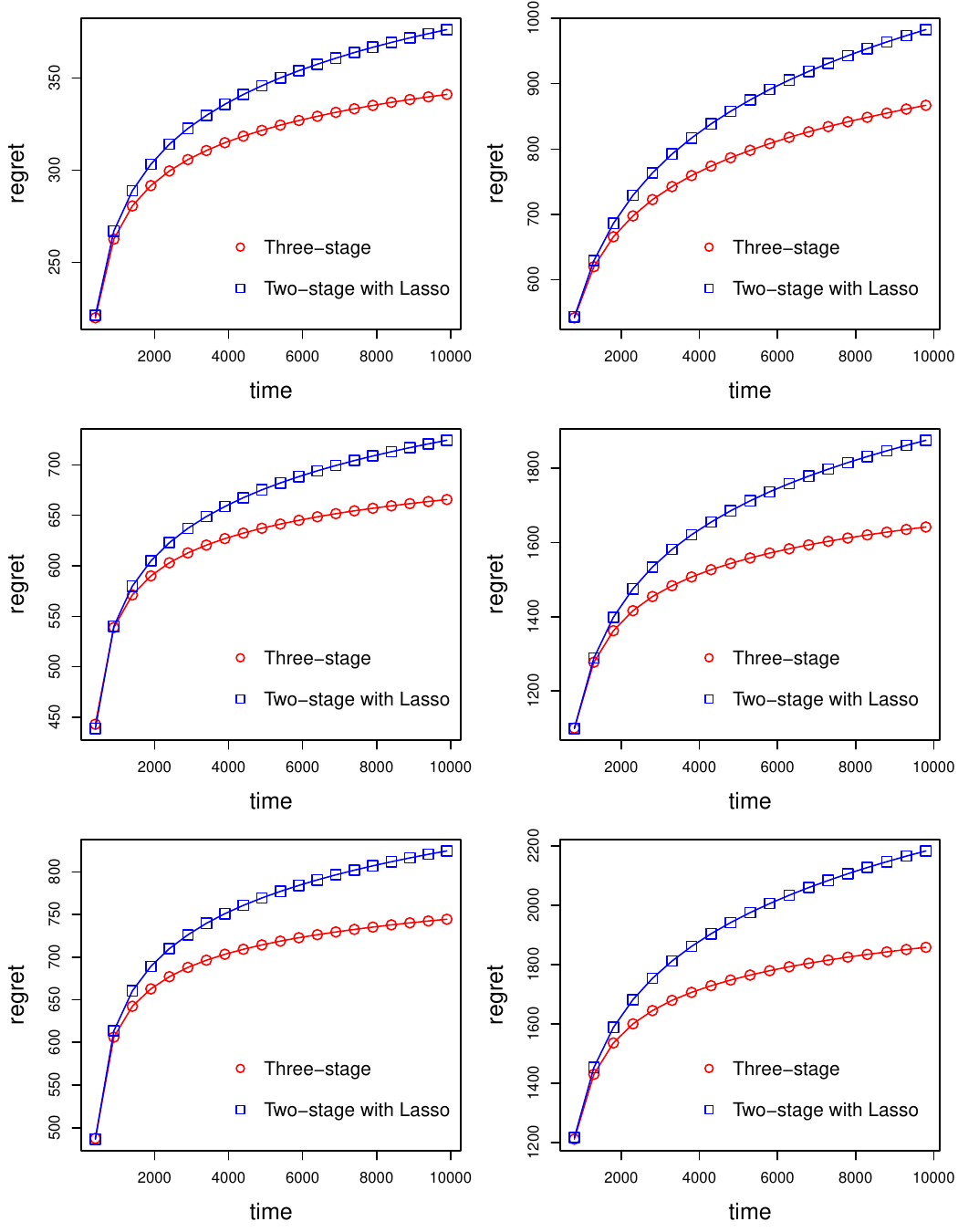}
\caption{The $x$-axis represents time, and the $y$-axis denotes the cumulative regret up to time $t \in [\gamma_2, T]$. The three rows in the left column correspond to scenarios (a), (c), and (g), with $C_0 = 2$ and $C_0^{\textup{hard}} = 0.6$, while the three rows in the right column correspond to scenarios (b), (d), and (h), with $C_0 = 2$ and $C_0^{\textup{hard}} = 1$.} \label{fig: bandit plot 1}
\end{figure}
\begin{table}[!t]
\centering
\caption{The tables show the cumulative regret for scenarios (a), (b), (c), (d), (g) and (h). The parameter combinations with the smallest cumulative regret are marked in bold.
}\label{table: sensi bandit 1}
$(a)$
\begin{tabular}{c|ccccc}
\hline
$C_0\;\backslash \; C_0^{\textup{hard}}$ & 0.2         & 0.6         & 1.0           & 1.5         & 2.0             \\ \hline
1.0   & 826.2 $\pm$ 2.9 & 607.3 $\pm$ 2.3 & 441.0 $\pm$ 1.4   & 355.4 $\pm$ 1.1  & 351.0 $\pm$ 1.1     \\
1.6 & 526.0 $\pm$ 2.1   & 361.4 $\pm$ 1.7 & \textbf{341.0 $\pm$ 1.5}   & 405.6 $\pm$ 2.2  & 503.2 $\pm$ 3.1   \\
2.0   & 416.8 $\pm$ 1.9 & \textbf{341.4 $\pm$ 1.6} & 388.6 $\pm$ 2.0   & 520.9 $\pm$ 3.1  & 721.5 $\pm$ 5.0     \\
2.6 & 356.2 $\pm$ 1.7 & 398.0 $\pm$ 2.0     & 525.4 $\pm$ 3.3 & 813.5 $\pm$ 5.4  & 1258.5 $\pm$ 8.5  \\
3.0   & 366.7 $\pm$ 1.8 & 466.1 $\pm$ 2.7 & 670.5 $\pm$ 4.4 & 1107.5 $\pm$ 7.9 & 1740.3 $\pm$ 10.7 \\ 
\end{tabular}
\medskip

$(b)$
\begin{tabular}{c|ccccc}
\hline
$C_0\;\backslash \; C_0^{\textup{hard}}$ & 0.2          & 0.6          & 1.0            & 1.5          & 2.0             \\ \hline
1.0   & 2166.8 $\pm$ 4.4 & 1879.5 $\pm$ 4.0   & 1546.9 $\pm$ 2.8 & 1211.7 $\pm$ 2.4 & 1014.4 $\pm$ 2.2  \\
1.6 & 1664.5 $\pm$ 4.0   & 1237.2 $\pm$ 3.7 & 967.2 $\pm$ 3.3  & 872.5 $\pm$ 3.1  & 920.5 $\pm$ 3.5   \\
2.0   & 1401.0 $\pm$ 3.7   & 1000.3 $\pm$ 3.3 & \textbf{869.1 $\pm$ 2.8}  & 916.1 $\pm$ 3.6  & 1082.7 $\pm$ 4.5  \\
2.6 & 1112.5 $\pm$ 3.4 & 877.6 $\pm$ 3.1  & 918.3 $\pm$ 3.6  & 1131.9 $\pm$ 4.9 & 1510.1 $\pm$ 7.9  \\
3.0   & 1009.3 $\pm$ 3.5 & 891.7 $\pm$ 3.4  & 1018.4 $\pm$ 4.6 & 1365.0 $\pm$ 6.6   & 1929.6 $\pm$ 10.5 \\  
\end{tabular}
\medskip

$(c)$
\begin{tabular}{c|ccccc}
\hline
$C_0\;\backslash \; C_0^{\textup{hard}}$ & 0.2         & 0.6          & 1.0            & 1.5          & 2.0             \\ \hline
1.0   & 1668.7 $\pm$ 3.7 & 1155.4 $\pm$ 3.1 & 793.2 $\pm$ 1.4  & 664.0 $\pm$ 1.2    & 688.4 $\pm$ 1.3   \\
1.6 & 1007.7 $\pm$ 2.7 & 674.1 $\pm$ 2.1  & 673.4 $\pm$ 2.1  & 797.8 $\pm$ 2.7  & 1017.9 $\pm$ 4.0    \\
2.0   & 781.8 $\pm$ 2.5  & \textbf{666.0 $\pm$ 2.2}    & 776.2 $\pm$ 2.7  & 1028.8 $\pm$ 4.2 & 1418.2 $\pm$ 6.2  \\
2.6 & 682.1 $\pm$ 2.4  & 774.3 $\pm$ 2.7  & 1042.1 $\pm$ 4.0   & 1576.0 $\pm$ 7.1   & 2360.6 $\pm$ 10.4 \\
3.0   & 702.3 $\pm$ 2.4  & 905.3 $\pm$ 3.3  & 1309.9 $\pm$ 5.6 & 2096.3 $\pm$ 9.2 & 3161.9 $\pm$ 12.3 \\
 \end{tabular}
 \medskip

$(d)$
\begin{tabular}{c|ccccc}
\hline
$C_0\;\backslash \; C_0^{\textup{hard}}$ & 0.2          & 0.6          & 1.0            & 1.5          & 2.0             \\ \hline
1.0   & 4263.2 $\pm$ 6.4 & 3614.1 $\pm$ 6.3 & 2832.8 $\pm$ 3.0   & 2108.4 $\pm$ 2.6 & 1779.4 $\pm$ 2.3  \\
1.6 & 3177.1 $\pm$ 5.6 & 2242.0 $\pm$ 4.5   & 1741.6 $\pm$ 4.1 & 1674.8 $\pm$ 4.1 & 1852.6 $\pm$ 4.9  \\
2.0   & 2623.9 $\pm$ 5.1 & 1810.8 $\pm$ 4.1 & \textbf{1644.5 $\pm$ 4.2} & 1827.6 $\pm$ 4.6 & 2218.2 $\pm$ 7.0    \\
2.6 & 2070.1 $\pm$ 4.9 & \textbf{1652.3 $\pm$ 3.9} & 1813.7 $\pm$ 5.1 & 2309.7 $\pm$ 7.1 & 3084.4 $\pm$ 10.8 \\
3.0   & 1861.3 $\pm$ 4.4 & 1706.8 $\pm$ 4.4 & 2039.1 $\pm$ 5.9 & 2790.1 $\pm$ 9.1 & 3869.8 $\pm$ 13.3 \\ 
\end{tabular}\\

\medskip

$(g)$
\begin{tabular}{c|ccccc}
\hline
$C_0\;\backslash \; C_0^{\textup{hard}}$  & 0.2              & 0.6              & 1.0                & 1.5               & 2.0                 \\ \hline
1.0   & 2129.6 $\pm$ 4.7 & 1404.3 $\pm$ 3.4 & 905.1 $\pm$ 2.8  & \textbf{745.2 $\pm$ 2.2}   & 779.9 $\pm$ 2.4   \\
1.6 & 1239.8 $\pm$ 3.3 & 759.1 $\pm$ 2.5  & 755.6 $\pm$ 2.5  & 904.0 $\pm$ 3.2     & 1154.0 $\pm$ 4.4    \\
2.0   & 923.6 $\pm$ 2.7  & \textbf{744.8 $\pm$ 2.5}  & 863.4 $\pm$ 2.7  & 1157.1 $\pm$ 4.7  & 1609.4 $\pm$ 6.7  \\
2.6 & 765.5 $\pm$ 2.4  & 856.2 $\pm$ 2.9  & 1164.2 $\pm$ 4.5 & 1763.5 $\pm$ 7.4  & 2636.6 $\pm$ 10.6 \\
3.0   & 778.8 $\pm$ 2.5  & 999.2 $\pm$ 3.8  & 1460.9 $\pm$ 6.3 & 2336.2 $\pm$ 10.1 & 3456.2 $\pm$ 13.8 \\ 
\end{tabular}

\medskip

$(h)$
\begin{tabular}{c|ccccc}
\hline
$C_0\;\backslash \; C_0^{\textup{hard}}$ & 0.2          & 0.6          & 1.0            & 1.5          & 2.0             \\ \hline
1.0   & 4935.7 $\pm$ 7.1 & 4227.7 $\pm$ 6.8 & 3246.3 $\pm$ 5.9 & 2336.2 $\pm$ 4.9 & 1986.5 $\pm$ 4.8  \\
1.6 & 3817.5 $\pm$ 6.3 & 2597.0 $\pm$ 5.6   & 1963.8 $\pm$ 4.3 & 1906.1 $\pm$ 4.6 & 2130.3 $\pm$ 5.8  \\
2.0   & 3138.4 $\pm$ 5.8 & 2069.7 $\pm$ 4.9 & \textbf{1861.3 $\pm$ 4.7} & 2104.8 $\pm$ 5.8 & 2557.4 $\pm$ 7.8  \\
2.6 & 2422.7 $\pm$ 5.3 & \textbf{1854.3 $\pm$ 4.8} & 2069.9 $\pm$ 5.4 & 2658.4 $\pm$ 8.4 & 3575.2 $\pm$ 12.4 \\
3.0   & 2149.2 $\pm$ 4.9 & 1909.7 $\pm$ 4.8 & 2297.2 $\pm$ 7.0   & 3197.0 $\pm$ 10.9  & 4505.6 $\pm$ 16.2 
\end{tabular}
\end{table}
\end{appendix}

\end{document}